\documentclass[12pt,a4paper,twoside]{book}
\usepackage[T1]{fontenc}
\usepackage[Conny]{fncychap}
\usepackage{times}
\usepackage{lmodern}
\usepackage[a4paper]{geometry}
\usepackage{soul}
\usepackage{fancyhdr}
\usepackage{amscd,amsfonts,amsmath,amssymb,latexsym,mathrsfs,pifont,xcolor}
\usepackage{amsmath, amssymb, amsfonts}
\usepackage{graphicx,color}
\usepackage[english]{babel}
\usepackage{color}
\usepackage{tikz}
\usetikzlibrary{plotmarks}
\usepackage{pdfrender}
\usepackage{xcolor}
\usepackage{multirow}
\usepackage{graphicx}
\usepackage{amsfonts,amsthm,amssymb,epsf,amsmath}
\usepackage{bm}
\usepackage{xcolor}
\definecolor{DarkRed}{RGB}{139,0,0}
\usepackage{float}
\usepackage[pdftex,colorlinks=true,linkcolor=blue,citecolor=magenta,urlcolor=red]{hyperref}
\usepackage{graphicx}
\usepackage{mathrsfs}
\usepackage{fancyhdr}
\usepackage{booktabs}

\newtheorem{Theorem}{Theorem}[section]
\newtheorem{Definition}[Theorem]{Definition}
\newtheorem{Proposition}[Theorem]{Proposition}
\newtheorem{Corollary}[Theorem]{Corollary}
\newtheorem{Remark}[Theorem]{Remark}
\begin{document}
%\maketitle
%\maketitle
\thispagestyle{empty}
\begin{titlepage}

%\begin{figure}[h]
%\begin{minipage}{0.5\linewidth}
%%\includegraphics[angle=0,scale=0.6]{p1}
%\includegraphics[angle=0,scale=0.3]{fac}
%\end{minipage}
%\hspace{5cm}
%\begin{minipage}{0.5\linewidth}
%%\includegraphics[angle=0,scale=0.68]{p2}
%\includegraphics[angle=0,scale=0.2]{logo-lamsin}
%\end{minipage}
%\end{figure}

\begin{figure}
\vspace*{-2cm}
\begin{center}
\includegraphics[width=3.5cm,height=4cm]{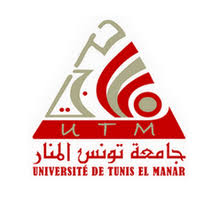}
\end{center}
\includegraphics[width=2.5cm,height=2.5cm]{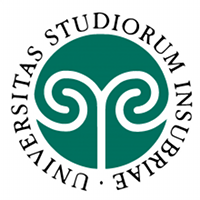}
%% fst.jpg: 0x0 pixel, 0dpi, 0.00x0.00 cm, bb=
%%[width=2.5cm,height=2.5cm]
\hfill
\includegraphics[width=3.5cm,height=1.85cm]{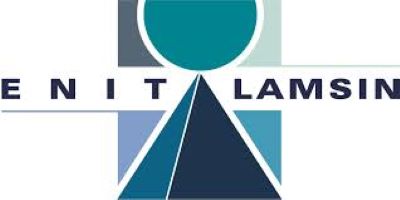}
%% lamsin.jpg: 0x0 pixel, 0dpi, nanxnan cm, bb=
%%[width=3.5cm,height=1.85cm]
%%[scale=0.4]
\end{figure}
\begin{center}
\textsc{University of Tunis El Manar} \\
\textsc{National engineering school of Tunis} \\
\textsc{Faculty of Sciences of Tunis}\\
\vspace*{\stretch{1}}%HTMLD
\normalsize
{\Large \hrule \ \\ \bf
\textsf{ \textcolor{blue}{Spectral analysis and fast methods for large matrices arising from PDE approximation}}\\ [3mm] \hrule} \ \\ \ \\
\emph{\textbf{ }}\par
Thesis submitted to obtain the grade of \par
\emph{\textbf{Doctor of Philosophy}}\par
% Préparée au\\
\small Laboratory of Mathematical and Numerical Modelling in Engineering Sciences \par
\emph{\textbf{Applied Mathematics}}\par

\vspace*{\stretch{0.5}}%HTMLD
 \par
 {\bf \textit{Ryma Imene RAHLA}}
\\ \ \\
Under the supervision of\\
{\bf \textit{Skander BELHAJ and Stefano SERRA-CAPIZZANO}}
\\ \ \\
%&Soutenu le ......... devant la commission du jury :\\
%\begin{center}
%\begin{tabular}{ll}
%{\bf { ...........}}\hspace{2cm} & \bf{Président} \\
%{\bf { Maher \ MOAKHER}} \hspace{2cm} &  \bf{Rapporteur}\\
%{\bf { Skander \ BELHAJ}} \hspace{2cm} &  \bf{Encadrant}
%\end{tabular}
%\end{center}
\end{center}
\end{titlepage}
\newpage
\thispagestyle{empty}
\vspace*{\fill}
\begin{center}

{\huge "\textit{\textsc{Mathematics is the language in which GOD has written the universe}}"}
\end{center}
\begin{flushright}
  {\large Galileo Galilei}
\end{flushright}
\vspace*{\fill}

\newpage
\thispagestyle{empty}
\vspace*{\fill}
I dedicate my Ph.D. thesis to my wonderful supervisor Stefano SERRA-CAPIZZANO
\vspace*{\fill}

\newpage
\chapter*{Acknowledgements}

First and foremost, I would like to thank Almighty \textbf{GOD} for providing me the courage and patience to accomplish this work.\\

I would like to thank my thesis supervisor \textbf{Pr. Skander BELHAJ}  for guiding me and for his patience, availability, and invaluable assistance during my thesis years.\\

I am grateful to my co-supervisor \textbf{Pr. Stefano SERRA-CAPIZZANO} for sharing his knowledge and experience with me. All my gratitude for your trust and prompt assistance. Finally, during this project, I was acutely aware of his human qualities of listening and comprehension.\\

I would like to express my heartfelt appreciation to \textbf{Pr. Amel BEN ABDA} for the privilege of chairing the jury. \\

I would like to express my gratitude to \textbf{Pr. Carla MANNI } and \textbf{Pr. Paris VASSALOS} for agreeing to be the reviwers of my thesis work and for their insightful comments that helped me improve the quality of this manuscript. I would like to thank them for granting me this honour by serving on this jury.\\

It is also with pleasure that I thank \textbf{Pr. Maher MOAKHER} for the honour of reviewing my work as a defense examiner.\\

My heartfelt gratitude goes to my parents, who have consistently supported me during these hard years of study. Thank you so much for your words of encouragement, love, and spiritual support. GOD protect you and provide you luck, health, and a long life. A big thank you also to my brother, sisters, grandmother, uncle and aunts.\\ %especially to my aunt Tat for her support and help.\\

My thoughts also go to my friends. I sincerely appreciate them for the memories we enjoyed together, even if our reunions and meetings are becoming increasingly distant.\\

Many thanks to all my collaborators, without whom this work would have not been possible.\\

Finally, I would like to thank  everyone who helped me to develop this study, whether directly or indirectly.

\begin{center}
 \includegraphics[width=5cm]{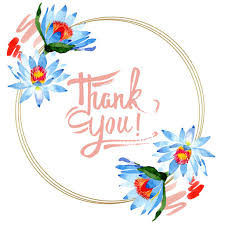}
\end{center}

%for giving encouragement and support in times of need and advice she has provided through my times as his student

\listoftables
\listoffigures
\tableofcontents
\chapter*{Introduction}
\addcontentsline{toc}{chapter}{Introduction}
%Mathematical modelling plays a crucial role in many areas: chemistry, economy, physics, eng la modelisation mathematique consiste a modeliser les problemes qui apparaissent dans le monde reel enfrequmment les edp se resultent dans la modelisation des systemes complexe par exp ........
%Mathematical modelling il consitste a presenter mathematiqument certain phenomene qui existe dans le monde reel. dont l'analyse theorique ou numerique de cette presentation nous donne des  reponses, orientations et le resoudre pour mieux controler
%Mathematical modeling appear in many areas it consists of presenting mathematically Partial diffrential equentions

Mathematical modelling occurs in several areas: chemistry, economy, physics, engineering and many other fields. It consists in presenting mathematically a real phenomenon such that the theoretical or numerical analysis of this mathematical presentation gives answers, insights, and direction relevant to the original application.

When modelling complex systems, such as elastic deformation, wave propagation, heat transport, etc., Partial Differential Equations (PDEs) are frequently encountered, and in general, the analytical solution of these PDEs seldom has a closed form expression. Thus, it is critical to use appropriate numerical approaches to estimate the solution $\mathbf{x}$ of the PDE.

  The principle behind numerical approximations is to convert the PDE, from its continuous form, to a discrete version, by introducing a mesh depending on a parameter $n$, and gradually refined when $n$ increases. Thus, the initial infinite-dimensional PDE problem is reduced to a finite-dimensional problem whose solution $\mathbf{x}_n$ converges to $\mathbf{x}$, when $n$ tends to $\infty$.
Thus, by considering a linear approximation of the linear PDE of the form
$$\mathcal{A}\mathbf{x}=\mathbf{b},$$ (where $\mathcal{A}$ indicates the linear differential operator), and by refining progressively the mesh, we obtain a sequence of linear systems of the form $$\{A_n\mathbf{x}_n=\mathbf{b}_n\}_n.$$
\newpage
These linear systems are in general of large dimensions, as the dimension of the problem depends on the number of discretisation points, and therefore, increasing the size of the system is essential to attain more precision. Thus, the study of appropriate techniques frequently focuses on analysing their behaviour as the system size grows. As a result, it is essential to analyse not only the single system $A_n\mathbf{x}_n=\mathbf{b}_n$ but also the entire sequence of linear systems $\{A_n\mathbf{x}_n=\mathbf{b}_n\}_n$ that depends on $n$.

%When approximating PDEs with constant coefficient by local methods such that Finite element, finite differences, etc., the resulting matrices are Toeplitz

The matrices arising from the approximation of PDEs with constant coefficient when using local methods such that Finite Elements, Finite Differences, etc., are frequently of type Toeplitz in the monodimensional case and multilevel block Toeplitz in multidimensional case. Nevertheless, the approximation of PDEs with a non-constant coefficient, gives rise to matrix-sequences belonging to the Generalized Locally Toeplitz (GLT) class.

For the resolution of a large ill-conditioned system, direct methods are usually unstable \cite{bunch1985stability} and require high computational time. Moreover, these methods do not take fully into consideration the structure or sparsity of the matrix in order to {minimize the number of algebraic operations or optimize} the storage space.

On the other hand, iterative methods such as multigrid and preconditioned Krylov methods are much more favourable: they are less sensitive to numerical instability caused by the ill-conditioning and more adapted to problems of specific structure since they use the spectral information of the given matrix.

The spectrum of the matrix has a major effect on the convergence of the related system. Regarding the structured matrix-sequences, the asymptotic distribution of the spectrum of the discretized problem is described by a function called "symbol", that is, for all continuous functions with bounded support $F\in C_c(\mathbb C)$, the spectrum of $A_n$ has a spectral distribution as the symbol $\mathbf{f}:D\subset\mathbb R^t\to\mathbb{C}^{r\times r}$ in the sense of the eigenvalues:

    \begin{equation}\label{rela1}
       \lim_{n\to\infty}\frac1{d_n}\sum_{i=1}^{d_n}F(\lambda_i(A_n))=\frac1{\mu_t(D)}\int_D\frac{\sum_{i=1}^{r}F(\lambda_i(\mathbf{f}(\mathbf x)))}{r}{\rm d}\mathbf x,
   \end{equation}
 where $D$ is a measurable subset with $0<\mu_t(D)<\infty$, and $\lambda_i(A_n),\lambda_j(\mathbf{f}(\mathbf x))$, $i=1,\ldots,d_n$, $j=1,\ldots,s$, represent the eigenvalues of $A_n$ and the eigenvalues of $\mathbf{f}$, respectively. The interpretation of the relation (\ref{rela1}) is as follows: as $d_n\rightarrow\infty$, the spectrum of $A_n$ is approximated by evaluating of the eigenvalues of the symbol $\lambda_i(\mathbf{f}(\mathbf x))$ over an equispaced grid on $D$.

 The analysis of the symbol gives a precise information regarding the asymptotic behaviour of the spectrum of $A_n$. For instance when the matrix $A_n$ is Hermitian, the symbol often provides an effective estimation of the spectral condition number $A_n$ for sufficiently large $n$.

 The objective of this thesis is to illustrate some concrete examples in order to show the role of the symbol in analysing the spectrum of the matrix and in performing an efficient method to solve the problem arising from PDE approximation. We deal with a class of elliptic partial differential equations with Dirichlet boundary conditions, where the operator is  $\mathrm{div} \left(-a(\mathbf{x}) \nabla \cdot\right)$, see (\ref{eq:modello}), (\ref{eq:modello1}). In Chapter 3, we will investigate the features of the symbol, extracted from the matrices arising in the classical $\mathbb{P}_k$ Finite Elements Approximation, while in Chapter 4, we will develop a multigrid method for the matrices resulting in the case of the $\mathbb{Q}_k$ Finite Elements Approximation.

 In the first part \cite{P_k}, we discuss the spectral analysis of matrix sequences resulting from the $\mathbb{P}_k$ Lagrangian Finite Element approximation of the elliptic problem (\ref{eq:modello}) where the operator is $\mathrm{div} \left(-a(\mathbf{x}) \nabla \cdot\right)$,  and $a$ being continuous and positive on $\overline\Omega$. \\
 Our theoretical analysis focuses on the stiffness
matrix-sequences $\{A_n\}_n$ related to $\mathbb{P}_k$ Finite
Element approximations on uniform structured meshes
\cite{ciarlet,brezzi,Q,p-hp-book}, such as Friedrichs-Keller
triangulations, in which context the powerful spectral tools derived from the Toeplitz theory
\cite{BoSi,glt-1,glt-book-I,glt-book-II,glt-2} greatly facilitate
the required spectral analysis.
\par
We provide a detailed study in the case where $a(\mathbf{x}) \equiv 1$. Then we sketch the general setting by considering a
Riemann integrable diffusion coefficient $a$ and/or a domain $\Omega$ not necessarily of Cartesian structure. We recall that
the same type of analysis of the linear Finite Elements in two dimensions is already considered in  \cite{BeSe,morozov} for the
same equation considered in our study, while coupled partial differential equations (PDEs) with stable pairs of Finite Element
Approximations in two dimensions are considered in \cite{schur-maya}. It is worth noticing the systematic work in
\cite{Q_k}, where the case of tensor rectangular Finite Element Approximations $\mathbb{Q}_{\mathbf{k}}$ of any degree $\mathbf{k}=(k_1,\ldots,k_d)$ and of any dimensionality $d\ge 1$ is studied.
\par
Following the pattern indicated in   \cite{Q_k}, we start a
systematic approach for the  Finite Element Approximations
$\mathbb{P}_k$ for $k\ge2$ and for $d=2$. The analysis for $d=1$
is contained in \cite{Q_k}, trivially because $\mathbb{Q}_k\equiv
\mathbb{P}_k$ for every $k\ge 1$, while for $d=2$, and even more
for $d\ge 3$, the situation is greatly complicated by the fact
that we do not encounter a tensor structure. Nevertheless, the
spectral picture is quite similar and the obtained information in
terms of spectral symbol is sufficient for deducing a quite
accurate analysis concerning the distribution and the extremal
behaviour of the eigenvalues of the resulting matrix-sequences.
\par
More in details, regarding the resulting stiffness matrices, we will consider the following items,
from the perspective of (block) multilevel Toeplitz operators \cite{BoSi,Tilli} and
(block) Generalized Locally Toeplitz (GLT) sequences \cite{glt-1,glt-2}:
\begin{itemize}
\item spectral distribution in the Weyl sense,
\item spectral clustering,
\end{itemize}
with a concise analysis of the extremal eigenvalues, conditioning, spectral localization, and
where the ultimate objective is:
\begin{itemize}
\item the analysis and the design of fast iterative solvers for the associated linear systems.
\end{itemize}
\par
We recall that the spectral distribution and the clustering
results are essential components in the design and in the
convergence analysis of specialised multigrid methods  and
preconditioned Krylov solvers \cite{saad}, such as preconditioned
conjugate gradient (PCG); see \cite[Subsection 3.7]{glt-2} and
\cite{ADS,BK,cmame2,cmame1,our-sinum,Fiorentino-Serra,opiccolo} for more information. In fact, the knowledge of the spectral distribution is the key for
explaining the superlinear convergence history of PCG, thus
improving the classical bounds; see \cite{BK} and references
therein. Most of our study is devoted to the
identification of the spectral symbol using the GLT technology, and
hence to the first item.

Numerous studies have investigated the spectral analysis of the matrices arising from PDE approximations, but with different approximation techniques. We mention, for instance, the case of the stiffness/colloca\-tion matrices coming from the  $\mathbf{k}$\,-degree B-spline Isogeometric Analysis (IgA) approximation of maximal smoothness \cite{IgA-book} of \eqref{eq:modello}; see \cite{our-mathcomp,CCFSH}. The same kind of analysis in the case of $\mathbb{Q}_{\mathbf{k}}$ Finite Elements approximating again  \eqref{eq:modello} is discussed in \cite{Q_k}. A review comparing the latter two approaches with a language tailored for engineers can be found in \cite{tom-paper}.
\par
In the IgA case, the (spectral) symbol  $f_{\rm IgA_{\mathbf{k}}}$
describing the spectral distribution is a scalar-valued
$d$-variate function defined over $[-\pi,\pi]^d$, and hence the
eigenvalues of the IgA discretization matrices are approximated by
a uniform sampling of  $f_{\rm IgA_{\mathbf{k}}}$ over
$[-\pi,\pi]^d$. In this context, the surprising behaviour is that,
when  all the spline degrees $k$ increase, $f_{\rm
IgA_{\mathbf{k}}}(\boldsymbol{\theta})$ collapses exponentially to
zero at all points
$\boldsymbol{\theta}=(\theta_1,\ldots,\theta_d)$ with some
component $\theta_j=\pi$. According to the interpretation based on
the theory of Toeplitz matrices and matrix algebras, this
phenomenon implies that the IgA matrices are ill-conditioned, not
only in the low frequencies (as expected), but also in the high
frequencies, as in the approximation of integral operators
\cite{hanke}. Through the explicit use of this spectral information, it was possible to develop ad hoc iterative solvers with an optimal
convergence rate, substantially independent of
 $\mathbf{k}$ and $d$; see
\cite{cmame2,cmame1,our-sinum}.
\par
In the  $\mathbb{Q}_{\mathbf{k}}$ Lagrangian setting, we are still able to identify the spectral distribution, as for the IgA case.
The related symbol $\mathbf{f}_{\mathbb{Q}_{\mathbf{k}}}$ is $d$-variate and defined on $[-\pi,\pi]^d$, but the surprise is that
$\mathbf{f}_{\mathbb{Q}_{\mathbf{k}}}$ is a $N(\mathbf{k})\times N(\mathbf{k})$ Hermitian matrix-valued function, with
$\mathbf{k}=(k_1,\ldots,k_d)$ vector of the partial degrees in the different directions, $N(\mathbf{k})=\prod_{j=1}^d k_j$ (a similar situation is encountered when dealing with discontinuous Galerkin methods \cite{krause,molteni}). No specific pathologies regarding $f_{\mathbb{Q}_{\mathbf{k}}}$ are observed for large ${\mathbf{k}}$ at the points
$\boldsymbol{\theta}$ such that $\theta_j=\pi$ for some $j$, implying that the source of ill-conditioning, with respect to the
fineness parameters, is only in the low frequencies. However, exactly as in our $\mathbb{P}_k$ setting, where the scalar $k$ represents the global polynomial degree, we observe a serious dimensionality problem, since for moderate $k$
and $d$, the quantity $N(k,d)=k^d$ is very large.

More specifically, the problem is that the spectrum of the $\mathbb{Q}_{\mathbf{k}}$ ($\mathbf{k}=(k_1,\ldots,k_d)$, $d\geq1$ ) Lagrangian Finite Element stiffness matrices is split into $N({\mathbf{k}})=\prod_{j=1}^d k_j$ subsets, or branches in the engineering terminology \cite{crbh06,her14,hrs08,r06}, of approximately the same cardinality and the $i$-th branch is approximately a uniform sampling of the scalar-valued function $\lambda_i(\mathbf{f}_{\mathbb{Q}_{\mathbf{k}}})$, $i=1,\ldots, N(\mathbf{k})$. The exponential scattering (in $\mathbf{k}$ and $d$) of the eigenvalue functions $\lambda_i(\mathbf{f}_{\mathbb{Q}_{\mathbf{k}}})$ provides an explanation of the difficulties encountered by the solvers in the literature, already for moderate $\mathbf{k}$ and $d$. Indeed, it is relatively easy to design a mesh-independent solver, but the dependency on $\mathbf{k}$ and $d$ is generally bad.
\par
 At this point, it is worthwhile stressing that  a cardinality of the branches equal to  $N({\mathbf{k}})$ is expected in the  tensor setting $\mathbb{Q}_{\mathbf{k}}$, while it is somehow a surprise with the current choice of $\mathbb{P}_k$  Finite Elements,
 where $k$ is a scalar denoting the global degree  and $N(k,d)=k^d$. Indeed, we have verified this formula only for $d=2$ and $k=2,3,4$, and hence, a deeper analysis of this phenomenon will be the subject of future investigations.

In the second part of the thesis \cite{multigridQ_k}, a multigrid method is developed for the resolution of the problem coming from $\mathbb{Q}_k$ Finite Element Approximation of the same elliptic problem considered previously, see (\ref{eq:modello1}). By defining the prolongation operator using the inclusion property between the coarser and the finer grid, we provide a study of the relevant analytical features of  all the involved spectral symbols, both of the stiffness matrices $\mathbf{f}_{_{\mathbb{Q}_k}}$ and of the projection operators $\mathbf{p}_{_{\mathbb{Q}_k}}$, $k=1,2,3$. While the two-grid and V-cycle show optimal or quasi-optimal convergence rate,  with respect to all the relevant parameters (size, dimensionality, polynomial degree $k$, and diffusion coefficient), the theoretical prescriptions are only partly satisfied. In fact, with reference to the brief study in Section \ref{optimal2grid}, our choices are in agreement with the mathematical conditions set in {i}tems {\bf (A)} and {\bf (B)}, while {Condition} {\bf (C)} is violated. {Here,  by quasi-optimal convergence rate, we mean that the convergence speed does not depend on the size (optimality with respect to this parameter) and it mildly depends on the other relevant parameters such as dimensionality, polynomial degree $k$, and diffusion coefficient.}
By looking at the mathematical derivations in \cite{multi-block}, we observe that the latter condition indeed is a technical one.
In reality, we believe that Condition {\bf (C)} is not essential and the commutation request can be substituted by a less restrictive one, possibly following the considerations in Remark \ref{rmk:singular_commutator}.

Now, we provide a concise summary of the Thesis content.
\begin{itemize}
\item In \textbf{Chapter 1:} we provide the notations, the necessary definitions, and some known results that lay the groundwork for the succeeding chapters. The final section of this chapter reviews the principles of iterative methods, with specific emphasis on preconditioning and multigrid methods.
\item In \textbf{Chapter 2:} we introduce the fundamental tools for  spectral analysis of matrix-sequences, by discussing distribution in terms of eigenvalues and singular values, clustering, essential range, spectral attraction, and by offering helpful characterizations of zero-distributed sequences. Then, we will explore Toeplitz and Circulant matrices and their spectral properties in both the scalar and block settings. Finally, we will investigate Generalized Locally Toeplitz sequences.
\item In \textbf{Chapter 3:} we consider a class of elliptic partial differential equations with Dirichlet boundary conditions and where the operator is  $\mathrm{div} \left(-a(\mathbf{x}) \nabla \cdot\right)$, with $a$ continuous and po\-si\-ti\-ve  over $\overline \Omega$, $\Omega$ being an open and bounded subset of
$\mathbb{R}^d$, $d\ge 1$. For the numerical approximation, we
 consider the classical $\mathbb{P}_k$ Finite Elements, in the case of Friedrichs-Keller triangulations, which results a sequences of matrices of increasing size. The new results concern the spectral analysis of the resulting matrix-sequences in the direction of the global distribution in the Weyl sense, with a concise overview on localization, clustering, extremal eigenvalues, and asymptotic conditioning. We examine in detail the case of constant coefficients on $\Omega=(0,1)^2$ and provide a quick overview of the more complicated case of variable coefficients and more
 general domains. Tools are drawn from the Toeplitz technology and
 from the rather new theory of GLT
 sequences. Numerical results are included to demonstrate the theoretical findings.
\item In \textbf{Chapter 4:} Our focus is on multigrid technique for solving linear systems deriving from $\mathbb{Q}_k$ Finite Element approximations of elliptic PDEs with Dirichlet boundary conditions, where the operator is $\mathrm{div} \left(-a(\mathbf{x}) \nabla \cdot\right)$, with $a$ continuous and positive over $\overline{\Omega}=[0,1]^{2}$. While the analysis is performed in one dimension, the numerics are carried out also in higher dimension, showing an optimal behaviour in terms of the dependency on the matrix size and a substantial robustness with respect to the dimensionality \texorpdfstring{$d$}{d} and to the polynomial degree \texorpdfstring{$k$}{k}.

A final chapter with conclusions, open problems, and perspectives ends the current PhD Thesis.

\end{itemize}

%Due to the importance that provides the spectral analysis of the matrix in designing a fast method for the resolution of the system, the knowledge %The spectral distribution  in this way
%It is well known that there is direct link%The spectral characteristics of the matrix influence the convergence

%the spectrum of the matrix joue un role improtant dont le but final est la conception d'une methode rapide pour la resoltiopn du systeme

%Designing a fast method for solving a linear system has been the subject of numerous research over many year

%The matrices arising from the approximation of PDEs with constant coefficient by local methods such that Finite element, finite differences, etc., are Toeplitz matrices.

%In general, the matrices associated to theses systems are sparse

%%kifach yjou hadouk les matrice toeplitz wglt les method directe a eviter le role symbol

%%%%%%%%%%%%%%%%%kifach yjou les matrice w cha ya3touna w men ba3d nchouf le symbole w hadik la relation w ngoul 3la l role ta3ah aprs glt ngoul les method direct ghi chekla%%%%%%%%%%%%%%%la matrice a le spectre , la matrice le symbole le role du symbole le matrice de toeplitz et la glt le role de la matrice de toeplitz %then, the discretised problem is

% Thus, solving Now, if  we have the linear PDE of the form $$\mathcal{A}x=b$$

% due to the lack of the exact solution of the PDE numerical discretization is required  To solve a PDE we use   The resolution of

\chapter{Notations, definitions and preliminaries}
The goal of the current chapter relies in laying a groundwork for the succeeding chapters by providing the notations, the definitions, and some known results, that will be used throughout the whole thesis.
\section{General notations}

\begin{itemize}

     \item The space of real (resp. complex) $n\times m$ matrices is indicated by $\mathbb{R}^{m\times n}$ (resp. $\mathbb{C}^{m\times n}$).

%\item $\mathbb{R}^{m\times n}$ and $\mathbb{C}^{m\times n}$ indicate the space of real and complex $n\times m$ matrices.

  %\item $\mathbb{R}^{m\times n}$ (resp. $\mathbb{C}^{m\times n}$) is the space of real (resp. complex) $n\times m$ matrices.

     \item \#$S$ denotes the cardinality of the set $S$.

     \item Let $p\in[1,\infty]$, the $p$-norm of the vector $\mathbf{x}$ is denoted by $\Vert {\mathbf{x}}\Vert_{p}$ and defined as:
  $$\| \mathbf{x}\|_{p}=\left\lbrace
  \begin{array}{l}
  \left( \sum_{i=1}^{n}\vert x_{i} \vert^{p} \right)^{1/p}\mbox{, if  }p\in[1,\infty),\\
 \max_{i=1,\ldots,n}{\vert x_{i} \vert}\hbox{, if  }p=\infty  .
\end{array}\right.   $$
    % \item $\mathbf{0}$ (resp. $\mathbf{e}$) is a vector of all zeros (resp. ones).

        %\item[$\circ$] the transpose (resp.  the conjugate transpose) of the vector $\mathbf{x}$ is denoted by $\rm{\mathbf{x}}^{T}$ (resp. $\rm{\mathbf{x}}^{*}$), the transpose (resp.  the conjugate transpose) of the matrix $X$ is denote by ${X}^{T}$ (resp.$X^{*}$).   $\mathbf{x}^{T}$ and  ${X}^{T}$ denote respectively the transpose of the vector $\mathbf{x}$ and the transpose of the matrix $X$.

   \item If $A\in \mathbb C^{m\times n}$, we denote by ${{A}}^{T}$ the transpose of ${A}$ and by ${{A}}^{*}$ the conjugate transpose of ${A}$.

   \item If $A\in \mathbb C^{m\times n}$, we denote by $A^\dag$ the Moore-Penrose pseudoinverse of $A$ \\($A^\dag=A^{-1}$ whenever $A$ is invertible).

    \item[$\bullet$]The identity matrix (resp. the zero matrix) of order $m$ is denoted by $I_{m}$ (resp. $O_{m}$).
   %\item[$\circ$] For all $A\in \mathbb C^{m\times m}$, $O_{m}$ (resp. $I_{m}$) denote

    \item Let $A\in\mathbb{C}^{n\times n}$, the singular values and the eigenvalues of $A$ are denoted by $\sigma_1(A),\ldots,\sigma_n(A)$ and $\lambda_1(A),\ldots,\lambda_n(A)$, respectively.

    \item Given $A\in\mathbb{C}^{n\times n}$, $\Lambda(A)=\left\{\lambda_1(A),\ldots,\lambda_n(A)\right\}$ denotes the spectrum of $A$ and $\rho(A)=\max_{\lambda\in\Lambda(A)}|\lambda|$ is the spectral radius of $A$.

    \item A matrix $A\in\mathbb{C}^{n\times n}$ is said Hermitian if $A=A^*$.

    \item A matrix $A\in\mathbb{C}^{n\times n}$ is said Hermitian Positive Semi-Definite if $A$ is Hermitian and all its eigenvalues are positive. Furthermore $A$ is Hermitian Positive Definite if it is Hermitian and all its eigenvalues are nonnegative. Notice that there exist other equivalent characterisations, like that using the Rayleigh quotient.

    \item If $A\in\mathbb{C}^{n\times n}$ is Hermitian (which implies that its eigenvalues are real), $\lambda_{max}$ and $\lambda_{min}$ denote the maximal and minimal eigenvalue of $A$, respectively.

    \item A matrix $A\in\mathbb{C}^{n\times n}$ is said unitary if $A^{*}A=AA^{*}=I_n$.

    \item The matrix $A\in\mathbb{C}^{n\times n}$ is normal if and only if $A^{*}A=AA^*$.

     \item Given $A, B\in\mathbb{C}^{n\times n}$, if $A, B$ are Hermitian, the symbol $A>B$ (resp. $A\geq B$) is used to indicate that $A-B$ is Hermitian Positive Definite (resp. Hermitian Positive Semi-Definite).

    \item Given $A\in\mathbb{C}^{n\times n}$,  $A^{1/2}$ is the square root of $A$, if $A$ is Hermitian Positive Semi-Definite matrix.

     %Given a matrix $A\in\mathbb{C}^{m\times m}$ %Hermitian Positive Semi-Definite matrix, $A^{1/2}$ is %used to indicate the square root of the matrix $A$.

     \item Given a matrix norm $\|\cdot\|$, the condition number of an invertible matrix $A\in\mathbb{C}^{n\times n}$ is defined by: $$\kappa(A)=\|A\|\|A^{-1}\|.$$

     \item If $A, B\in\mathbb{C}^{n\times n}$, $A\thicksim B$ indicates that the matrix $A$ is similar to the matrix $B$.

   % \item $J_m^{(l)}$ is the matrix of order $m$ whose $(i,j)$
%entry equals $1$ if $i-j=l$ and zero otherwise.

    \item The tensor (Kronecker) product of two matrices $A\in\mathbb{C}^{n_1\times m_1}$ and $B\in\mathbb{C}^{n_2\times m_2}$ (with $n_1,n_2,m_1,m_2\in\mathbb{N}^*$) is defined as $$A\otimes B=\begin{bmatrix}
    a_{11}B& a_{12}B&\cdots&a_{1n_2} B\\
    a_{21}B& a_{22}B&\cdots&a_{2n_2}B\\
    \vdots&\vdots &\ddots&\vdots\\
    a_{n_{1}1}B&a_{n_{1}2}B&\cdots&a_{n_{1}n_{2}}B\\
    \end{bmatrix}\in\mathbb{C}^{n_1 n_2\times m_1 m_2}.$$
    \item Let $\mathbf{y}=[y_j]_{j=1}^{n}\in\mathbb{C}^{n}$, $M=\underset{j=1,\ldots,n}{\textrm{diag}}(\mathbf{y})$ is the diagonal matrix whose entries are $y_1,y_2,\ldots,y_n$, i.e., $M_{jj}=y_j$, for $j=1,\ldots,n$.

    %If $\mathbf{y}=[y_j]_{j=1}^{n}\in\mathbb{C}^{n}$, the notation $\underset{j=1,\ldots,n}{\textrm{diag}}(\mathbf{y})$ is the diagonal matrix whose entries are $y_1,y_2,\ldots,y_n$ i.e; $\left(\underset{j=1,\ldots,n}{\textrm{diag}}(\mathbf{y})\right)_{jj}=y_j$.

     \item $\mu_t$ is used to indicate the Lebesgue measure in $\mathbb R^t$, $t\ge 1$.

     \item  We denote by
${C}_c(\mathbb R)$ (resp. ${ C}_c(\mathbb C)$) the space of
continuous complex-valued functions with bounded support defined
on $\mathbb R$ (resp. $\mathbb{C}$).

     \item Given two positive sequences  $\{ x_{n}\}_{n}$ and $\{y_{n}\}_{n}$
     \begin{itemize}
  \item The notation $x_{n}=o(y_{n})$ means that $x_{n}/y_{n}\longrightarrow 0$ as $n\longrightarrow\infty$.
  \item The notation $x_{n}=O(y_{n})$ is used to indicate that there exists $C>0$ such that $x_{n}\leq C y_{n}$, $\forall n\geq 0$.
     \end{itemize}
     \item For $z\in\mathbb C$ and $\epsilon>0$, let $B(z,\epsilon)$
the disk with centre $z$ and radius $\epsilon$, $$B(z,\epsilon):=\{w\in\mathbb C:\,|w-z|<\epsilon\}.$$
For $S\subseteq\mathbb C$ and $\epsilon>0$, we denote by $B(S,\epsilon)$ the $\epsilon$-expansion of $S$,
defined as $B(S,\epsilon):=\bigcup_{z\in S}B(z,\epsilon)$.
   \item$\hat{\imath}$ is the imaginary unit, $\hat{\imath}^2=-1$.
   \item Given two real valued functions $h_1,h_2$ defined on a domain $D$, the notation $(h_1\sim h_2)$ means that there exist two constants  $C_1,C_2\in\mathbb{R}^+$ such that $$\forall x\in D, \qquad C_1h_1(x)\leq h_2(x)\leq C_2h_1(x).$$
   \item Let $f:D\subseteq\mathbb{R}^{d}\rightarrow\mathbb{C}$. We say that the function $f$ has a zero of order $s$ at $x_0$ if $f(x)\sim\|x-x_0\|^s$ at least locally. If $D\subseteq\mathbb{R}$ and $f$ is smooth enough, then the general definition is reduced to the following relations
   $$\begin{array}{l}
   f^{(j)}(x_0)=0 \hbox{, for } j<s,\\
   f^{(s)}(x_0)\neq 0,
   \end{array}$$
   where $f^{(j)}(x_0)$ is the $j-$th derivative of $f(x)$ at $x_0$.
      \end{itemize}

 \subsection{Multi-index notation}

   \begin{itemize}

     \item {A vector $\mathbf{j} $ of the form $\mathbf{j}=(j_{1},j_{2},\ldots,j_{t})\in\mathbb{Z}^{t}$} is called a multi-index.

     \item {For all $\mathbf{n}\in\mathbb{N}_+^{t}$, the multi-dimensional length is} $N(\mathbf{n})=\prod_{i=1}^{t}n_{i}$.
     \item $\mathbf{0}, \mathbf{e}, \mathbf{2},\ldots$ denote respectively the vectors $(0,\ldots,0),(1,\ldots,1),(2,\ldots,2,)\ldots$.

     \item  $\mathbf{n}\rightarrow \infty$ has the meaning that $\min_ {1\leq j\leq t} n_j\rightarrow \infty $.

     \item Given $\mathbf{m},\mathbf{n}\in\mathbb{Z}^{t}$,
     \begin{enumerate}

       \item $\mathbf{m}\leq \mathbf{n}\Leftrightarrow m_{r}\leq n_{r}\hbox{, }\forall r=1,\ldots, t  $.
     \item $\mathbf{k}=\mathbf{m},\ldots,\mathbf{n}$ means that the multi-index $\mathbf{k}$ varies from $\mathbf{m}$ to $\mathbf{n}$ following the  lexicographical way. \\
      For example, if $t=2$ the ordering as follows:
   % {\fontsize{10}{1} \selectfont
   \begin{multline}
     (m_{1},m_{2}),(m_{1},m_{2}+1),\ldots,(m_{1},n_{2}),(m_{1}+1,m_{2}),\\ (m_{1}+1,m_{2}+1),\ldots,(m_{1}+1,n_{2}),\ldots,(n_{1},m_{2}),(n_{1},m_{2}+1),\ldots,(n_{1},n_{2}).
\end{multline}
     \item If $\mathbf{m}\leq \mathbf{n}$, $\sum_{\mathbf{k}=\mathbf{m}}^{\mathbf{n}}$ $\left(\hbox{resp. } \prod_{\mathbf{k}=\mathbf{m}}^{\mathbf{n}}\right)$ is used to indicate the summation (resp. product) over all the multi-indices $\mathbf{k}$ in the range $\mathbf{m},\ldots,\mathbf{n}$.
      \end{enumerate}
     \item If $\mathbf{n}\in\mathbb{Z}^{t}$, $\mathbf{x}=[x_{\mathbf{i}}]_{\mathbf{i=e}}^{\mathbf{n}} $ is a vector of size $N(\mathbf{n})$ and it is given by  \\
     $\mathbf{x}=
     \left[\begin{array}{c}
     x_{(1,1,\ldots,1,1)}\\
     x_{(1,1,\ldots,1,2)}\\
     \vdots\\
     x_{(1,1,\ldots,1,n_{t})}\\
     \hline
     x_{(1,1,\ldots,2,1)}\\
     x_{(1,1,\ldots,2,2)}\\
     \vdots\\
     x_{(1,1,\ldots,2,n_{t})}\\
     \hline
      \vdots\\
      \vdots\\
     \hline
     x_{(n_{1},n_{2},\ldots,1)}\\
     x_{(n_{1},n_{2},\ldots,2)}\\

     \vdots\\
     x_{(n_{1},n_{2},\ldots,n_{t})}
     \end{array}\right].$

     \item Given $\mathbf{n}\in\mathbb{Z}^{t}$, $\mathbf{X}=[x_{\mathbf{ij}}]_{\mathbf{i,j=e}}^{\mathbf{n}}$ is the $N(\mathbf{n})\times N(\mathbf{n})$ matrix. \\ For instance in the case $t=2$ the matrix $\mathbf{X}=[x_{\mathbf{ij}}]_{\mathbf{i,j=e}}^{\mathbf{n}}$  is given by:\\
     %
     %\renewcommand{\arraystretch}
    % $${\small\hspace{-1.2cm}
    %\renewcommand{\arraystretch}{0.7}\setlength{\tabcolsep}{1pt}

    $$ {\hspace{-1.2cm}\small\left[\resizebox{\textwidth}{!}{$\begin{array}{cccc|ccc|l|cccccc}
     x_{(1,1),(1,1)}& x_{(1,1),(1,2)}&\ldots&x_{(1,1),(1,n_{2})}&x_{(1,1),(2,1)}&\ldots&x_{(1,1),(2,n_{2})}&\ldots&x_{(1,1),(n_{1},1)}&\ldots& x_{(1,1),(n_{1},n_{2})} \\
     x_{(1,2),(1,1)}& x_{(1,2),(1,2)}&\ldots&x_{(1,2),(1,n_{2})}&x_{(1,2),(2,1)}&\ldots&x_{(1,2),(2,n_{2})}&\ldots&x_{(1,2),(n_{1},1)}&\ldots&x_{(1,2),(n_{1},n_{2})}\\
     \vdots&\vdots&&\vdots&\vdots&&\vdots&\cdots&\vdots&\cdots&\vdots\\
     \vdots&\vdots& &\vdots&\vdots& &\vdots&\ldots&\vdots&\vdots&\vdots\\
  x_{(1,n_{2}),(1,1)}&x_{(1,n_{2}),(1,2)}&\ldots & x_{(1,n_{2}),(1,n_{2})}&x_{(1,n_{2}),(2,1)}&\ldots&x_{(1,n_{2}),(2,n_{2})}&\ldots&x_{(1,n_{2}),(n_{1},1)}&\ldots&x_{(1,n_{2}),(n_{1},n_{2})}\\
  \hline
  x_{(2,1),(1,1)}&x_{(2,1),(1,2)}&\ldots&x_{(2,1),(1,n_{2})}&x_{(2,1),(2,1)}&\ldots&x_{(2,1),(2,n_{2})}&\ldots&x_{(2,1),(n_{1},1)}&\ldots&x_{(2,1),(n_{1},n_{2})}\\

      \vdots&\vdots&&\vdots&\vdots&&\vdots&\cdots&\vdots&\cdots&\vdots\\
  x_{(2,n_{2}),(1,1)}&x_{(2,n_{2}),(1,2)}&\ldots&x_{(2,n_{2}),(1,n_{2})}&x_{(2,n_{2}),(2,1)}&\ldots&x_{(2,n_{2}),(2,n_{2})}&\ldots&x_{(2,n_{2}),(n_{1},1)}&\ldots&x_{(2,n_{2}),(n_{1},n_{2})}\\
  \hline
  \vdots&\vdots&\vdots&\vdots&\vdots&\vdots&\vdots&\ldots&\vdots&\vdots&\vdots\\

\vdots&\vdots&\vdots&\vdots&\vdots&\vdots&\vdots&\ldots&\vdots&\vdots&\vdots\\
\vdots&\vdots&\vdots&\vdots&\vdots&\vdots&\vdots&\ldots&\vdots&\vdots&\vdots\\

\hline

      x_{(n_{1},1),(1,1)}&x_{(n_{1},1),(1,2)}&\ldots&x_{(n_{1},1),(1,n_{2})}&x_{(n_{1},1),(2,1)}&\ldots&x_{(n_{1},1),(2,n_{2})}&\ldots&x_{(n_{1},1),(n_{1},1)}&\ldots&x_{(n_{1},1),(n_{1},n_{2})}\\
      \vdots&\vdots&&\vdots&\vdots&&\vdots&\cdots&\vdots&\cdots&\vdots\\
  x_{(n_{1},n_{2}),(1,1)}&x_{(n_{1},n_{2}),(1,2)}&\ldots&x_{(n_{1},n_{2}),(1,n_{2})}&x_{(n_{1},n_{2}),(2,1)}&\ldots&x_{(n_{1},n_{2}),(2,n_{2})}&\ldots&x_{(n_{1},n_{2}),(n_{1},1)}&\ldots&x_{(n_{1},n_{2}),(n_{1},n_{2})}\\
  \end{array}$} \right]}.$$

   \end{itemize}

   \section{Matrix norms}
   Given  $A,B\in\mathbb C^{n\times n}$, a matrix norm is an application $\|\cdot\|:\mathbb{C}^{n\times n}\rightarrow\mathbb{R}^+$ satisfiying
   \begin{itemize}
   \item $\|A\|=0\Leftrightarrow A=0$;
   \item $\|\lambda A\|=|\lambda|\|A\|$, $\forall\lambda\in\mathbb{C}$;
   \item $\|A+B\|\leq\|A\|+\|B\|$;
   \end{itemize}

   There exists various type of matrix-norm:
   \begin{itemize}

   \item \subsection{The norms induced by vector norms}For $p\in[1,\infty]$, the $p$-norm of the matrix $A$ is defined by $$\|A\|_{p}=\sup_{u\in\mathbb{C}^{n}, \|u\|_{p}=1}\|Au\|_{p}=\sup_{u\in\mathbb{C}^{n},u\neq 0}\dfrac{\|Au\|_{p}}{\|u\|_{p}}. $$ For $p=1,2,\infty$, the norm $\|A\|_{p}$ is presented as follows \cite{golub2013matrix}:\\
   $$\begin{array}{l}
   \|A\|_{1}=\max_{j=1,\ldots,n}\sum_{i=1}^{n}|a_{ij}|,\\
     \\
   \|A\|_{2}=\sqrt{\rho(A^{*}A)}=\sigma_{max} \hbox{ (spectral norm) ,}\\
    \\
   \|A\|_{\infty}=\max_{i=1,\ldots,n}\sum_{j=1}^{n}|a_{ij}|.
   \end{array}$$

   \item \subsection{Frobenius norm}
$$\|A\|_{F}=\left(\textrm{tr}(A^{*}A)\right)^{1/2}=\left(\sum_{i=1}^{n}\sum_{j=1}^{n}|a_{ij}|^{2}\right)^{1/2}=\sqrt{\sigma_{1}^{2}(A)+\ldots+\sigma_{n}^{2}(A)}.$$

   \item \subsection{The Schatten norms}
    Given $1\le p\le\infty$,  $\|| A\||_p$ is the Schatten $p$-norm of $A$, i.e., the $p$-norm of the vector $\sigma(A)=(\sigma_1(A),\ldots,\sigma_{n}(A))$. % formed by the singular values of $A$.
        \\The Schatten $\infty$-norm $\||A\||_\infty$ is the largest singular value of $A$ and coincides
        with the spectral norm  $$\||A\||_\infty=\sigma_{max}=\|A\|_{2}.$$ \\The Schatten 1-norm $\|A\|_1$ is the sum of the singular values of $A$ and is often referred to as the trace-norm of $A$ $$\||A\||_1=\sum_{i=1}^{n}\sigma_i(A).$$ The Schatten 2-norm $\||A\||_2$ coincides with the  Frobenius norm of $A$ $$\||A\||_{2}=\|A\|_{F}).$$ For more details on Schatten $p$-norms, see \cite{bhatia1997matrix}.

\item \subsection{The Euclidean norm weighted by a matrix}

If $A\in\mathbb{C}^{n\times n}$ is an Hermitian Positive Definite matrix, then the Euclidean norm of a matrix weighted by $A$ is defined by
\begin{equation}\label{eclnorm}
\|B\|_{A}=\max_{\|x\|_{A}=1}\|Bx\|_{A}
\end{equation}
where $\|x\|_A=\|A^{\frac{1}{2}}x\|_2$ is called the energy norm.\\
It is easy to verify $\|B\|_{A}=\|A^{\frac{1}{2}}BA^{-\frac{1}{2}}\|_{2}$. In fact,
\begin{equation}\|B\|_{A}=\max_{\|x\|_{A}=1}\|Bx\|_{A}=\max_{\|A^{\frac{1}{2}}x\|_2=1}\|A^{\frac{1}{2}}Bx\|_{2}=\max_{\|y\|_2=1}\|A^{\frac{1}{2}}BA^{-\frac{1}{2}}y\|_{2}=\|A^{\frac{1}{2}}BA^{-\frac{1}{2}}\|_{2}
\end{equation}
where $y=A^{\frac{1}{2}}x$.

   \end{itemize}

      Here, we report some useful properties on the matrix norms that are often utilised, (see \cite{golub2013matrix}).

   \begin{enumerate}
   \item The induced norms are sub-multiplicative {,} i.e $\forall A,B\in\mathbb{C}^{n\times n}$,\\ $\|AB\|_{p}\leq\|A\|_{p}\|B\|_{p}.$

\item The induced norms satisfy $\rho(A)\leq\|A\|_{p}$, $\forall A\in\mathbb{C}^{n\times n}$.

%\item   The spectral norm is unitarily invariant, i.e
  % $\|UAV\|_{2}=\|A\|_{2}\hbox{, }\forall A\in\mathbb{C}^{n\times n}$
 %and $\forall\hbox{ } U,V\in\mathbb{C}^{n\times n}$ unitary matrices.

 \item {The Schatten $p-$norms are unitarily invariants, i.e}
   $\||UAV\||_{p}=\||A\||_{p},\\ \forall A\in\mathbb{C}^{n\times n}$
 and $\forall\hbox{ } U,V\in\mathbb{C}^{n\times n}$ unitary matrices.

   \item The estimation of the spectral norm can be deduced by the following relation
   $$\|A\|_{2}\leq\sqrt{\|A\|_{1}\|A\|_{\infty}}\hbox{, }\forall A\in\mathbb{C}^{n\times n}.$$

\item $\max_{1\leq i,j\leq n}|a_{ij}|\leq\|A\|_{2}\leq n\max_{1\leq i,j\leq n}|a_{ij}|$.
\item If $B\in\mathbb{C}^{n-1\times n-1}$ is sumbatrix of $A\in\mathbb{C}^{n\times n}$ then $\|B\|_{p}\leq\|A\|_{p}$.
\item The Frobenius norm is unitarily invariant, that is, $\forall A\in\mathbb{C}^{n\times n}$, $\forall\hbox{ } U,V\in\mathbb{C}^{n\times n} $  unitary matrices, $\|UAV\|_{F}=\|A\|_{F}$. The same is true for every Schatten norm.
   \item The Frobenius norm is equivalent to the spectral norm and the following inequalities hold
$$\|A\|_{2}\leq \|A\|_{F}\leq\sqrt{n}\|A\|_{2}.$$
   \item All the norms acting on {f}inite dimensional spaces are topologically equivalent.
    \end{enumerate}

\section{$L^{p}$ spaces}
Let $D\subseteq\mathbb{R}^{t}$ a measurable set, for $p\geq 1$ we define
\begin{itemize}

\item  The space $L^{p}(D)=\left\lbrace f:D \longrightarrow \mathbb{C}\mbox{ : }  f\textrm{  measurable}\mbox{, such that }\displaystyle\int_{D} \vert f\vert<\infty \right\rbrace $.\\
  \item The space of all measurable functions essentially bounded is denoted by:\\
      $$L^{\infty}(D)=\left\lbrace f:D\longrightarrow \mathbb{C}\mbox{ : }  f\textrm{ measurable}\mbox{, such that } \textrm{ess}\sup_{D}\vert f\vert<\infty \right\rbrace .$$
  \item If $f\in L^{p}(D)$, the $L^{p}$-norm of the function $f$ is denoted by $\Vert f\Vert_{p}$ and defined as:
  $$
  \begin{array}{l}
  \| f\|_{p}=\left( \int_{D}\vert f \vert^{p} \right)^{1/p}\mbox{, if  }p\in[1,\infty),\\
 \| f\|_{\infty}=\textrm{ess}\sup_{D}\vert f\vert\hbox{, if  }p=\infty.  \end{array}   $$
\end{itemize}
Let $\mathbf{f} $ be a matrix-valued function $\mathbf{f} :D\subseteq\mathbb{R}^{t}\longrightarrow \mathbb{C}^{r\times r}$, for $p\in[1,\infty)$ we set
$$L^{p}(D,r)=\left\lbrace \mathbf{f}  :D\longrightarrow \mathbb{C}^{r\times r} \hbox{ such that } \mathbf{f}\in L^{p}(D) \right\rbrace .$$

\begin{itemize}
\item A matrix-valued function $\mathbf{f}$ is said to be measurable (resp.continuous, Riemann-integrable, in $L^p(D)$, etc.) if its
components $f_{\alpha\beta}:D\to\mathbb{C},\
\alpha, \beta=1,\ldots,r$, are measurable (resp. continuous,
Riemann-integrable, in $L^p(D)$, etc.).
 \item If $\mathbf{f}_m,\mathbf{f}:D\subseteq\mathbb
R^t\to\mathbb C^{r\times r}$ are measurable matrix-valued
functions, we say that $\mathbf{f}_m$ converges to $\mathbf{f}$ in measure (resp., almost everywhere, in $L^p(D)$, etc.) if $(\mathbf{f}_m)_{\alpha\beta}$ converges to
$\mathbf{f}_{\alpha\beta}$ in measure (resp., almost everywhere, in $L^p(D)$, etc.) for
all $\alpha,\beta=1,\ldots,r$.

\item If $\mathbf{f}\in L^{p}(D,r)$, we define
$$  \begin{array}{ll}
  \| \mathbf{f}\|_{p}=\left( \int_{D} \| \mathbf{f}(\boldsymbol{\theta})\|_{p}^{p} \right)^{1/p},&\mbox{if  }p\in[1,\infty),\\
 \| \mathbf{f}\|_{\infty}=\textrm{ess}\sup_{\boldsymbol{\theta}\in D}\| \mathbf{f}(\boldsymbol{\theta})\|_{\infty},&\hbox{if  }p=\infty.  \end{array} $$

 \item If $1\leq p,q\leq\infty$ such that $\frac{1}{p}+\frac{1}{q}=1$ and $\textbf{f}\in L^{p}(D,r)$, $\mathbf{g}\in L^{q}(D,r)$ then the inequality of Holder holds for $p$-schatten norm
 $$\|\mathbf{f}\mathbf{g}\|_{1}\leq\|\mathbf{f}\|_{p}\|\mathbf{g}\|_{q}.$$

 \item When $\mathbf{f}$ is a matrix-valued function, $\sigma_{i}(\mathbf{f}(\boldsymbol{\theta}))$ (resp. $\lambda_{i}(\mathbf{f}(\boldsymbol{\theta}))$), $i=1, \ldots, r$ is used to indicate the evaluation of singular value (resp. the eigenvalue) functions at the point $\boldsymbol{\theta}\in D$.

\section{Iterative Methods }
 Now, in this part, we will start by providing a brief description of the two major types of approaches for solving a linear system. We will next focus on the iterative approach, discussing preconditioning for Toeplitz structures. The Multigrid method is finally introduced.
 %%Suppose that we have the linear system Ax = b, (1) with A invertible and A invertible, and we have A invertible and n, x, b, and Cn (therefore the system admits one and only one solution).
 \\Suppose we have the following linear system
 \begin{equation}\label{linearpb}
 A\mathbf{x}=\mathbf{b},
 \end{equation} where $\mathbf{b}\in\mathbb{C}^n$ and $A\in\mathbb{C}^{n\times n}$ is invertible, thus the system admits only one solution.\\
  Numerical techniques for solving the linear system (\ref{linearpb}) break down into two major classes:
  \begin{enumerate}
  \item \textbf{Direct methods}: they consist in converting a generic linear system through a finite number of operations into a linear system with a specific structure that simplifies its resolution (for instance Gauss elimination method, $LU$ and Cholesky factorization etc...)

  \item \textbf{Iterative methods}: they consist in building a sequence of vectors $\{\mathbf{x}^{(k)}\}_{k}$ converging to exact solution $\overline{\mathbf{\mathbf{x}}}=A^{-1}b$ such that $\underset{k\rightarrow\infty}\lim \mathbf{x}^{(k)}=\overline{\mathbf{\mathbf{x}}}$. These types of methods are effective for sparse matrices because they retain the sparsity of the matrix, and hence they are highly favored for large problems, {as} they demand less computer's memory. They are also more suitable for solving ill-conditioned problems {since} they allow a better control of the error.

  \end{enumerate}

\begin{Definition}
An iterative method for solving the linear system (\ref{linearpb}) has the form \begin{equation}\label{iterpb}
\mathbf{x}^{(k)}=Q\mathbf{x}^{(k-1)}+q\quad k=1,2,\ldots, \quad\mbox{ and } \mathbf{x}^{(0)} \mbox{ is an initial vector }
\end{equation}
where $Q\in\mathbb{C}^{n\times n}$ is called the iteration matrix and $q\in\mathbb{C}^{n}$ is a fixed vector. The method (\ref{iterpb}) is called stationary iterative method, where the word "stationary" highlights the fact that the matrix of iteration $Q$ is fixed a priori.
\end{Definition}
A method as in ($\ref{iterpb}$) is obtained by decomposing the matrix $A$ in the form $A=M-N$, where $M\in\mathbb{C}^{n\times n}$ is an invertible matrix. We can then
rewrite the system ($\ref{linearpb}$) in the form $\mathbf{x} = M ^{-1}(N\mathbf{x}+ b)$,  setting $Q=M^{-1}N$ and $q=M^{-1}b$ we obtain (\ref{iterpb}).
\begin{Definition}
The method (\ref{iterpb}) is said to be consistent with (\ref{linearpb}) if the solution $\overline{\mathbf{\mathbf{x}}}$ of (\ref{linearpb}) is a fixed point, i.e., $\overline{\mathbf{\mathbf{x}}}=Q\overline{\mathbf{\mathbf{x}}}+q$.
\end{Definition}

\begin{Definition}
We say that the method (\ref{iterpb}) is convergent if for every initial choice $\mathbf{x}^{(0)}\in\mathbb{C}^{n\times n}$, the sequence $\{\mathbf{x}^{(n)}\}_{n\in\mathbb{N}}$ converges to $\overline{\mathbf{\mathbf{x}}}$  \\$(\lim_{k\rightarrow\infty}\|\overline{\mathbf{\mathbf{x}}}-\mathbf{x}^{(k)}\|=0)$.
\end{Definition}
The next {theorem} establishes the link between the convergence of the method (\ref{iterpb}) and the iteration matrix $Q$

\begin{Theorem}\cite{ciarlet1990introduction}
Let $Q\in\mathbb{C}^{n\times n}$, then the following propositions are equivalent:
\begin{itemize}
\item The method (\ref{iterpb}) is convergent for any choice of $b\in\mathbb{C}^{n}$ and $\mathbf{x}^{(0)}$;
\item $\rho(Q)<1$;
\item There exists  $p$ such that $\|Q\|_{p}<1$.
\end{itemize}
\end{Theorem}

Now, we provide the definition of the optimality for the iterative methods, which is useful for evaluating the efficiency and the performance of the method.
\begin{Definition}[Optimality]
An iterative method is said to be optimal for a class of a sequence of linear systems $\{A_n\mathbf{x}_n=b_n\}_n$ if
\begin{enumerate}
\item the number of iterations is bounded by a constant and independent from the size of the matrix $A_n$,
\item each iteration has a cost proportional to the matrix-vector product with matrix $A_n$.
\end{enumerate}

% problems
%Let $\{A_nx_n=b_n\}_n$ be a sequence of linear systems
%An iterative method is optimal
\end{Definition}

 \subsection{Preconditioning}
The basic principle behind the preconditioning process is to replace the ill-conditioned system
\begin{equation}
 A_{n}\mathbf{x}=\mathbf{b}
 \end{equation}
where $A_n\in\mathbb{C}^{d_n\times d_n}$, $\mathbf{x,b}\in\mathbb{C}^{d_n}$ ($\kappa(A_{n})\rightarrow\infty$ when $n\rightarrow\infty$) by an analogous system $\tilde{A}_n\mathbf{x}=\tilde{\mathbf{b}}$ for which the new system converges {faster} than the original system (the number of iterations required for the convergence is reduced). \\
A common strategy is to consider a Hermitian Positive Definite matrix  \\$P\in\mathbb{C}^{d_n\times d_n}$ such that
 \begin{equation}
 P^{-1}_nA_n\mathbf{x}=P^{-1}_n\mathbf{b}.
 \end{equation}
 Since the convergence of the iterative method depends a lot on the spectral properties of the associated matrix, $P_n$ should be chosen in such a way, the matrix $P_n^{-1}A_n$ is better conditioned and its spectrum is clustered around one. Consequently, the matrix $P_n$  should be determined in accordance with the following conditions:
 %based on the following criteria:
 %able to meet the following two conditions:
 \begin{enumerate}
 \item {The construction of the preconditioner demands  at most the cost of $A_n\mathbf{x}$;}
 \item The solution of $P_n\mathbf{x}=\mathbf{y}$ has {at most} the cost of {$A_n\mathbf{x}$ product};
 \item $P_n$  has the same spectral properties as $A_n$, i.e; the spectrum of\\$P_n^{-1}A_n-I_n$ is strongly clustered around zero or the condition number of $P_n^{-1}A_n$ is bounded by $c$, i.e $\exists c>0,\kappa(P_n^{-1}A_n)<c$, $\forall n$.
\end{enumerate}
There are two extreme cases: $P_n=I_n$ satisfies the requirements 1.,{2.} but not 3., in which case convergence is not accelerated; when $P_n=A_n$ satisfies the first and the third requirements, but it fails to meet the second: in this case the convergence is immediate but not beneficial, due to the difficulty of the problem. Creating a preconditioner that is efficient should help to strike a good compromise between the three requirements of creating a matrix $P_n$ similar to $A_n$, at least spectrally, but less computationally expensive to invert.

 \subsubsection{Preconditioners for Toeplitz matrices}
 Because there is so much information on preconditioners for structured matrices in the literature, we will only mention a few results in the case of Toeplitz matrices, which have been explored by many authors, see for instance \cite{chan1996conjugate,ng2004iterative,serra1999korovkin}, (for more information regarding Toeplitz matrices, see Section \ref{ToepCirc}).
\begin{itemize}
\item\textbf{Strang's preconditioner}: In 1986 G.Strang \cite{strang1986proposal} introduced the first {Circulant} preconditioner (for more information concerning Circulant matrices, see Section \ref{ToepCirc}) given by
$$S=\begin{bmatrix}
 a_0&a_1 &\cdots &a_{\lfloor\frac{n}{2}\rfloor} &a_{\lfloor\frac{n-1}{2}\rfloor} & \cdots& a_2 &a_1 \\
a_1 &a_0 &a_1 &\cdots & a_{\lfloor\frac{n}{2}\rfloor}&a_{\lfloor\frac{n-1}{2}\rfloor} &\cdots &a_2 \\
 \vdots&a_1 &a_0 &a_1 & \cdots&\ddots &\ddots &\vdots \\
 a_{\lfloor\frac{n}{2}\rfloor}&\ddots &\ddots &\ddots &\ddots &\ddots &\ddots &a_{\lfloor\frac{n-1}{2}\rfloor} \\
 a_{\lfloor\frac{n-1}{2}\rfloor}&a_{\lfloor\frac{n}{2}\rfloor} &\ddots &\ddots &\ddots & \ddots& \ddots&a_{\lfloor\frac{n}{2}\rfloor} \\
\vdots &\ddots &a_{\lfloor\frac{n}{2}\rfloor} &\cdots &\ddots &a_0 &a_1 & \vdots\\
 a_2&\cdots &a_{\lfloor\frac{n-1}{2}\rfloor} &a_{\lfloor\frac{n}{2}\rfloor} &\cdots &a_1 &a_0 &a_1 \\
a_1 &a_2 &\cdots &a_{\lfloor\frac{n-1}{2}\rfloor} &a_{\lfloor\frac{n}{2}\rfloor} &\cdots &a_1 & a_0 \\
\end{bmatrix}$$
where the coefficients $a_j$ are the {coefficients} of the Toeplitz matrix\\$T_n(a)=[a_{i-j}]_{i,j=1,\cdots,n}$. {The matrix $S$ has interesting properties given by} $S=\underset{C:{\rm Circulant}}{\rm{arg}\min}\|C-T_n(a)\|_1$ and $S=\underset{C:{\rm Circulant}}{\rm{arg}\min}\|C-T_n(a)\|_\infty$.\\
Under some conditions on the generating function $a$ the PCG method with the Strang preconditioner {converges} superlinearly{, see} \cite{chan1996conjugate,estatico2008superoptimal,serra1999superlinear}.
\item\textbf{T.Chan's preconditioner}: Chan's preconditioner is given by
$$\hat{C}=\underset{C:{\rm Circulant}}{\rm{arg}\min}\|C-T_n(a)\|_F$$  and the entries of the matrix are given by
$$c_j=\left\{
\begin{array}{ll}
\frac{(n-j)a_j+ja_{j-n}}{n},& 0\leq j<n\\
c_{n+j,}&0<-j<n

\end{array}.\right.$$
The optimality of the PCG with $P_n=\hat{C}$ is ensured when the generating function $a$ is positive continuous function, see \cite{chan1992circulant,serra1999korovkin,serra1999superlinear}.

\end{itemize}

 \subsection{Multigrid Methods}

%The classical iterative methods produce a poor approximation of the solution when $\kappa(A_n)$ diverges as $n\rightarrow\infty$, this is caused by the space generated by the eigenvectors of $A_n$ associated to eigenvalues that are close to zero, resulting in global slowdown of the method. This impediment can be addressed by utilising the coarse grid correction, which enable the solution to be approximated successfully in this subspace. It is feasible to construct an optimal method for classes of problems for which an adequate coarse grid correction may be defined using the multigrid methods. These methods are extremely effective for dealing with issues that display behaviour at several scales

The classical iterative methods produce a poor approximation of the solution when $\kappa(A_n)$ diverges as $n\rightarrow\infty$. The slow convergence is caused by the space generated by the eigenvectors of $A_n$ associated to eigenvalues that are close to zero, resulting in global slowdown of the method. This impediment can be addressed by using either preconditioning in Krylov algorithms or Multigrid methods. The latter methods are extremely effective for dealing with quite ill-conditioned linear systems, under the condition of having a partial knowledges of the ill-conditioned subspaces.
The principle of Multigrid, when applied to differential problems, is to speed up the convergence of a basic iterative approach, which in general reduces effectively the oscillatory modes of the error (short-wavelength errors are eliminated), and solves the problem in a coarse-grid (where the resolution of the problem is cheaper, and the smooth mode of the error appears more oscillatory). The coarse problem contains also both short and long-wavelength errors. Additionally, a combination of relaxation and coarser grids can be used to overcome the problem. Then the procedure is continued until a grid is found where the cost is much cheaper compared to the first one (the finer one). For a detailed study of Multigrid methods, we recommend \cite{briggs2000multigrid,Hack,ruge1987algebraic}. In this section, we concentrate on the Algebraic Multigrid method by presenting Two-Grid Method and the V-cycle method.

%the coarse grid correction, which enable the solution to be approximated successfully in this subspace. It is feasible to construct an optimal method for classes of problems for which an adequate coarse grid correction may be defined using the multigrid methods. These methods are extremely effective for dealing with issues that display behaviour at several scales

Consider the linear system $A_n\mathbf{x}=\mathbf{b}$ where $A_n\in\mathbb{C}^{n\times n}$ and $\mathbf{x},\mathbf{b}\in\mathbb{C}^{n}$
and let $\mathbf{x}^{(k)}$ an approximation to $\overline{\mathbf{x}}$. There are two important measures
\begin{itemize}
\item The algebraic error $$\mathbf{e}^{(k)}=\overline{\mathbf{x}}-\mathbf{x}^{(k)},$$
\item The residual $$\mathbf{r}^{(k)}=\mathbf{b}-A_n\mathbf{x}^{(k)},$$

\end{itemize}
however, since the algebraic error is not available {whenever} the exact solution $\overline{\mathbf{x}}$ is unknown, we use the residual which compute how much the approximation fails to solve the original problem $A_n\mathbf{x}=\mathbf{b}$.
By uniqueness of the solution $\mathbf{r}^{(k)}=\mathbf{0}_n\Longleftrightarrow \mathbf{e}^{(k)}=\mathbf{0}_n$.\\
Since $A_n\overline{\mathbf{x}}=\mathbf{b}$ and by using the definitions of the residual and the algebraic error we obtain the following residual equation $$A_n\mathbf{e}^{(k)}=\mathbf{r}^{(k)}.$$
In fact $\mathbf{r}^{(k)}=\mathbf{b}-A_n\mathbf{x}^{(k)}=A_n\overline{\mathbf{x}}-A_n\mathbf{x}^{(k)}=A_n\mathbf{e}^{(k)}.$\\
When $\mathbf{b}$ of the original system is substituted by $\mathbf{r}^{(k)}$, the solution $\mathbf{e}^{k}$ of the residual equation satisfies the same set of equations as the unknown $\mathbf{x}$.

We now have a handle on how the residual equation will be put to use: after obtaining the approximation $\mathbf{x}^{(k)}$ by some iterative method, we compute the residual $\mathbf{r}^{(k)}=\mathbf{b}-A_n\mathbf{x}^{(k)}$, solve the residual equation and lastly apply the residual correction to improve the approximation $\mathbf{x}^{(k+1)}= \mathbf{x}^{(k)}+\mathbf{e}^{(k)}$.\\
The residual equation $A_n\mathbf{e}^{(k)}=\mathbf{r}^{(k)}$ is not simpler to solve than the original problem; however, here we know that the exact solution {'the error'} is smooth if $x^{(k)}$ is obtained by an iterative method, contrary to the original problem, we do not know anything about the behaviour of the exact solution. Therefore instead of solving the residual equation in the fine grid, we project it into the coarse grid where the problem will be solved; then, by re-projecting the solution found on the coarse grid into the fine grid, we obtain a new approximate solution that is considerably closer to $\overline{\mathbf{x}}$ than $\mathbf{x}^{(k)}$ was, this procedure called   Coarse Grid Correction (CGC).\subsubsection{Algebric Two-Grid Method (TGM)}\label{sec:multigrid_theory}
We start by taking  into consideration
 the generic linear system $A_n \mathbf{x}_n =\mathbf{b}_n$ with large
dimension $n$, where $A_n \in \mathbb{C}^{n\times n}$ is a
Hermitian Positive Definite matrix and $\mathbf{x}_n,\mathbf{b}_n \in
\mathbb{C}^{n}$. Let $n_0=n >n_1> \ldots
>n_s> \ldots > n_{s_{\min}}$ and let $P^{s+1}_s\in
\mathbb{C}^{n_{s+1}\times n_{s}}$  be a full-rank matrix for
any $s$. At last, let us denote by $\mathcal{V}_s$  a class of
stationary iterative methods for given linear systems of dimension
$n_s$.\\
In accordance with \cite{Hack}, the algebraic two-grid
Method (TGM) can be easily seen a stationary iterative method whose generic steps are reported below.

%TGM
\[
\begin{tabular}{lcl}
     \multicolumn{3}{c} {$\mathbf{x}_s^{\mathrm{out}}=\mathcal{TGM}(s,\mathbf{x}_s^{\mathrm{in}},\mathbf{b}_s)$} \\ %\mathcal{V}_{s,\mathrm{pre}},\mathcal{V}_{s,\mathrm{post}})} \\
     \hline \\
    \fbox{\begin{tabular}{l} $\mathbf{x}_s^{\mathrm{pre}}=\mathcal{V}_{s,\mathrm{pre}}^{\nu_\mathrm{pre}}(\mathbf{x}_s^{\mathrm{in}},b_s)
    \hskip 0.75cm \phantom{pA}
    $ \end{tabular} }  && {Pre-smoothing iterations}\\
    \ \\
   \fbox{\begin{tabular}{l}
          $\mathbf{r}_{s}=\mathbf{b}_s-A_s \mathbf{x}_s^{\mathrm{pre}}$ \\
          $\mathbf{r}_{s+1}=(P^{n_{s+1}}_{n_s})^*_{n_s} \mathbf{r}_{s}$     \\
          $A_{s+1}=(P^{n_{s+1}}_{n_s})^* A_s P^{n_{s+1}}_{n_s}$  \\
          $\mathrm{Solve\ } A_{s+1}\mathbf{y}_{s+1}=\mathbf{r}_{s+1}$  \\
          $\hat {\mathbf{x}}_s=\mathbf{x}_s^{\mathrm{pre}}+P^{n_{s+1}}_{n_s} \mathbf{y}_{s+1}$\\
    \end{tabular}
    } && {Exact Coarse Grid Correction {(CGC)}}\\
    \ \\
    \fbox{\begin{tabular}{l}$\mathbf{x}_s^{\mathrm{out}}=\mathcal{V}_{s,\mathrm{post}}^{\nu_\mathrm{post}}(\hat
    {\mathbf{x}}_s,\mathbf{b}_s)\hskip 0.8cm \phantom{pA} $\end{tabular}} && {Post-smoothing iterations} \\ %\quad x_s^{\mathrm{out}}=x_s^{\mathrm{post}}\\
\end{tabular}
\]
where we refer to the dimension  $n_s$ by means of its subscript
$s$.

In the first and last  steps, a \emph{pre-smoothing
iteration} and a \emph{post-smoothing iteration} are applied
 $\nu_{\rm pre}$ times and $\nu_{\rm post}$ times, respectively, in accordance with the
considered stationary iterative method in the class $\mathcal{V}_s$.
Furthermore, the intermediate steps define the  \emph{exact
coarse grid correction operator}, which is depending on the considered
projector operator $P^{s+1}_s$.
The resulting  iteration matrix of the TGM is then defined as
\begin{eqnarray}
 TGM_s &=& V_{s,\mathrm{post}}^{\nu_\mathrm{post}}
           CGC_s
           V_{s,\mathrm{pre}}^{\nu_\mathrm{pre}},\\
 CGC_s &=& {I^{(s)}}-P^{n_{s+1}}_{n_s}  A_{s+1}^{-1} (P^{n_{s+1}}_{n_s})^{*} A_s \quad
A_{s+1}= (P^{n_{s+1}}_{n_s})^* A_s P^{n_{s+1}}_{n_s},
\label{eq:CGC_s}
\end{eqnarray}
where $V_{s,\mathrm{pre}}$ and $V_{s,\mathrm{post}}$ represent the
pre-smoothing and post-smoothing iteration matrices, respectively, and { $I^{(s)}$ is the identity matrix at the $s$th level}.\\
The following theorem allows us to estimate an upper bound on the TGM's speed of convergence.
\begin{Theorem}\cite{ruge1987algebraic}
Let $A\in\mathbb{C}^{n\times n}$ be a Hermitian Positive Definite matrix, $P_{m}^{n}\in\mathbb{C}^{n\times m}$ a full-rank matrix with $n>m$  and $V_{n,\mathrm{post}}$ a post-smoothing iteration matrix. If we assume
$$\begin{array}{ll}
(i)\hbox{  }\exists\alpha_{\mathrm{post}}>0:\hbox{  } \|V_{n,\mathrm{post}}\mathbf{x}_n\|^{2}_{A_n}\leq \|\mathbf{x}_n\|^{2}_{A_n}-\alpha_{\mathrm{post}}\|\mathbf{x}_n\|^{2}_{{A_n}^{2}},\hbox{ }\forall\mathbf{x}\in\mathbb{C}^{n}, &\hbox{(smoothing property)}\\
(ii)\hbox{  }\exists\beta>0:\hbox{  }\min_{\mathbf{y}\in\mathbb{C}^{n\times n}}\|\mathbf{x}_{n}-P_{m}^{n}\mathbf{y}\|_{2}^{2}\leq\beta\|\mathbf{x}_n\|^{2}_{A_n},\hbox{ }\forall\mathbf{x}\in\mathbb{C}^{n},&\hbox{(approximation property)}
\end{array}$$
Then $\beta>\alpha_{\mathrm{post}}$ and $$\|TGM\|_{A_n}\leq\sqrt{1-\frac{\alpha_{\mathrm{post}}}{\beta}}.$$
\end{Theorem}
The fact that $\alpha_{\mathrm{post}}$ and $\beta$ are independent of $n$ means that, if the hypotheses $(i),(ii)$, are satisfied, then the TGM is not just convergent, but also optimal.
\subsubsection{V-cycle method}
By employing a recursive procedure, the TGM leads to a Multi-Grid Method (MGM): indeed, the standard V-cycle can be expressed in the following way:

\newpage
\[
\begin{tabular}{lll}
     \multicolumn{3}{c} {$\mathbf{x}_s^{\mathrm{out}}=\mathcal{MGM}(s,\mathbf{x}_s^{\mathrm{in}},\mathbf{b}_s)$} \\ %\mathcal{V}_{s,{\rm pre}},\mathcal{V}_{s,{\rm post}})} \\
     \hline \\
     \texttt{if}\quad $s\le s_{\min}$& \texttt{then} \\
     \ \\
     & \fbox{\begin{tabular}{l}$\mathrm{Solve\ }
     A_{s}\mathbf{x}_s^{\mathrm{out}}=\mathbf{b}_{s}$ \hskip 1.4cm \phantom{pA} \end{tabular}} & {Exact solution}\\
     \ \\
     \texttt{else} &  \\
     & \fbox{\begin{tabular}{l} $\mathbf{x}_s^{\mathrm{pre}}=\mathcal{V}_{s,\rm pre}^{\nu_{\rm pre}}
                             (\mathbf{x}_s^{\mathrm{in}},\mathbf{b}_s)\hskip 1.2cm \phantom{pA} $ \end{tabular}}  & {Pre-smoothing iterations}\\
                             \ \\
     &              \fbox{\begin{tabular}{l} $\mathbf{r}_{s}=\mathbf{b}_s-A_s \mathbf{x}_s^{\mathrm{pre}}$ \\
                                             $\mathbf{r}_{s+1}=(P^{m_{s+1}}_{m_s})^* \mathbf{r}_{s}$     \\
%                                            $A_{s+1}=P^{m_{s+1}}_{m_s} A_s (P^{m_{s+1}}_{m_s})^H$  {(precomputing phase)}\\
                                             $\mathbf{y}_{s+1}=\mathcal{MGM}(s+1,\mathbf{0}_{s+1},\mathbf{r}_{s+1})$  \\
                                             $\hat{\mathbf{x}}_s=\mathbf{x}_s^{\mathrm{pre}}+P^{m_{s+1}}_{m_s} \mathbf{y}_{s+1}$\\
     \end{tabular}}   & {Coarse Grid Correction}\\
     \ \\
     &  \fbox{\begin{tabular}{l}  $\mathbf{x}_s^{\mathrm{out}}=\mathcal{V}_{s,\rm post}^{\nu_{\rm post}}(\hat{\mathbf{x}}_s,\mathbf{b}_s) \hskip 1.2cm \phantom{pA} %\quad  x_s^{\mathrm{out}}=x_s^{\mathrm{post}}
     $\end{tabular}}& {Post-smoothing iterations}\\
\end{tabular}
\]

From a computational viewpoint, it is more efficient that the matrices \\$A_{s+1}=P^{s+1}_s A_s
(P^{s+1}_s)^*$ are computed in the so-called
\emph{setup phase} for reducing the related costs.\par
According to the previous setting, the global iteration matrix of the MGM is recursively defined as

\[
\begin{array}{lcl}
MGM_{s_{\min}} &=& O \in \mathbb{C}^{s_{\min} \times s_{\min}}, \\
\\
MGM_s &=& V_{s,\mathrm{post}}^{\nu_\mathrm{post}}
   \left[
   {I^{(s)}}-P^{m_{s+1}}_{m_s}
            \left( {I^{(s+1)}}-MGM_{s+1} \right)A_{s+1}^{-1}
   (P^{m_{s+1}}_{m_s})^{*} A_s \right] V_{s,\mathrm{pre}}^{\nu_\mathrm{pre}}, \\
   \ \\
   && \hfill s=s_{\min}-1,\ldots , 0.\\
\end{array}
\]

%The solution $\mathbf{e}^{(k)}$ of the residual equation is used to improve the approximation  $\mathbf{x}^{(k+1)}= \mathbf{x}^{(k)}+\mathbf{e}^{(k)}$, once we obtain $\mathbf{x}^{(k)}$ some iterative method and
%The residual equation is used after computing the approximation $\mathbf{x}^{(k)}$ of the solution $\overline{\mathbf{x}}$

%This equation is:if we can approximate the solution u using some method, it is straightforward to calculate the residual $r=b-Anu$ and then to improve the approximation, we simply need to solve the equation of the residual and obtain e. This gives us a new approximation, e, and we can use this to compute a better approximation to the solution,u.

%This equation is easy to use because once we have an estimate of u, it is straightforward to compute the residual $r = b-Anu$, and in order to enhance the approximation, we may use the error definition and solve for e.

%TGM et MGM il y a 2 type geometric et algebric et on focus sur l'algebrique

\end{itemize}
\chapter{Spectral properties of matrix-sequences described by a function}
In this chapter, we begin by introducing the basic tools for spectral analysis of matrix-sequences, by presenting the notion
 of distribution both in the sense of the eigenvalues and singular values, clustering, essential range and spectral attraction and by providing useful characterizations of zero-distributed sequences. Then we will go through Toeplitz and Circulant matrices and their spectral characteristics both in the scalar and the block setting. Finally, we will explore the *-algebra of Generalized Locally Toeplitz matrix sequences.

%\section{Spectral distribution and spectral properties of matrix-sequences}\label{ssez:matrix-seq}
%In this section we introduce the fundamental tools for the spectral analysis of matrix-sequences, by presenting the notion of distribution, clustering both in the sense of the eigenvalues and singular values and then we focus
\section{Asymptotic distribution of matrix-sequences}\label{ssez:matrix-seq}

\begin{Definition}[\textbf{Matrix-sequence distributed as complex-valued function}]\label{def-distribution realvalued}
Let $\{A_n\}_n$ be a sequence of matrices, with $A_n$ of size
$d_n$, with $d_k<d_{k+1}$ for every $k\in\mathbb{N}$ and let $f:D\subset\mathbb R^t\to\mathbb{C}$ be
a measurable function defined on a set $D$ with
$0<\mu_t(D)<\infty$.
\begin{itemize}
    \item We say that $\{A_n\}_n$ has a (asymptotic) singular value distribution described by $f$, and we write $\{A_n\}_n\sim_\sigma f$, if
    \begin{equation}\label{distribution:sv realvalued}
     \lim_{n\to\infty}\frac1{d_n}\sum_{i=1}^{d_n}F(\sigma_i(A_n))=\frac1{\mu_t(D)}\int_D F(|f(\mathbf x)|){\rm d}\mathbf x,\qquad\forall\,F\in C_c(\mathbb R).
    \end{equation}
    \item We say that $\{A_n\}_n$ has a (asymptotic) spectral (or eigenvalue) distribution described by $f$, and we write $\{A_n\}_n\sim_\lambda f$, if
    \begin{equation}\label{distribution:eig realvalued}
     \lim_{n\to\infty}\frac1{d_n}\sum_{i=1}^{d_n}F(\lambda_i(A_n))=\frac1{\mu_t(D)}\int_D F(f(\mathbf x)){\rm d}\mathbf x,\qquad\forall\,F\in C_c(\mathbb C).
    \end{equation}
\end{itemize}
If $\{A_n\}_n$ has both a singular value and an eigenvalue distribution described by $f$, then we write $\{A_n\}_n\sim_{\sigma,\lambda}f$.
\end{Definition}
Whenever we write a relation such as $\{A_n\}_n\sim_\lambda f$, it is understood that $f$ is as in Definition~\ref{def-distribution realvalued}; that is, $f$ is a measurable function defined on a subset $D$ of some $\mathbb R^t$ with
$0<\mu_t(D)<\infty$ and taking values in $\mathbb C$.

\begin{Definition}[\textbf{Matrix-sequence distributed as matrix-valued function}]\label{def-distribution}
Let $\{A_n\}_n$ be a sequence of matrices, with $A_n$ of size
$d_n$, and let \\$\mathbf{f}:D\subset\mathbb R^t\to\mathbb{C}^{r\times r}$ be
a measurable function defined on a set $D$ with\\
$0<\mu_t(D)<\infty$.
\begin{itemize}
    \item We say that $\{A_n\}_n$ has a (asymptotic) singular value distribution described by $\mathbf{f}$, and we write $\{A_n\}_n\sim_\sigma \mathbf{f}$, if
    \begin{equation}\label{distribution:sv-sv}
     \lim_{n\to\infty}\frac1{d_n}\sum_{i=1}^{d_n}F(\sigma_i(A_n))=\frac1{\mu_t(D)}\int_D\frac{\sum_{i=1}^{r}F(\sigma_i(\mathbf{f}(\mathbf x)))}{r}{\rm d}\mathbf x,\qquad\forall\,F\in C_c(\mathbb R).
    \end{equation}
    \item We say that $\{A_n\}_n$ has a (asymptotic) spectral (or eigenvalue) distribution described by $\mathbf{f}$, and we write $\{A_n\}_n\sim_\lambda \mathbf{f}$, if
    \begin{equation}\label{distribution:sv-eig}
     \lim_{n\to\infty}\frac1{d_n}\sum_{i=1}^{d_n}F(\lambda_i(A_n))=\frac1{\mu_t(D)}\int_D\frac{\sum_{i=1}^{r}F(\lambda_i(\mathbf{f}(\mathbf x)))}{r}{\rm d}\mathbf x,\qquad\forall\,F\in C_c(\mathbb C).
    \end{equation}
\end{itemize}
If $\{A_n\}_n$ has both a singular value and an eigenvalue distribution described by $\mathbf{f}$, we write $\{A_n\}_n\sim_{\sigma,\lambda}\mathbf{f}$.
\end{Definition}
We note that Definition~\ref{def-distribution} is well-posed because the functions
\[
\mathbf x\mapsto\sum_{i=1}^{r}F(\sigma_i(\mathbf f(\mathbf x)))
%\]
\textrm{\ \  and \  \ }
%\[
\mathbf x\mapsto\sum_{i=1}^{r}F(\lambda_i(\mathbf{f}(\mathbf x)))
\]
 are measurable.\\
  The informal meaning behind the spectral
distribution \eqref{distribution:sv-eig} is the following: if $\mathbf{f}$
is continuous, then a suitable ordering of the eigenvalues
$\{\lambda_j(A_n)\}_{j=1,\ldots,d_n}$, assigned in correspondence
with an equispaced grid on $D$, reconstructs approximately the $r$
surfaces  $\mathbf x$ $\mapsto\lambda_i(\mathbf{f}(\mathbf x)),\
i=1,\ldots,r$.\\For instance, if $t=1$, $d_n=nr$, and $D=[a,b]$,
then the eigenvalues of $A_n$ are approximately equal to
$\lambda_i(\mathbf{f}(a+j(b-a)/n))$, $j=1,\ldots,n,\ i=1,\ldots,r$; if $t=2$, $d_n=n^2 r$, and $D=[a_1,b_1]\times [a_2,b_2]$, then the
eigenvalues of $A_n$ are approximately equal to
$\lambda_i(\mathbf{f}(a_1+j_1(b_1-a_1)/n,a_2+j_2(b_2-a_2)/n)),\
j_1,j_2=1,\ldots,n,\ i=1,\ldots,r$ (and so on for $t\ge3$). This
type of information is useful in engineering applications
\cite{tom-paper}, e.g. for the computation of the relevant
vibrations, and in the analysis of the (asymptotic) convergence
speed of iterative solvers for large linear systems or for
improving the convergence rate by e.g. the design of appropriate
preconditioners \cite{BK,BeSe}.\par
The next theorem gives useful tools for computing the spectral
distribution of sequences formed by Hermitian matrices. For the
related proof, we refer the reader to
\cite[Theorem~4.3]{opiccolo}.
\begin{Theorem}\label{extradimensional}
Let $\{A_n\}_n$ be a sequence of matrices, with $A_n$ Hermitian of size $d_n$, and let $\{P_n\}_n$
be a sequence such that $P_n\in\mathbb C^{d_n\times\delta_n}$, $P_n^*P_n=I_{\delta_n}$, $\delta_n\le d_n$
and $\delta_n/d_n\to1$ as $n\to\infty$. Then, $\{A_n\}_n\sim_{\sigma,\lambda}\mathbf{f}$ if and only if $\{P_n^*A_nP_n\}_n\sim_{\sigma,\lambda}\mathbf{f}$.
\end{Theorem}
  \begin{Definition}[\textbf{Clustering of a matrix sequence} ]\label{def-cluster}

Given $\{A_n\}_n$  a sequence of matrices of increasing size $d_n$ and let $S\subseteq\mathbb C$ be a non-empty closed subset of $\mathbb C$.
 \begin{itemize}
\item $\{A_n\}_n$ is {\em strongly clustered} at $S$
in the sense of the eigenvalues if, every $\epsilon>0$ there exists a constant $q_\epsilon(n,S)$ independent from $n$ such that the number of eigenvalues of $A_n$ outside $B(S,\epsilon)$
is bounded by $q_\epsilon$ . In symbols,
$$q_\epsilon(n,S):=\#\{j\in\{1,\ldots,d_n\}: \lambda_j(A_n)\notin B(S,\epsilon)\}=O(1),\quad\mbox{as $n\to\infty$.}$$
\item  $\{A_n\}_n$ is {\em weakly clustered} at $S$ if, for all $\epsilon>0$,
$$q_\epsilon(n,S)=o(d_n), \quad \mbox{as $n\to\infty.$}$$
\end{itemize}
If $\{A_n\}_n$ is strongly or weakly clustered at $S$ and $S$ is not connected, then the connected components of $S$ are called sub-clusters.
  \end{Definition}

\begin{Definition}[\textbf{Essential range of complex-valued function}]
Let \\$f:D\subseteq\mathbb R^t\to\mathbb C$ be a measurable function, with $D$ a measurable set of finite measure.
The set $\displaystyle \mathcal{ER}(f):=\{z\in\mathbb C:\,\mu_t(\{f\in B(z,\epsilon)\})>0\mbox{ for all $ \epsilon>0$}\},$ is {\em the essential range of $f$}  where $\{f\in B(z,\epsilon)\}:=\{x\in D:\,f(x)\in B(z,\epsilon)\} $ .
\end{Definition}
$\mathcal{ER}(f)$ is always closed; moreover, if $f$ is continuous
and $D$ is contained in the closure of its interior, then
$\mathcal{ER}(f)$ coincides with the closure of the image of $f$.

  \begin{Definition}[\textbf{Essential range of matrix-valued function}]
  Let \\$\mathbf{f}:D\subseteq\mathbb R^t\to\mathbb C^{r\times r}$ be a measurable function, with $D$ a measurable set of finite measure, the essential range of $\mathbf{f}$ is defined as the union of the essential ranges of the eigenvalue functions $\lambda_i(\mathbf{f}),\ i=1,\ldots, r$: $\mathcal{ER}(\mathbf{f}):=\bigcup_{i=1}^r\mathcal{ER}(\lambda_i(\mathbf{f}))$.
  \end{Definition}
  \begin{Theorem}
  If $\{A_n\}_n\sim_\lambda \mathbf{f}$ (with $\{A_n\}_n,\  \mathbf{f}$ as in Definition \ref{def-distribution}), then,
by \cite[Theorem 4.2]{gol-serra}, $\{A_n\}_n$ is weakly clustered at the essential range of $\mathbf{f}$, defined as the union of the essential ranges of the eigenvalue functions $\lambda_i(\mathbf{f}),\ i=1,\ldots, r$: $\mathcal{ER}(\mathbf{f}):=\bigcup_{i=1}^s\mathcal{ER}(\lambda_i(\mathbf{f}))$.
  \end{Theorem}

  \begin{Definition}[\textbf{Spectral attraction}]
  Let $\{A_n\}_n$ be a sequence of matrices of size $d_n$ tending to infinity, let $z\in\mathbb{C}$ and let $\lambda_1,\lambda_2,\ldots,\lambda_{d_n}$ the eigenvalues of $A_n$ ordered  according to their distance from $z$, i.e. $$|\lambda_1-z|\leq|\lambda_2-z|\leq\ldots\leq|\lambda_{d_n}-z|,$$
  We say that {\em $z$ strongly attracts the spectrum $\Lambda(A_n)$}  with infinite order if, for each fixed $k$, $$\lim_{n\longrightarrow\infty}|\lambda_k-z|=0.$$

  \end{Definition}

\begin{Theorem} \cite[Theorem 4.2]{gol-serra}
  If $\{A_n\}_n\sim_\lambda \mathbf{f}$ (with $\{A_n\}_n,\  \mathbf{f}$ as in Definition \ref{def-distribution}), then each point $z\in\mathcal{ER}(\mathbf{f})$ strongly attracts $\Lambda(A_n)$ with infinite order.
\end{Theorem}

Now, we provide a classical theorem that is the Gershgorin's {theorem}; see \cite{bhatia1997matrix} which provide the spectral localization of a matrix.
  \begin{Theorem}
Let $A\in\mathbb{C}^{n\times n}$, we define the Gershgorin disk as
$$B(a_{jj},z_j)=\{z\in\mathbb{C}:|a_{jj}-z|\leq z_j\}\hbox{, where }z_j=\sum_{k\neq j}|a_{jk}|.$$
Then $$\displaystyle\Lambda(A)\subset \cup_{j=1}^{n}B(a_{jj},z_j).$$
\end{Theorem}

  The next result is known as {Cauchy's} interlacing theorem for Hermitian matrices, which is useful to get information about the eigenvalues of a Hermitian matrix if we know those of any principal submatrix. We refer to \cite[Corollary III.1.5]{bhatia1997matrix} and \cite{hwang2004cauchy} for the proof.

  \begin{Theorem}
  Let $A\in\mathbb{C}^{n\times n}$ be a Hermitian matrix and let $B\in\mathbb{C}^{m\times m}$ its principal submatrix. Let $\lambda_1\leq\lambda_2\leq\ldots\leq\lambda_n$ and $\mu_1\leq\mu_2\leq\ldots\leq\mu_{m}$ be the eigenvalues of $A$ and $B$ respectively arranged in non-decreasing order, then $$\lambda_k\leq\mu_k\leq\lambda_{k+n-m},\mbox{ for }k=1,\ldots,m.$$
  Morever, if $m=n-1$,
  $$\lambda_1\leq\mu_1\leq\lambda_2\leq\mu_2\leq\ldots\leq\mu_{n-1}\leq\lambda_n.$$
 \end{Theorem}

To end this section, we give a useful result for computing the spectral distribution of a sequence of perturbed Hermitian matrices of the form $\{A_n+B_n\}_n$, for more details, see \cite[Theorem 3.3]{garoni2015tools}.

 \begin{Theorem}
 Let $\{A_n\},\{B_n\}$ be a sequence of matrices of increasing size $d_n$, if we assume that the following conditions are satisfied.
 \begin{itemize}
 \item Each $A_n$ is Hermitian and $\{A_n\}\sim_{\lambda}\mathbf{f}$, where $\mathbf{f}:D\subseteq\mathbb R^t\to\mathbb C^{r\times r}$.
 \item There exists a constant $C>0$, for all $n$, $\|A_n\|,\|B_n\|\leq C$.
 \item $\||B_n\||_{1}=o(d_n)$.

 \end{itemize}
 Then $\{C_n\}\sim_{\lambda}\mathbf{f}$ with $C_n=A_n+B_n$, and   $\displaystyle\lambda_{i}(\mathbf{f})\in\overline{\cup_n\Lambda(A_n)}\subseteq[-C,C]$ almost everywhere, for $i=1,\ldots,r$.
 \end{Theorem}
 \noindent\section{Zero-Distributed matrix-sequences}
A sequence of matrices $\{Z_n\}_n$ such that $\{Z_n\}_n\sim_\sigma0$ is referred to as a zero-distributed sequence.
Note that, for any $r\ge1$, $\{Z_n\}_n\sim_\sigma0$ is equivalent to $\{Z_n\}_n\sim_\sigma O_r$. Proposition~\ref{0cs} provides an important
characterization of zero-distributed sequences, together with a useful sufficient condition for detecting such sequences. Throughout this paper we use the natural convention $1/\infty=0$.

\begin{Proposition}\label{0cs}
Let $\{Z_n\}_n$ be a sequence of matrices, with $Z_n$ of size $d_n$ with $d_k<d_{k+1}$, $\forall k\in\mathbb{N}$. Then
\begin{itemize}
   \item  $\{Z_n\}_n$ is zero-distributed if and only if $Z_n=R_n+N_n$ with ${\rm rank}(R_n)/d_n\to0$ and $\|N_n\|_{2}\to0$ as $n\to\infty$.
    \item
     $\{Z_n\}_n$ is zero-distributed if there exists $p\in[1,\infty]$ such that \\$\|Z_n\|_p/(d_n)^{1/p}\to0$ as $n\to\infty$.
\end{itemize}
\end{Proposition}
 %\section{Toeplitz matrices and their properties}
 \section{Toeplitz and Circulant matrices}\label{ToepCirc}
 Toeplitz matrices are a significant and current topic, which has been used for well over a century, introduced in Toeplitz's initial works \cite{toeplitz1911theorie} and have been the subject of numerous works \cite{bottcher2000toeplitz,BoSi, book:25026}. These types of matrices arise in several applications such as the discretization of systems of constant coefficient-differential equations \cite{glt-2}, Markov chains \cite{bini2005numerical}, in the reconstruction of signals and images with missing data \cite{del2014symbol,hansen2006deblurring}, Riccati equations \cite{bini2011numerical}.

In this part, we describe Toeplitz and Circulant matrices, giving both their description and some of their features. Subsections \ref{scalarToep}, \ref{scalarCirc} describe the case of the scalar setting, while Subsections \ref{bloctoep}, \ref{bloccirc} focus on the most general setting where Toeplitz and Circulant matrices are extended to a Multilevel Block form.
%discuss the
%First, we'll focus on the scalar setting in the Subsections \ref{scalarToep}, \ref{scalarCirc}.
%Then we will extend the definitions to the multilevel Block form in the Subsections \ref{bloctoep},\ref{bloccirc}.
%This part is dedicated to the issue of Toeplitz matrices, which appears several times throughout the thesis, and is covered in books [22, 24,25, 67].

 \subsection{Scalar Toeplitz matrices}\label{scalarToep}

% A matrix $A_n\in\mathbb{C}^{n\times n}$ is called a Toeplitz matrix if \\
 Any matrix  $A_n\in\mathbb{C}^{n\times n}$ having the form $$A_n=[a_{i-j}]_{i,j=1}^{n}=\begin{bmatrix}
  a_0 & a_{-1}&a_{-2}&\ldots &\ldots& a_{-(n-1)} \\
  a_1 &\ddots &\ddots &\ddots & & \vdots \\
  a_2&\ddots &\ddots &\ddots &\ddots & \vdots \\
  \vdots&\ddots &\ddots &\ddots &\ddots & a_{-2} \\
  \vdots& &\ddots &\ddots &\ddots & a_{-1} \\
  a_{n-1}&\ldots &\ldots &a_2 &a_1 & a_{0} \\
\end{bmatrix},$$ i.e., whose elements are constant along each diagonal, is called Toeplitz. This matrix is completely characterized by the elements of the first row and column \\$[a_{n-1},\ldots,a_2,a_1,a_0,a_{-1},a_{-2},\ldots,a_{-(n-1)}]$. \\
The entries of $A_n$ can be generated from the Fourier coefficients of a $2\pi-$periodic function $f\in L^1([-\pi,\pi])$
 $$a_{k}=\frac{1}{2\pi}\int_{-\pi}^{\pi}f(\theta)\textrm{e}^{-\hat{\imath}k\theta}d\theta,\hspace{1cm}k\in\mathbb{Z},$$
where $f$ is called the generating function, and it is given by Fourier series $$f(\theta)=\sum_{k=-\infty}^{\infty}a_k\textrm{e}^{\hat{\imath}k\theta}.$$ In this case $A_n$ is the Toeplitz matrix associated with $f$, expressed through the Toeplitz operator, i.e., $A_n=T_n(f)$.\\
It follows that the properties of the matrix $T_n(f)$ and its eigenvalues can be defined from the properties of the generating function $f$, see \cite{glt-book-I,S-LAA-1998}. The most relevant features are listed below:
\begin{enumerate}

\item If $f$ is real-valued almost everywhere $\left(f(\theta)=\overline{f(\theta)}\right)$, then $T_n(f)$ is Hermitian $(a_{-k}=\overline{a_{k}},$ $k\in\mathbb{Z})$ for all $n$.
\item If $f$ is even $\left(f(\theta)=f(-\theta)\right)$, then $T_n(f)$ is (complex) symmetric $(a_k=a_{-k}, k\in\mathbb{Z})$ for all $n$.
\item If $f$ is positive almost everywhere $\left(f\geq0\right)$, then $T_n(f)$ is Hermitian Positive Semi-Definite $(T_n(f)\geq0)$ for all $n$.
\item Let $f\in L^{1}([-\pi,\pi])$ real-valued almost everywhere, and let $m_f=\textrm{ess}\inf_{\theta\in[-\pi,\pi]}f(\theta)$,  $M_f=\textrm{ess}\sup_{\theta\in[-\pi,\pi]}f(\theta)$.
\begin{itemize}
\item If $m_f=M_f$, then $T_n(f)=m_fI_n$ and $\Lambda(T_n(f))=\{m_f\}$.
\item If $m_f<M_f$, then $\Lambda(T_n(f))\subset(m_f,M_f)$.

\end{itemize}
\item Let $\lambda_{1}(T_{n}(f))\leq\lambda_{2}(T_{n}(f))\leq\cdots \leq\lambda_{n}(T_{n}(f))$ be the eigenvalues of $T_{n}(f)$ ordered in non-decreasing order. Then for all $n$ fixed, for each $s\geq1$ independent of $n$, we have
$$\lim_{n\longrightarrow\infty}\lambda_{s}(T_{n}(f))={m_f},\ \ \ \ \ \lim_{n\longrightarrow\infty}\lambda_{n-s+1}(T_{n}(f))={M_f}.$$
\item Let $f\in L^{1}([-\pi,\pi])$ be real-valued almost everywhere, and let $m_f=\textrm{ess}\inf_{\theta\in[-\pi,\pi]}f(\theta)$,  $M_f=\textrm{ess}\sup_{\theta\in[-\pi,\pi]}f(\theta)$.
 Assume that $f(\theta)-m_f$ has a finite number of zeros of order $\alpha_1,\alpha_2,\ldots,\alpha_k>0$ and assume that $M_f-f(\theta)$ has a finite number of zeros of order $\beta_1,\beta_2,\ldots,\beta_q>0$, then
 \begin{itemize}
 \item For each fixed $j$ with respect to $n$
 $$\lambda_j(T_n(f))-m_f\sim\frac{1}{n^\alpha}, \qquad\hbox{ with }\alpha=\max_{1\leq s\leq k}\alpha_s.$$
 \item For each fixed $j$ with respect to $n$
 $$M_f-\lambda_j(T_n(f))\sim\frac{1}{n^\beta}, \qquad\hbox{ with }\beta=\max_{1\leq s\leq q}\beta_s.$$
 \end{itemize}
\end{enumerate}
 \subsection{Scalar Circulant matrices }\label{scalarCirc}
 Now, we provide a particular case of Toeplitz matrices which have the following form
   $$A_n=[a_{(i-j){\rm mod\ }n}]_{i,j=1}^{n}=\begin{bmatrix}
  a_0 & a_{n-1}&a_{n-2}&\ldots &\ldots& a_1 \\
  a_1 &\ddots &\ddots &\ddots & & \vdots \\
  a_2&\ddots &\ddots &\ddots &\ddots & \vdots \\
  \vdots&\ddots &\ddots &\ddots &\ddots & a_{n-2} \\
  \vdots& &\ddots &\ddots &\ddots & a_{n-1} \\
  a_{n-1}&\ldots &\ldots &a_2 &a_1 & a_{0} \\
\end{bmatrix}.$$
  This kind of matrix is called a Circulant matrix and it can be written as the following expression
  \begin{equation}\label{zmatrix}
  A_n=\sum_{s=0}^{n-1}a_sZ_{n}^{(s)},
  \end{equation} where $Z_{n}^{(s)}=\left\{\begin{array}{ll}
  1&\hbox{if }(i-j)\ {\rm mod\ } n=s,\\
  0&\mbox{otherwise},
  \end{array}\right.$ and $Z_{n}^{(s)}=Z_{n}^{s\rm{mod} n}$ for every $s\in\mathbb{Z}$.

 The Circulant matrices are diagonalizable by the unitary Fourier matrix $F_n$ given by $$F_n=\left[\frac{1}{\sqrt{n}}e^{-\hat{\imath}\frac{2\pi ij}{n}}\right]_{i,j=1}^{n}.$$
  Therefore, the algebra of Circulant matrices is defined as $$\mathcal{C}_{n}=\left\{A_n\in\mathbb{C}^{n\times n}|A_n=F_nD_nF_n^*,\  D_n=\underset{j=0,\ldots,n-1}{\textrm{diag}} f\left(\frac{2\pi j}{n}\right)\right\},$$ where $f(x)=p\circ e^{\hat{\imath}x}$, $p\in\mathcal{P}_{n-1}$  $(\mathcal{P}_{n}$ denote the space of polynomials of degree less or equal to $n$), and $p(z)=\sum_{s=0}^{n-1}a_{s}z^{s}$, that is the spectrum of every Circulant is uniquely determined by its first column.
%\\The generating function of a Toeplitz matrix is given by Fourier series $$f(\theta)=\sum_{j=-\infty}^{\infty}a_je^{\hat{\imath}j\theta}.$$
 \subsection{Multilevel block-Toeplitz matrices}\label{bloctoep}
 This subsection is devoted to the concept of multilevel block-Toeplitz matrices, these matrices have the same structure as the scalar Toeplitz matrices, but this time the entries $a_k$ are matrices instead of scalars.% where its entries are matrices
%Let $A_n\in\mathbb{C}^{n\times n}$
\\A matrix $\mathbf{A}_{n}\in\mathbb{C}^{rn\times rn}$ of the following structure
%A matrix $\mathbf{A}_{n}$ of the form
$$\mathbf{A}_{n}=[A_{i-j}]_{i,j=1}^{n}=\begin{bmatrix}
  A_0 & A_{-1}&A_{-2}&\ldots &\ldots& A_{-(n-1)} \\
  A_1 &\ddots &\ddots &\ddots & & \vdots \\
  A_2&\ddots &\ddots &\ddots &\ddots & \vdots \\
  \vdots&\ddots &\ddots &\ddots &\ddots & A_{-2} \\
  \vdots& &\ddots &\ddots &\ddots & A_{-1} \\
 A_{n-1}&\ldots &\ldots &A_2 &A_1 & A_{0} \\
\end{bmatrix}$$ is a block-Toeplitz matrix, where $A_k\in\mathbb{C}^{r\times r},\hbox{ }k\in\mathbb{Z}$. \\
Notice that the unilevel block Toeplitz matrix $\mathbf{A}_{n}$ can be written as   \begin{equation}\label{jmatrix}
{\mathbf{A}}_{n}=\sum_{k=-(n-1)}^{n-1}J_{n}^{(k)}\otimes A_k,
\end{equation}
where $J_m^{(l)}$ is the matrix of order $m$ whose $(i,j)$
entry equals $1$ if $i-j=l$ and zero otherwise. % %where are the fourier coefficients of ,t of finl1
\begin{Definition}\label{def-block}
The Fourier coefficients of a matrix-valued function $\mathbf{f}\in L^1([-\pi,\pi],r)$ are given by $$\hat f_k:=\frac{1}{2\pi}\int_{(-\pi,\pi)}\mathbf{f}(\theta){\rm e}^{-\hat{\imath}k \theta}d\theta\in\mathbb{C}^{r\times r},\qquad  \, k\in\mathbb Z,$$
where the integrals are computed componentwise. Then, we define the block-Toeplitz matrix $T_n(\mathbf{f})$ associated with $\mathbf{f}$ {as} the $rn \times rn$  matrix given by
$$T_n(\mathbf{f})=\sum_{|k|<n}J_{n}^{(k)}\otimes\hat f_k,$$
with $J_n^{(k)}$ is the matrix as defined in \eqref{jmatrix}. The set $\{T_{ n}({\mathbf{f}})\}_{n}$ is
called the family of block-Toeplitz matrices generated by ${\mathbf{f}}$.
\end{Definition}
The next definition is an extension of the definition of the Toeplitz matrix in a more general case, when the entries elements are generated by a $t$-variate matrix-valued function, and the resulting matrix is multilevel block Toeplitz matrix.

%$${-\hati j\theta}d\theta$$
%If $\mathbf{f}\in L^1([-\pi,\pi],r)$, then the corresponding fourier coefficients are defined by $$\hat f_j:=\frac{1}{2\pi}$$
%$$\hat f_j:=\frac{1}{2\pi}\int_{(-\pi,\pi)}f(\theta){\rm e}^{- \hati j \theta} d\theta\in\mathcal M_k,
%  \qquad  \, j\in\mathbb Z.$$
\begin{Definition}
Given $\mathbf{n}\in \mathbb{N}^t$, a matrix of the form
\[
[A_{\mathbf{i}-\mathbf{j}}]_{\mathbf{i},\mathbf{j}=\mathbf{e}}^{\mathbf{n}} \in \mathbb{C}^{N(\mathbf{n})r \times N(\mathbf{n})r}
\]
with blocks $A_\mathbf{k}\in \mathbb{C}^{r\times r}$, $\mathbf{k} =
-(\mathbf{n}-\mathbf{e}), \ldots, \mathbf{n}-\mathbf{e}$, is
called a multilevel block Toeplitz matrix, or, more precisely, a
$t$-level $r$-block Toeplitz matrix.\\ Given a matrix-valued
function $\mathbf{f}\in L^1([-\pi, \pi]^t,r)$, we denote its Fourier coefficients by
\begin{equation}\label{eq:fourier_coefficients}
\hat{\mathbf{f}}_\mathbf{k} =\frac{1}{(2\pi)^t} \int_{[-\pi,\pi]^t}
\mathbf{f}(\boldsymbol{\theta}){\rm e}^{-\hat{\imath} (\mathbf{k}, \boldsymbol{\theta})}
d\boldsymbol{\theta} \in \mathbb{C}^{r\times r}, \ \ \ \mathbf{k}
\in \mathbb{Z}^t,
\end{equation}
where the integrals are computed componentwise and $(\mathbf{k},
\boldsymbol{\theta}) = k_1\theta_1 + \ldots + k_t\theta_t$. For
every $\mathbf{n} \in \mathbb{N}^t$, the $\mathbf{n}$-th Toeplitz
matrix associated with $\mathbf{f}$ is defined as
\begin{equation}\label{eq:toeplitz}
T_\mathbf{n}(\mathbf{f}) := [\hat{\mathbf{f}}_{\mathbf{i}-\mathbf{j}}]_{\mathbf{i},\mathbf{j}=\mathbf{e}}^{\mathbf{n}}
\end{equation}
or, equivalently, as
\begin{equation}\label{eq:toeplitz_kron}
T_\mathbf{n}(\mathbf{f}) = \sum_{|j_1|<n_1} \ldots \sum_{|j_t|<n_t} [J_{n_1}^{(j_1)} \otimes \ldots \otimes J_{n_t}^{(j_t)}]
\otimes \hat{\mathbf{f}}_{(j_1,\ldots, j_t)}.
\end{equation}

\end{Definition}
Analogously to the scalar Toeplitz matrix case, some properties of the matrix $T_{\mathbf{n}}(\mathbf{f})$ and its eigenvalues can be extracted from the properties of $\mathbf{f}$ as follows, see \cite{Ba2-ETNA,Ba1-ETNA,garoni2014structured,glt-book-II} %we provide some properties on the generating function $\mathbf{f}$ which \\
%Let $\mathbf{f}\in L^{1}([-\pi,\pi]^{t})$ matrix-valued function
\begin{enumerate}
\item If the matrix $\mathbf{f}\in L^{1}([-\pi,\pi]^{t},r)$ is Hermitian ($\mathbf{f}(\boldsymbol{\theta})=\mathbf{f}^*(\boldsymbol{\theta})$) almost everywhere, then all the matrices $T_n(\mathbf{f})$ are Hermitian.
\item If the matrix $\mathbf{f}\in L^{1}([-\pi,\pi]^{t},r)$ is Hermitian Positive Semi-Definite \\($\mathbf{f}(\boldsymbol{\theta})\geq0$) almost everywhere, then all the matrices $T_n(\mathbf{f})$ are Hermitian Positive Semi-Definite.
\item Let $\mathbf{f}\in L^1([-\pi,\pi]^t,r)$ be a Hermitian matrix and let $\mathbf{m_f}=\underset{\boldsymbol{\theta}\in([-\pi,\pi])^{t}}{\textrm{ess}\inf}\lambda_{\min}(\mathbf{f}(\boldsymbol{\theta}))$, $\mathbf{M_f}=\underset{\boldsymbol{\theta}\in([-\pi,\pi])^{t}}{\textrm{ess}\sup}\lambda_{\max}(\mathbf{f}(\boldsymbol{\theta}))$ then
\begin{itemize}
\item If $\mathbf{m_f}=\mathbf{M_f}$ then $T_{\mathbf{n}}(\mathbf{f})=\mathbf{m_f}I_{N(\mathbf{n})r}$ and $\Lambda(T_{\mathbf{n}}(\mathbf{f}))=\{\mathbf{m_f}\}$.
\item If $\mathbf{m_f}<\underset{\boldsymbol{\theta}\in([-\pi,\pi])^{t}}{\textrm{ess}\sup}\lambda_{\min}(\mathbf{f}(\boldsymbol{\theta}))\leq \mathbf{M_f} $ then $\Lambda(\mathbf{f})\subset(\mathbf{m_f},\mathbf{M_f}]$.
\item  If $\mathbf{m_f}\leq\underset{\boldsymbol{\theta}\in([-\pi,\pi])^{t}}{\textrm{ess}\inf}\lambda_{\max}(\mathbf{f}(\boldsymbol{\theta}))<\mathbf{M_f}$ then $\Lambda(\mathbf{f})\subset[\mathbf{m_f},\mathbf{M_f})$.
\end{itemize}

\item Let $\lambda_{1}(T_{\mathbf{n}}(\mathbf{f}))\leq\lambda_{2}(T_{\mathbf{n}}(\mathbf{f}))\leq\cdots \leq\lambda_{N(\mathbf{n})r}(T_{\mathbf{n}}(\mathbf{f}))$ be the eigenvalues of $T_{\mathbf{n}}(\mathbf{f})$ ordered in non-decreasing order. Then for each fixed $s\geq1$, independent of $\mathbf{n}$, we have
$$\lim_{\mathbf{n}\longrightarrow\infty}\lambda_{s}(T_{\mathbf{n}}(\mathbf{f}))=\mathbf{m_f},\ \ \ \ \ \lim_{\mathbf{n}\longrightarrow\infty}\lambda_{N(\mathbf{n})r-s+1}(T_{\mathbf{n}}(\mathbf{f}))=\mathbf{M_f}.$$
\item  Let $\mathbf{f}\in L^{1}([-\pi,\pi]^t,r)$ be a Hermitian matrix-valued and let $\mathbf{m_f}=\underset{\boldsymbol{\theta}\in([-\pi,\pi])^{t}}{\textrm{ess}\inf}\lambda_{\min}(\mathbf{f}(\boldsymbol{\theta}))$, $\mathbf{M_f}=\underset{\boldsymbol{\theta}\in([-\pi,\pi])^{t}}{\textrm{ess}\sup}\lambda_{\max}(\mathbf{f}(\boldsymbol{\theta}))$.
 Assume that $\lambda_{\min}(\mathbf{f}(\boldsymbol{\theta}))-\mathbf{m_f}$ has a finite number of zeros of order $\alpha_1,\alpha_2,\ldots,\alpha_k>0$ , then
 \begin{itemize}
 \item For each fixed $j$ with respect to $\mathbf{n}$ and assuming that $n_1\sim n_2\sim\ldots\sim n_t$
 $$\lambda_j(T_{\mathbf{n}}(\mathbf{f}))-\mathbf{m_f}\sim\frac{1}{[N(\mathbf{n})]^{\alpha/t}}, \qquad\hbox{ with }\alpha=\max_{1\leq s\leq k}\alpha_s.$$

 \end{itemize}	
\end{enumerate}
\begin{Remark}
From Item 2 we conclude that if $\mathbf{f}_{1},\mathbf{f}_{2}\in L^{1}([-\pi,\pi]^t,r)$ and $\mathbf{f}_{1}(\boldsymbol{\theta})\geq \mathbf{f}_{2}(\boldsymbol{\theta})$ almost everywhere, then $T_n(\mathbf{f}_{1})\geq T_n(\mathbf{f}_{2})$.
\end{Remark}

\subsection{Multilevel block Circulant matrices}\label{bloccirc}

If $A_{\mathbf{k}}\in\mathbb{C}^{r\times r}$ then, the matrix having the following form $$\mathbf{A}=\left[A_{(\mathbf{i}-\mathbf{j}){\rm mod\ }\mathbf{n}}\right]_{\mathbf{i},\mathbf{j}=\mathbf{e}}^{\mathbf{n}}\in\mathbb{C}^{rN(\mathbf{n})\times rN(\mathbf{n})}$$
%=\sum_{k=0}^{n-1}Z_{m}^{k}\otimes A_k$$
is called multilevel Circulant matrix, where $Z_{n}^{(j)}$ is defined as in \eqref{zmatrix}.\\
%A multilevel block Circulant matrix has the following
The $\mathbf{n}$-th Circulant  matrix associated with $\mathbf{f}:[-\pi,\pi]^{t}\rightarrow\mathbb{C}^{r\times r}$ is defined as
\begin{equation*}\label{eq:circ}
\mathcal{C}_{\mathbf{n}}(\mathbf{f}) = \sum_{|j_1|<n_1} \ldots \sum_{|j_t|<n_t} [Z_{n_1}^{(j_1)} \otimes \ldots \otimes Z_{n_t}^{(j_t)}]
\otimes \hat{\mathbf{f}}_{(j_1,\ldots, j_t)}.
\end{equation*}
 \begin{Theorem}\cite{garoni2014structured}
The Circulant matrix  $\mathbf{A}$ can be decomposed as follows
$$\mathbf{A}=(F_{\boldsymbol{n}}\otimes I_{\mathbf{r}})\underset{\mathbf{j=0,\ldots,n-1}}{\rm diag}{\mathbf{g}_{\mathbf{n}}\left(\frac{2\pi\mathbf{j}}{\mathbf{n}}\right)}(F_{\boldsymbol{n}}\otimes I_{\mathbf{r}})^*$$
where $${\mathbf{g}}_{\mathbf{n}}(\boldsymbol{\theta})=\sum_{|k_1|<n_1}\sum_{|k_2|<n_2}\cdots\sum_{|k_t|<n_t}\hat{\mathbf{f}}_{\mathbf{k}}{\rm e}^{\imath(\mathbf{k},\boldsymbol{\theta})}                             \hbox{ and}\quad F_n=\frac{1}{\sqrt{N(\mathbf{n})}}\left({\rm e}^{\imath(\mathbf{k},\frac{2\pi\mathbf{j}}{\mathbf{n}})}\right)_{\mathbf{j,k=0}}^{\mathbf{n-e}}.$$
%Given a matrix-valued function  $\mathbf{f}\in L^{1}([-\pi,\pi]^{t},r)$
\end{Theorem}
\section{Generalized Locally Toeplitz (GLT) sequences}
A GLT sequences \cite{Ba2-ETNA,Ba1-ETNA,glt-book-I,glt-book-II} is a specific sequence of matrices of size $d_n$ tending to infinity associated with a Lebesgue-measurable complex-valued function $\kappa$ which is defined on $[0,1]^t\times[-\pi,\pi]^t,\quad t\geq 1$, where $[0,1]^t $ and $[-\pi,\pi]^t$ represent the domains of physical variable and Fourier variable, respectively.
%associate to each couple $(x,y)\in[0,1]^t\times[-\pi,\pi]^t\rightarrow \mathbb{C}$.\\
 In this section, we will give a brief overview on the multilevel form of GLT sequences class \cite{block-glt-II,block-glt-I} (or $r$-block GLT sequences class, $r\geq1$) where the associated function is measurable matrix-valued function $\boldsymbol\kappa:[0,1]^t\times[-\pi,\pi]^t\to\mathbb C^{r\times r}$, $t\ge1$.
  We use the notation $\{A_\mathbf{n}\}_\mathbf{n}\sim_{\rm
GLT}\boldsymbol\kappa$ to indicate that $\{A_\mathbf{n}\}_\mathbf{n}$ is a GLT sequences with
symbol $\boldsymbol\kappa$. The symbol of a GLT sequences is unique in the sense that if $\{A_\mathbf{n}\}_\mathbf{n}\sim_{\rm GLT}\boldsymbol\kappa$ and $\{A_\mathbf{n}\}_\mathbf{n}\sim_{\rm GLT}\boldsymbol\varsigma$ then $\boldsymbol\kappa=\boldsymbol\varsigma$ almost everywhere in $[0,1]^t\times[-\pi,\pi]^t$.\\
$\bullet$ \textbf{Block Diagonal Sampling Matrices.} For
$n\in\mathbb N$, $t=1$, and \\$a:[0,1]\to\mathbb C^{r\times r}$, we
define the block diagonal sampling matrix $D_n(a)$ as the block diagonal
matrix
\[
D_n(a)=\mathop{\rm diag}_{i=1,\ldots,n}a\Bigl(\frac{i}{n}\Bigr)=\left[\begin{array}{cccc}a(\frac1n) & & & \\ & a(\frac2n) & & \\ & & \ddots & \\
& & & a(1)\end{array}\right]\in\mathbb C^{rn\times rn}.
\]
For a general dimensionality $t\ge 2$, we consider
$a:[0,1]^t\to\mathbb C^{r\times r}$, $\mathbf{n}=(n_1,\ldots,n_t)$
and we define the block multilevel diagonal sampling matrix
$D_\mathbf{n}(a)$ as the block diagonal matrix
\[
D_\mathbf{n}(a)=\mathop{\rm diag}_{
\mathbf{i}=\mathbf{e},\ldots,\mathbf{n}}a\Bigl(\frac{\mathbf{i}}{\mathbf{n}}\Bigr)\in\mathbb
C^{r N(\mathbf{n})\times r N(\mathbf{n})},
\]
where the multi-index $\mathbf{i}/\mathbf{n}$ has to be intended as $(i_1/n_1,\ldots,i_t/n_t)$ and the ordering is the lexicographical one as in the work by E. Tyrtyshnikov, (see for instance \cite{ty-1}).

The main properties of $r$-block GLT sequences proved in \cite{block-glt-II,block-glt-I} are listed below: they represent a complete characterization of GLT sequences, equivalent to the full constructive definition.
\begin{enumerate}
    \item[\textbf{GLT\,1.}] If $\{A_\mathbf{n}\}_\mathbf{n}\sim_{\rm GLT}\boldsymbol\kappa$ then $\{A_\mathbf{n}\}_\mathbf{n}\sim_\sigma\boldsymbol\kappa$. If moreover each $A_\mathbf{n}$ is Hermitian then $\{A_\mathbf{n}\}_\mathbf{n}\sim_\lambda\boldsymbol\kappa$.
    \item[\textbf{GLT\,2.}] We have:
    \begin{itemize}
        \item $\{T_\mathbf{n}(\mathbf{f})\}_\mathbf{n}\sim_{\rm GLT}\boldsymbol\kappa(\mathbf{x},\boldsymbol{\theta})=\mathbf{f}(\boldsymbol{\theta})$ if $\mathbf{f}:[-\pi,\pi]^t\to\mathbb C^{r\times r}$ is in $L^1([-\pi,\pi]^t)$;
        \item $\{D_\mathbf{n}(a)\}_\mathbf{n}\sim_{\rm GLT}\boldsymbol\kappa(\mathbf{x},\boldsymbol{\theta})=a(\mathbf{x})$ if $a:[0,1]^t\to\mathbb C^{r\times r}$ is Riemann-integrable;
        \item $\{Z_\mathbf{n}\}_\mathbf{n}\sim_{\rm GLT}\boldsymbol\kappa(\mathbf{x},\boldsymbol{\theta})=O_r$ if and only if $\{Z_\mathbf{n}\}_\mathbf{n}\sim_\sigma 0$.
    \end{itemize}
    \item[\textbf{GLT\,3.}] If $\{A_\mathbf{n}\}_\mathbf{n}\sim_{\rm GLT}\boldsymbol\kappa$ and $\{B_\mathbf{n}\}_\mathbf{n}\sim_{\rm GLT}\boldsymbol\varsigma$ then:
    \begin{itemize}
        \item $\{A_\mathbf{n}^*\}_\mathbf{n}\sim_{\rm GLT}\boldsymbol\kappa^*$;
        \item $\{\alpha A_\mathbf{n}+\beta B_\mathbf{n}\}_\mathbf{n}\sim_{\rm GLT}\alpha\boldsymbol\kappa+\beta\boldsymbol\varsigma$ for all $\alpha,\beta\in\mathbb C$;
        \item $\{A_\mathbf{n}B_\mathbf{n}\}_\mathbf{n}\sim_{\rm GLT}\boldsymbol{\kappa\varsigma}$;
        \item $\{A_\mathbf{n}^\dag\}_\mathbf{n}\sim_{\rm GLT}\boldsymbol\kappa^{-1}$ provided that $\boldsymbol\kappa$ is invertible almost everywhere.
    \end{itemize}
    \item[\textbf{GLT\,4.}] $\{A_\mathbf{n}\}_\mathbf{n}\sim_{\rm GLT}\boldsymbol\kappa$ if and only if there exist $r$-block GLT sequences \\$\{B_\mathbf{n,m}\}_\mathbf{n}\sim_{\rm GLT}\boldsymbol\kappa_\mathbf{m}$ such that $\{B_\mathbf{n,m}\}_n\stackrel{\rm a.c.s.}{\longrightarrow}\{A_\mathbf{n}\}_\mathbf{n}$ and $\boldsymbol\kappa_\mathbf{m}\to\boldsymbol\kappa$ in measure, where the a.c.s. convergence is studied in \cite{glt-book-I}.
\end{enumerate}

 %
 % \begin{Definition}[at atraction]
 % gerchgorin interlacing Poincaré separation theorem approximation class theorem1.6 De garoni zero distrubuted sequence sparsely vanishing
%  \end{Definition}
 % \subsection{la fonction et matrice associ}

\chapter{Spectral analysis of ${\mathbb{P}_{k}}$ Finite Element Matrices in the case of Friedrichs-Keller triangulations via GLT Technology}

The current chapter deals with the spectral analysis of matrix-sequences arising from the $\mathbb{P}_k$  Lagrangian Finite Element approximation of the elliptic problem.
\begin{equation} \label{eq:modello}
\left \{
\begin{array}{l}
\mathrm{div} \left(-a(\mathbf{x}) \nabla u\right)
=f, \quad \mathbf{x}\in \Omega\subseteq \mathbb{R}^d, \\
u_{|\partial \Omega}=0,
\end{array}
\right.
\end{equation}
with $\Omega$ bounded connected subset of $\mathbb{R}^d$, $d\ge 1$, having smooth boundaries for $d\ge 2$, and $a$ being continuous and positive on $\overline\Omega$ and $f\in L^{2}(\Omega)$. For the numerical approximation we consider the classical $\mathbb{P}_k$ Finite Elements, where $k$ is a scalar indicating the global polynomial degree, in the case of Friedrichs-Keller triangulations,  leading, as usual, to sequences of matrices of increasing size. The new results concern the spectral analysis of the resulting matrix-sequences in the direction of the global distribution in the Weyl sense, with a concise overview on localization, clustering, extremal eigenvalues, and asymptotic conditioning. We study in detail the case of constant coefficients  on $\Omega=(0,1)^2$ and we give a brief account in the more involved case of variable coefficients and more general domains. Tools are drawn from the Toeplitz technology and from the rather new theory of GLT sequences, (see \cite{Ba2-ETNA,Ba1-ETNA,glt-book-I,glt-book-II} and references therein). Numerical results are shown for providing a practical evidence of the theoretical findings.

\section{Finite Element approximation} \label{sez:fem}
%-------------------------------------------------------------------------------
% Ipotesi + quadro funzionale
%-------------------------------------------------------------------------------
Problem (\ref{eq:modello}) can be formulated in variational form
as follows:
\begin{equation} \label{eq:formulazione_variazionale}
%\left \{
\begin{array}{l}
\textrm{find $u \in H_0^1(\Omega)$ such that} %\\
\int_\Omega \left ( a \nabla u \cdot \nabla \varphi
%-\b \cdot \nabla \varphi \ u
%+ \s u \varphi
\right ) =\int_\Omega f \varphi  \quad \textrm{for all } \varphi
\in  H_0^1(\Omega),
\end{array}
%
%\right.
\end{equation}
where $H_0^1(\Omega)$ is the space of square integrable functions
vanishing on $\partial \Omega$, with square integrable weak
derivatives. We assume that $\Omega \subseteq \mathbb{R}^2$ is  a
bounded connected set with smooth boundaries (in practice in our
numerical tests $\Omega$ will be a polygonal domain). Furthermore,
we make the following hypotheses on the coefficients:
\begin{equation} \label{eq:ipotesi_coefficienti}
%\left \{
\begin{array}{l}
a \in { C}(\overline \Omega), \textrm{ with } a(\mathbf{x}) \ge a_0 >0, \textrm{ and } %\\
%\b \in {\bf C}^1(\overline \Omega),\quad \textrm{ with } \mathrm{div} (\b) \ge 0 \textrm{ pointwise in } \Omega, \\
f \in {L}^2(\Omega),
\end{array}
%\right.
\end{equation}
so that  existence and uniqueness for problem
(\ref{eq:formulazione_variazionale}) are guaranteed. Hereafter, we
consider $\mathbb{P}_k$ Lagrangian Finite Element approximation of
problem (\ref{eq:formulazione_variazionale}). To this end, let
$\mathcal{T}_h=\{K\}$ be a usual Finite Element partition of
$\overline \Omega$ into triangles, with $h_K=\mathrm{diam}(K)$ and
$h=\max_K{h_K}$, and let $V_h \subset H^1_0(\Omega)$ be the space
of $\mathbb{P}_k$ Lagrangian Finite Element, i.e.
\[
\begin{array}{rcl}
V_h &=&  \{\varphi_h : \overline \Omega \rightarrow \mathbb{R} \
\textrm{ s.t. } \varphi_h \textrm{ is continuous,  }
{\varphi_h}_{|_K} \textrm{ is a polynomial of degree } \\
&& \textrm{\hskip 3cm  less or equal to $k$, and } {\varphi_h}_{|
\partial \Omega} =0
 \}.\end{array}
\]
The Finite Element approximation of problem
(\ref{eq:formulazione_variazionale})  reads as follows:
\begin{equation} \label{eq:formulazione_variazionale_fe}
%\left \{
\begin{array}{l}
\textrm{find $u_h \in V_h$ such that} %\\
\int_\Omega \left (a \nabla u_h \cdot \nabla \varphi_h
% -\b \cdot \nabla \varphi_h \ u_h
%+ \s u_h \varphi_h
\right )
=\int_\Omega f \varphi_h \quad \textrm{for all } \varphi_h \in  V_h.
\end{array}
%
%\right.
\end{equation}
For each internal node $i$ of the mesh $\mathcal{T}_h$, meaning
both vertices and additional nodal values associated to the
$\mathbb{P}_k$ approximation, let $\varphi_i \in V_h$ be such that
$\varphi_i(\textrm{node }i)=1$, and $\varphi_i(\textrm{node }j)=0$
if $i\ne j$. Then, the collection of all $\varphi_i$'s is  a basis
for $V_h$ and we denote by $n(h)$ its dimension. Then, we write
$u_h$ as $ u_h=\sum_{j=1}^{n(h)}u_j \varphi_j$  and the
variational equation (\ref{eq:formulazione_variazionale_fe})
becomes an algebraic linear system:
\begin{equation} \label{eq:modello_discreto}
\sum_{j=1}^{n(h)}\left (\int_\Omega a \nabla \varphi_j \cdot \nabla \varphi_i
%- \b \cdot \nabla \varphi_i\ \varphi_j
%+ \s \varphi_i \ \varphi_j
\right ) u_j =\int_\Omega f \varphi_i, \quad i=1,\ldots, n(h).
\end{equation}
Our aim is to analyse the spectral properties of the matrix-sequences $\{A_n(a,\Omega,{\mathbb{P}_k})\}_n$ arising in the quoted  linear systems (\ref{eq:modello_discreto}), both from the theoretical and numerical point of view.

\section{A Few remarks on the monodimensional case: $\mathbb{Q}_k\equiv \mathbb{P}_k$, $d=1$}\label{sez:Pk-1D}
We report some results derived in \cite{Q_k} for the Lagrangian
Finite Elements $\mathbb{Q}_k\equiv \mathbb{P}_k$, $d=1$. Let us
consider  the Lagrange polynomials $L_0,\ldots,L_k$ associated
with the reference knots $t_j= j/ k,\ j=0,\ldots,k$:
\begin{equation}\label{Lp}
\begin{aligned}
L_i(t)&=\prod_{\substack{j=0\\j\ne
i}}^k\frac{t-t_j}{t_i-t_j}=\prod_{\substack{j=0\\j\ne
i}}^k\frac{kt-j}{i-j},\quad i=0,\ldots,k,\\ % t=s/k quindi vista l'elisione s che si riscrive come tk
L_i(t_j)&=\delta_{ij},\quad i,j=0,\ldots,k,
\end{aligned}
\end{equation}
%-------------------------------------------------------------------------------
and let  the symbol $\langle\,,\rangle$ denote the scalar product
in $L^2([0,1])$, i.e.,
$\langle\varphi,\psi\rangle:=\int_0^1\varphi\psi$.
In the case $a(x)\equiv 1$ and $\Omega=(0,1)$ the $\mathbb{Q}_k$
matrix $A_n(a,\Omega,{\mathbb{Q}_k})$ equals the matrix $K_n^{(k)}$ in Theorem
\ref{itdsc}.
\begin{Theorem}\label{itdsc}
Let $k,n\ge1$. Then
{\allowdisplaybreaks\begin{equation}\label{Knr}
K_n^{(k)}=\left[\begin{array}{cccc}
K_0 & K_1^T & & \\
K_1 & \ddots & \ddots & \\
& \ddots & \ddots & K_1^T\\
& & K_1 & K_0
\end{array}\right]_-
%\quad
%M_n^{(k)}=\left[\begin{array}{cccc}
%M_0 & M_1^T & & \\
%M_1 & \ddots & \ddots & \\
%& \ddots & \ddots & M_1^T\\
%& & M_1 & M_0
%\end{array}\right]_-
\end{equation}}
where the  subscript `$-$'  means that the last row and column of
the  matrix in square brackets are deleted, while $K_0,K_1$ are
$k\times k$ blocks given by
%{\allowdisplaybreaks
\begin{align}\label{K_blocks}
\begin{aligned}
K_0&=\left[\begin{array}{ccc|c}
\langle L_1',L_1'\rangle & \cdots & \langle L_{k-1}',L_1'\rangle & \langle L_k',L_1'\rangle\\
\vdots & & \vdots & \vdots\\
\langle L_1',L_{k-1}'\rangle & \cdots & \langle L_{k-1}',L_{k-1_{\vphantom{\sum}}}'\rangle & \langle L_k',L_{k-1}'\rangle\\
\hline
\langle L_1',L_k'\rangle & \cdots & \langle L_{k-1}',L_k'\rangle & {\langle L_k',L_k'\rangle+\langle L_0',L_0'\rangle}^{\vphantom{\sum}}
\end{array}\right],\\
K_1&=\left[\begin{array}{cccc|c}
0 & 0 & \cdots & 0 & \langle L_0',L_1'\rangle\\
0 & 0 & \cdots & 0 & \langle L_0',L_2'\rangle\\
\vdots & \vdots & & \vdots & \vdots\\
0 & 0 & \cdots & 0 & \langle L_0',L_k'\rangle
\end{array}\right],
\end{aligned}
\end{align}
%\chapter{Multigrid for $\mathbb{Q}_k$ Finite Element Matrices Using a (Block) Toeplitz Symbol Approach}
with $L_0,\ldots,L_k$ being the Lagrange polynomials in
\eqref{Lp}. In particular, $K_n^{(k)}$ is the $(nk-1)\times(nk-1)$
leading principal submatrix of the block Toeplitz matrices
$T_n(\mathbf{f}_{k})$ and $\mathbf{f}_{k}:[-\pi,\pi]\to\mathbb C^{k\times k}$ is an Hermitian matrix-valued function given by
\begin{align}
\mathbf{f}_{k}(\theta)&:=K_0+K_1 e^{\hat{\imath}\theta}+K_1^T e^{-\hat{\imath} \theta}\label{fr(theta)}.
\end{align}
\end{Theorem}
An interesting property of the Hermitian matrix-valued functions
$\mathbf{f}_k(\theta)$ defined in \eqref{fr(theta)} is reported in the
theorem  below. From the point of view of the spectral
distribution, the message is that, independently of the parameter
$k$, the spectral symbol is of the same character as
$2-2\cos(\theta)$ which is the symbol of the basic linear Finite
Elements and the most standard Finite Differences approximation.

\begin{Theorem}\label{dfr}
Let $k\ge1$, then
\begin{equation}\label{dfrFormula}
\det(\mathbf{f}_k(\theta))=d_k(2-2\cos(\theta)),
\end{equation}
where $d_k=\det([\langle
L_j',L_i'\rangle]_{i,j=1}^k)=\det([\langle
L_j',L_i'\rangle]_{i,j=1}^{k-1})>0$ (with  $d_1=1$ being the
determinant of the empty matrix by convention) and
$L_0,\ldots,L_k$ are the Lagrange polynomials \eqref{Lp}.
%\footnote{\,The determinant of
%the empty matrix is 1 by convention, so $d_1=1$.}
\end{Theorem}
\section{Two dimensional case: $\mathbb{P}_k$, $d=2$ - symbol definition}\label{sez:Pk-2D}
Hereafter we focus on $\mathbb{P}_k$ Lagrangian Finite Elements in
the case of Friedrichs-Keller triangulations $\{\mathcal{T}_K\}$
of the domain $\Omega$ as reported in Figure \ref{fig:Friedrics-Keller}. Nodes, that is both vertices and
additional nodal values associated to the chosen $\mathbb{P}_k$
approximation, are ordered in standard lexicographical
way %order
from left to right, from bottom to top. \\
\begin{figure}
\centering
%\epsfig{file=disegno_mesh.eps,height=6cm}
%\vskip -1.5cm
\label{fig:Friedrics-Keller}
\includegraphics[width=\textwidth]{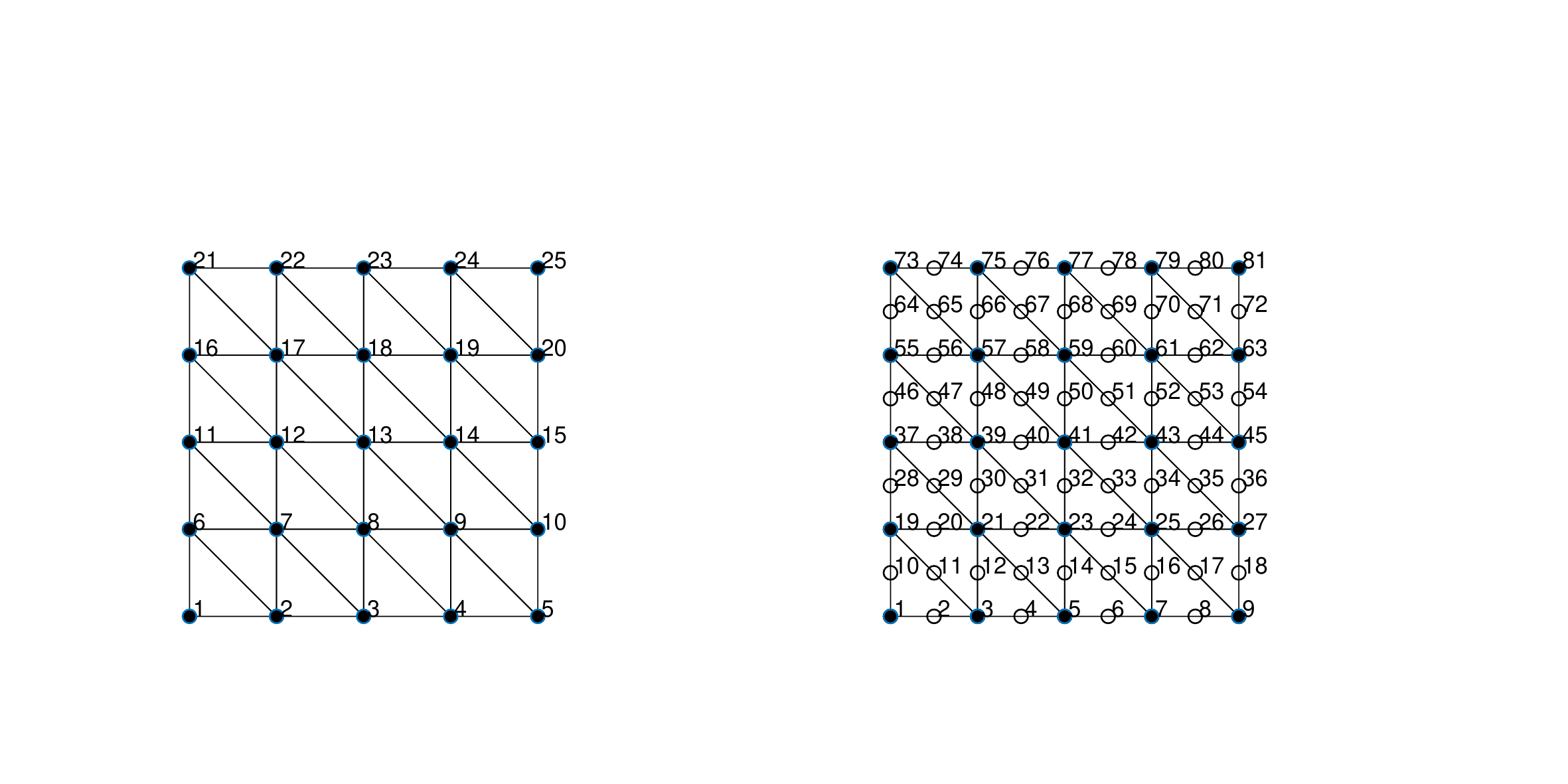}
\caption{Friedrics-Keller meshes for $\mathbb{P}_k,\ k=1,2$.}
\end{figure}
The stiffness matrix is built by considering the standard
assembling procedure with respect to the reference element $\hat{K}$ in Figure \ref{fig:elementodiriferimento}. Let $G$ be the affine
transformation  mapping  $\hat{K}$ onto a generic $K \in
\mathcal{T}_K$ defined as
\begin{equation*}
G\left ( \left [\begin{array}{c}
\hat{x} \\
\hat{y}
\end{array}\right] \right) = \left [\begin{array}{cc}
(e_{3})_1 & -(e_{2})_1 \\
(e_{3})_2 & -(e_{2})_2
\end{array}\right]
\left [\begin{array}{c}
\hat{x} \\
\hat{y}
\end{array}\right]
+
\left [\begin{array}{c}
x^{v_1} \\
y^{v_1}
\end{array}\right],
\end{equation*}
where $e_1=[x^{v_3}-x^{v_2},y^{v_3}-y^{v_2}]^T$,
$e_2=[x^{v_1}-x^{v_3},y^{v_1}-y^{v_3}]^T$,
$e_3=[x^{v_2}-x^{v_1},y^{v_2}-y^{v_1}]^T$ represent the oriented
 edge vectors and $(x^{v_i},y^{v_i})$ are the coordinates of the
$i^{th}$ vertex $v_i$.
%-------------------------------------------------------------------------------
\begin{figure}
\centering
\includegraphics[width=\textwidth]{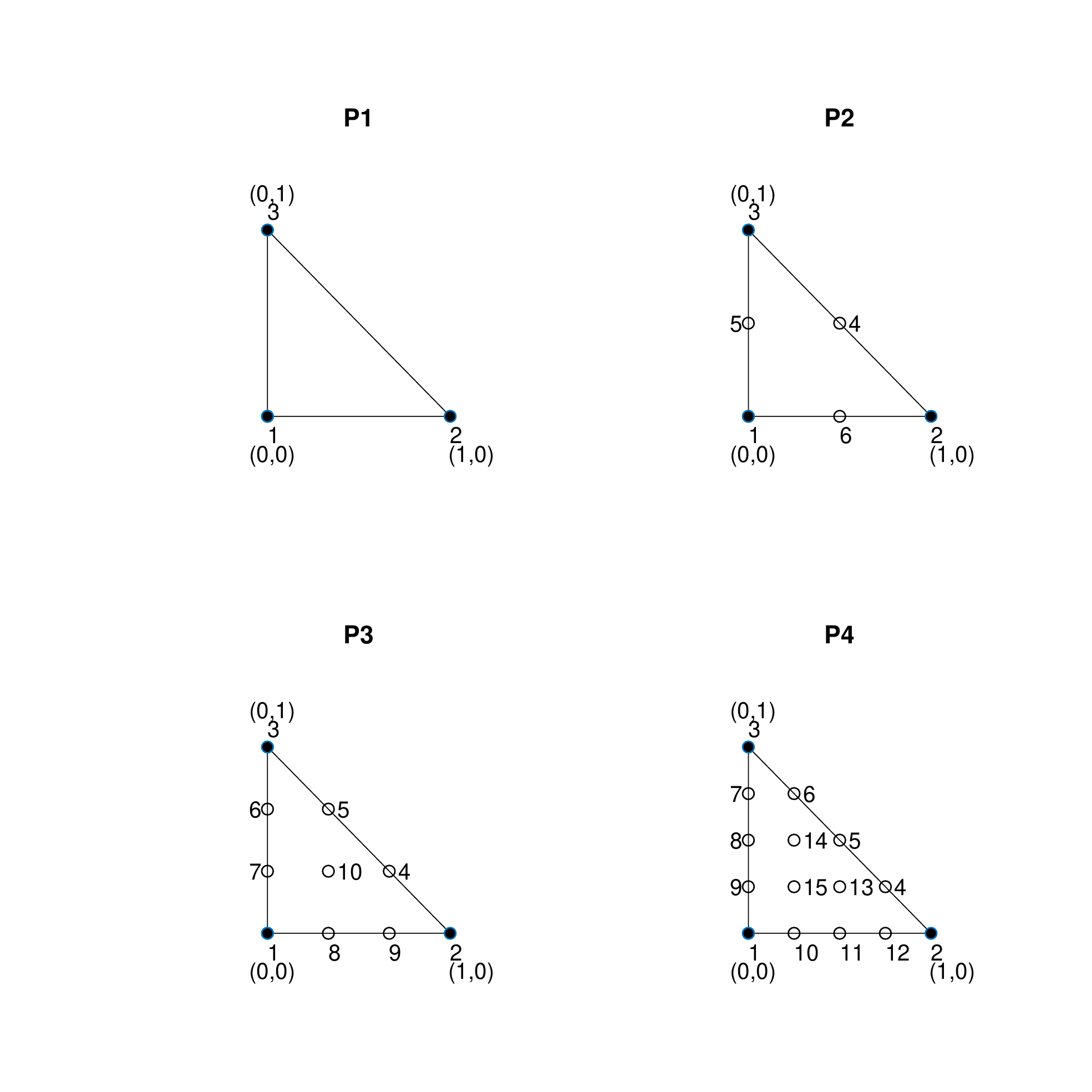}
\caption{Reference element $\hat{K}$ and nodal points for
$\mathbb{P}_k,\ k=1,\ldots,4$.} \label{fig:elementodiriferimento}
\end{figure}
%------------------------------------------------------------------------
Thus,
\begin{equation*}
A^{El}_K=  \left [  \int_K \nabla \varphi_j  \cdot \nabla
\varphi_i  \right ]_{i,j}
\end{equation*}
with
\begin{equation} \label{eq:int_ij}
\int_K \nabla \varphi_j  \cdot \nabla \varphi_i = \mathrm{det}
(J_G(\hat{x}, \hat{y})) \int_{\hat{K}} [J_{G^{-1}}^T \hat{\nabla}\hat{\varphi_j}
(\hat x,\hat y) ]  \cdot
                [J_{G^{-1}}^T \hat{\nabla} \hat{\varphi}_i(\hat x,\hat y) ] d\hat{x} d\hat{y},
\end{equation}
%-------
 where the $ \hat{\varphi}_s $'s are the shape functions on $\hat{K}$,
$\mathrm{det} (J_G(\hat{x}, \hat{y})) = 2 |K|$ and $J_{G^{-1}}^T$ is
the transpose of the Jacobian matrix of the inverse mapping
$G^{-1}$, that is,
\begin{equation*}
J_{G^{-1}}^T = \frac{1}{2|K|} \left [\begin{array}{rr}
-(e_{2})_2 & -(e_{3})_2 \\
(e_{2})_1 & (e_{3})_1
\end{array}\right].
\end{equation*}
In the present section we will preliminarily  consider the case
$a\equiv 1$ and $\Omega=(0,1)^2$.

\subsection{Case $k=1$}
Even if  well known, we start by considering the case $k=1$, that is the one of a linear Lagrangian FE approximation.  The shape functions on $\hat{K}$ are defined as
\begin{eqnarray}
\hat{\varphi}_1(\hat{x}, \hat{y})&=&-\hat{x} -\hat{y} +1, \nonumber\\
\hat{\varphi}_2(\hat{x}, \hat{y})&=&\hat{x},  \label{eq:baseP12D}\\
\hat{\varphi}_3(\hat{x}, \hat{y})&=&\hat{y}, \nonumber
\end{eqnarray}
%--------------------------------------------------------------------------
so that, according to (\ref{eq:int_ij}), the
{elemental} matrix for a generic triangle of the
Friedrichs-Keller triangulation, that is a right-angle triangle of
constant edge $h$, equals
\begin{equation} \label{eq:Ael_k1}
A_{K_1}^{El} = \frac{1}{2} \left [\begin{array}{rrr}
2 & -1 & -1 \\
-1 & 1 & 0 \\
-1 & 0 & 1
\end{array} \right ] \quad \mathrm{or} \quad
A_{K_2}^{El} = \frac{1}{2} \left [\begin{array}{rrr}
1 & -1 & 0 \\
-1 & 2 & -1 \\
0 & -1 & 1
\end{array} \right ],
\end{equation}
for triangles of type 1 (right angle in vertex 1) or type 2 (right angle in vertex 2), respectively. \\
The stiffness matrix $A_n=A_n(1,\Omega,{\mathbb{P}_1})$, that is
with $a\equiv 1$, is the twolevel Toeplitz matrix generated by the
symbol
$f_{\mathbb{P}_1}(\theta_{1},\theta_{2})=4-2\cos(\theta_{1})-2\cos(\theta_{2})$.
In fact, $A_n$ is block tridiagonal, i.e.
\[
A_n=\mathrm{tridiag}(A_1,A_0,A_{-1}),
\]
 where the triangular blocks are such that $A_0=\mathrm{tridiag}(a_1^0,a_0^0,a_{-1}^0)=\mathrm{tridiag}(-1,4,-1)$, $A_1=A_{-1}=\mathrm{diag}(a_1^1)=-I$, $I$
being the identity matrix.
Thus we can easily read the corresponding symbol as follows
\begin{equation} \label{eq:simbolotridiag}
f_{\mathbb{P}_1}(\theta_{1},\theta_{2})=f_{A_0}(\theta_{1})+f_{A_{-1}}(\theta_{1}) e^{-\hat{\imath} \theta_{2}} +f_{A_{1}}(\theta_{1}) e^{\hat{\imath} \theta_{2}},
\end{equation}
with $f_{A_0}(\theta_{1})=a_0^0 +a_{-1}^0 e^{-\hat{\imath} \theta_{1}}+a_{1}^0 e^{\hat{\imath} \theta_{1}}$ and $f_{A_{1}}(\theta_{1})=f_{A_{-1}}(\theta_{1})=a_1^1$.\\
%**** Short description of $f_{\mathbb{P}_1}$ properties ***. \\
%
Clearly, the natural arising question is: which properties are preserved in considering Lagrangian FE of higher order?
\subsection{Case $k=2$}
Hereafter, we will consider in full detail the case of quadratic Lagrangian FE ($k=2$), the aim being to introduce a suitable notation making easier the analysis of higher order approximations  as well.
By referring to the reference element, (see Figure \ref{fig:elementodiriferimento}), we have the following shape functions
\begin{eqnarray}
\hat{\varphi}_1(\hat{x}, \hat{y})&=& 2\hat{x}^2 + 2 \hat{y}^2 + 4 \hat{x} \hat{y} -3 \hat{x} -3 \hat{y} +1 ,\nonumber\\
\hat{\varphi}_2(\hat{x}, \hat{y})&=& \hat{x} (2 \hat{x} -1),  \nonumber \\
\hat{\varphi}_3(\hat{x}, \hat{y})&=& \hat{y} (2 \hat{y} -1), \nonumber \\
\hat{\varphi}_4(\hat{x}, \hat{y})&=& 4 \hat{x} \hat{y}, \label{eq:baseP22D}\\
\hat{\varphi}_5(\hat{x}, \hat{y})&=& -4 \hat{y} ( \hat{x} +\hat{y}-1),  \nonumber \\
\hat{\varphi}_6(\hat{x}, \hat{y})&=& -4 \hat{x} ( \hat{x} +\hat{y}-1), \nonumber
\end{eqnarray}

so that, according to (\ref{eq:int_ij}), the {elemental} matrix for a generic triangle equals
\begin{equation} \label{eq:Ael_k2}
\renewcommand{\arraystretch}{1.2}
A_{K_1}^{El} =  \left [\begin{array}{rrrrrr}
 \mbox{\small 1} &  \frac{1}{6}     &  \frac{1}{6}     &  \mbox{\small 0} & -\frac{2}{3}     & -\frac{2}{3} \\
 \frac{1}{6}     &  \frac{1}{2}     &  \mbox{\small 0} &  \mbox{\small 0} &  \mbox{\small 0} & -\frac{2}{3} \\
 \frac{1}{6}     &  \mbox{\small 0} &  \frac{1}{2}     &  \mbox{\small 0} & -\frac{2}{3}     &  \mbox{\small 0}   \\
 \mbox{\small 0} &  \mbox{\small 0} &  \mbox{\small 0} &  \frac{8}{3}     & -\frac{4}{3}     &  -\frac{4}{3} \\
-\frac{2}{3}     &  \mbox{\small 0} & -\frac{2}{3}     & -\frac{4}{3}     &  \frac{8}{3}     &  \mbox{\small 0} \\
-\frac{2}{3}     & -\frac{2}{3}     &  \mbox{\small 0} & -\frac{4}{3}     &  \mbox{\small 0} &  \frac{8}{3}\\
\end{array} \right ]
\end{equation}
in the case of triangles of type 1, or a suitable permutation in the case of triangles of type 2.
in the case of triangles of type 1, or a suitable permutation in the case of triangles of type 2.
Despite the use of Lagrangian quadratic approximation, the
stiffness matrix $A_n=A_n(1,\Omega,{\mathbb{P}_2})$ shows again a
block tridiagonal structure $A_n=\mathrm{tridiag}(A_1,A_0,A_{-1})$
as in the linear case, the higher approximation stressing its
influence just inside the blocks $A_i$. We might say that the
quoted tridiagonal structure refers once again to triangles'
vertices, while the internal structure is stressing the increased
number of additional nodal points, that is three in the case at
hand. In fact, we observe in each block $A_i$ a $2\times 2$ block
structure as  follows:

\begin{equation}
A_0=\left [\begin{array}{cc}
   B_0^{11}      & B_0^{12} \\
   (B_0^{12})^T  & B_0^{22} \\
\end{array} \right ], \quad
A_{-1}=\left [\begin{array}{cc}
   0 & 0 \\
   B_{-1}^{21}  & B_{-1}^{22}\\
\end{array} \right ], \quad A_{1} = A_{-1}^T,
\end{equation}
where the superscripts $i,j$ in $B_l^{ij}$ denote the position inside the $2 \times 2$ block and the subscript $l$ the belonging to the block $A_l$,
so that
\begin{equation} \label{eq:blocchi_Bl}
\renewcommand{\arraystretch}{1.2}
A_n=\left [\begin{array}{cc|cc|cc|cc}
   B_0^{11}     & B_0^{12}        & 0 & 0\\
  (B_0^{12})^T  & B_0^{22}        & B_{-1}^{21}   & B_{-1}^{22} \\
\hline
   0            & (B_{-1}^{21})^T & B_0^{11}      & B_0^{12}    & 0 & 0\\
   0            & (B_{-1}^{22})^T & (B_0^{12})^T  & B_0^{22}  & B_{-1}^{21}   & B_{-1}^{22} \\
   \hline
 & & \ddots & &  \ddots & &  \ddots & \\
 \hline
   && 0            & (B_{-1}^{21})^T & B_0^{11}      & B_0^{12}    & 0 \\
   && 0            & (B_{-1}^{22})^T & (B_0^{12})^T  & B_0^{22}  & B_{-1}^{21}    \\
   \hline
     &&&& 0            & (B_{-1}^{21})^T & B_0^{11}          \\
\end{array} \right ].
\end{equation}
More important, the very same structure depicted in (\ref{eq:blocchi_Bl}), including the very same cutting in the lower right corner, appears in every block $B_l^{ij}$ by considering suitable $2\times 2$ matrices as follows
\begin{equation*}
B_l^{ij}=\mathrm{tridiag}\left(
a_{1}^{B_l^{ij}},a_0^{B_l^{ij}},a_{-1}^{B_l^{ij}} \right ), \quad
l\in \{-1, 0, 1\}, \ i,j \in \{1, 2\},
\end{equation*}
where
\begin{equation*}
\renewcommand{\arraystretch}{1.2}
a_0^{B_0^{11}} = \left [\begin{array}{rr}
   \frac{16}{3}  & -\frac{4}{3} \\
  -\frac{4}{3}   & \frac{16}{3} \\
\end{array} \right ], \quad
a_{-1}^{B_0^{11}}  =\left [\begin{array}{rr}
   \mbox{\small 0}  & \mbox{\small 0} \\
  -\frac{4}{3}      & \mbox{\small 0} \\
\end{array} \right ], \quad
a_{1}^{B_0^{11}} = \left (a_{-1}^{B_0^{11}}\right )^T,
\end{equation*}

\begin{equation*}
\renewcommand{\arraystretch}{1.2}
a_0^{B_0^{22}} = \left [\begin{array}{rr}
   \frac{16}{3}  & -\frac{4}{3} \\
  -\frac{4}{3}   & \mbox{\small 4} \\
\end{array} \right ], \quad
a_{-1}^{B_0^{22}}  =\left [\begin{array}{rr}
   \mbox{\small 0}  & \mbox{\small 0} \\
  -\frac{4}{3}      & \frac{1}{3} \\
\end{array} \right ], \quad
a_{1}^{B_0^{22}} = \left (a_{-1}^{B_0^{22}}\right )^T,
\end{equation*}

\begin{equation*}
a_0^{B_0^{12}} = -\frac{4}{3}  I_{2}, \quad
%\left [\begin{array}{rr}
%  -\frac{4}{3}  & 0 \\
%  0   & -\frac{4}{3} \\
%\end{array} \right ], \quad
%
a_{-1}^{B_0^{12}}  =  a_{1}^{B_0^{12}} =  O_{2},
\end{equation*}
%-----------------
\begin{equation*}
a_0^{B_{-1}^{21}} = -\frac{4}{3}  I_{2}, \quad
%\left [\begin{array}{rr}
%  -\frac{4}{3}  & 0 \\
%  0   & -\frac{4}{3} \\
%\end{array} \right ], \quad
%
a_{-1}^{B_{-1}^{21}}  =  a_{1}^{B_{-1}^{21}} =  O_{2},
\end{equation*}
%-----------------
\begin{equation*}
a_0^{B_{-1}^{22}} = \left [\begin{array}{rr}
  0  & 0 \\
  0   & \frac{1}{3} \\
\end{array} \right ], \quad
a_{-1}^{B_{-1}^{21}}  =  a_{1}^{B_{-1}^{21}} =  O_{2}.
\end{equation*}
%-----------------
Thus, once again, just by taking into account that we are now
facing a matrix-valued symbol,  we can easily read the underlying
symbol as  follows:
\begin{equation} \label{eq:simbolotridiag_P2}
\mathbf{f}_{\mathbb{P}_2}(\theta_{1},\theta_{2})=f_{A_0}(\theta_{1})+f_{A_{-1}}(\theta_{1}) e^{-\hat{\imath} \theta_{2}} +f_{A_{1}}(\theta_{1}) e^{\hat{\imath} \theta_{2}},
\end{equation}

with
\begin{eqnarray*}
f_{A_0}(\theta_{1}) & = & \left [\begin{array}{cc}
   f_{B_0^{11}}(\theta_{1})      & f_{B_0^{12}}(\theta_{1}) \\
   f_{(B_0^{12})^T }(\theta_{1}) & f_{B_0^{22}}(\theta_{1}) \\
\end{array} \right ], \
\begin{array}{l}
f_{B_l^{ij}}(\theta_{1}) = a_0^{B_l^{ij}} +a_{-1}^{B_l^{ij}} e^{-\hat{\imath} \theta_{1}}+a_{1}^{B_l^{ij}} e^{\hat{\imath} \theta_{1}},\\
f_{(B_l^{ij})^T}(\theta_{1}) = \overline{f_{B_l^{ij}}(\theta_{1})},
\end{array}
\\%
%-------------
%
f_{A_{-1}}(\theta_{1}) &=&\left [\begin{array}{cc}
   0      & 0 \\
   f_{B_{-1}^{21}}(\theta_{1}) & f_{B_{-1}^{22}}(\theta_{1}) \\
\end{array} \right ], \quad
f_{A_1}(\theta_{1}) =\left [\begin{array}{cc}
   0      & f_{(B_{-1}^{21})^T}(\theta_{1}) \\
   0      & f_{(B_{-1}^{22})^T}(\theta_{1}) \\
\end{array} \right ].
\end{eqnarray*}

To sum up, we have a matrix-valued symbol $\mathbf{f}_{\mathbb{P}_2}: [-\pi,\pi]^2 \longrightarrow \mathbb{C}^{4\times 4}$ with
\begin{equation}\label{eq:P2symbol}
\mathbf{f}_{\mathbb{P}_2}(\theta_{1},\theta_{2})  =  \left [
\begin{array}{rr|rr}
\alpha                 & -\beta(1+e^{\hat{\imath} \theta_{1}})  & -\beta(1+e^{\hat{\imath} \theta_{2}})  & 0 \\
-\beta(1+e^{-\hat{\imath} \theta_{1}}) & \alpha                 & 0                      & -\beta(1+e^{\hat{\imath} \theta_{2}})\\
\hline
-\beta(1+e^{-\hat{\imath} \theta_{2}}) & 0                     & \alpha                 & -\beta(1+e^{\hat{\imath} \theta_{1}})\\
0                      & -\beta(1+e^{-\hat{\imath} \theta_{2}}) & -\beta(1+e^{-\hat{\imath} \theta_{1}}) & \gamma +\frac{\beta}{2} (\cos(\theta_{1})+\cos(\theta_{2}))\\
\end{array}
\right ]
\end{equation}
with $\alpha=\dfrac{16}{3}$, $\beta=\dfrac{4}{3}$, and $\gamma=4$.\\

Finally, it is worth stressing that the stiffness matrix $A_n(1,\Omega,\mathbb{P}_2)$ is a principal submatrix of  a suitable
permutation of the Toeplitz matrix  $T_n(\mathbf{f}_{\mathbb{P}_2})$
defined according to (\ref{eq:toeplitz_kron}).
 Indeed, the size of the twolevel matrix $A_n=A_n(1,\Omega,\mathbb{P}_2)$ is intrinsically odd both in inner and outer dimensions, (see Theorem \ref{itdsc} and the explanation after equation (\ref{Knr})), while
  $T_n(\mathbf{f}_{\mathbb{P}_2})$ has  even corresponding dimensions: it is
enough  to cut every last row/column in each inner block, together
with the last block with respect rows and columns, in order to
obtain $A_n$ from $T_n(\mathbf{f}_{\mathbb{P}_2})$. In other words $A_n$ is
a special principal submatrix of $T_n(\mathbf{f}_{\mathbb{P}_2})$ according
to the rule given in  Theorem \ref{itdsc}.
As for the permutation we have just to consider the one defined by
ordering nodal values as reported in Figure
\ref{fig:ordinamentoFK_P2}, where internal nodal values are
grouped four by four.
As a consequence the two matrix-sequences $\{T_n(\mathbf{f}_{\mathbb{P}_2})\}_n$ and
$\{A_n(1,\Omega,\mathbb{P}_2)\}_n$  share the same spectral
distribution,   that is, the same spectral symbol
$\mathbf{f}_{\mathbb{P}_2}$,  by invoking  Theorem \ref{extradimensional}.
The following proposition holds.
\begin{Proposition}\label{distr-P_2}
The two matrix-sequences $\{T_n(\mathbf{f}_{\mathbb{P}_2})\}_n$ and
$\{A_n(1,\Omega,\mathbb{P}_2)\}_n$  are spectrally distributed as
$\mathbf{f}_{\mathbb{P}_2}$ in the sense of Definition
\ref{def-distribution}.
\end{Proposition}

As an immediate consequence of Proposition \ref{distr-P_2} we deduce a corollary regarding the clustering and localization of the spectra of
$\{T_n(\mathbf{f}_{\mathbb{P}_2})\}_n$ and $\{A_n(1,\Omega,\mathbb{P}_2)\}_n$.
\begin{Corollary}\label{cor:distr-P_2}
The range of $\mathbf{f}_{\mathbb{P}_2}$ is a weak cluster set for the spectra of the two matrix-sequences $\{T_n(\mathbf{f}_{\mathbb{P}_2})\}_n$ and $\{A_n(1,\Omega,\mathbb{P}_2)\}_n$ in the sense of Definition \ref{def-cluster}. Furthermore, the convex hull of the range of $\mathbf{f}_{\mathbb{P}_2}$ contains all the eigenvalues of the involved matrices.
\end{Corollary}
\begin{proof} The proof of the first part is a direct consequence  of Proposition \ref{distr-P_2}, taking into account \cite[Theorem 4.2]{gol-serra} and observing that in this setting the standard range and the essential range coincide since $\mathbf{f}_{\mathbb{P}_2}$ is continuous, (see the subsection \ref{ssez:matrix-seq}). For the second part we observe that the result is known for Toeplitz matrices with Hermitian valued symbols \cite{marko}: then the localization  result for the eigenvalues of $A_n(1,\Omega,\mathbb{P}_2)$ follows, because   $A_n(1,\Omega,\mathbb{P}_2)$ is a principal submatrix of $T_n(\mathbf{f}_{\mathbb{P}_2})$ and since all the involved matrices are Hermitian.
\end{proof}
%-------------------------------------------------------------------------------
\begin{figure}
\centering
\includegraphics[width=0.7\textwidth]{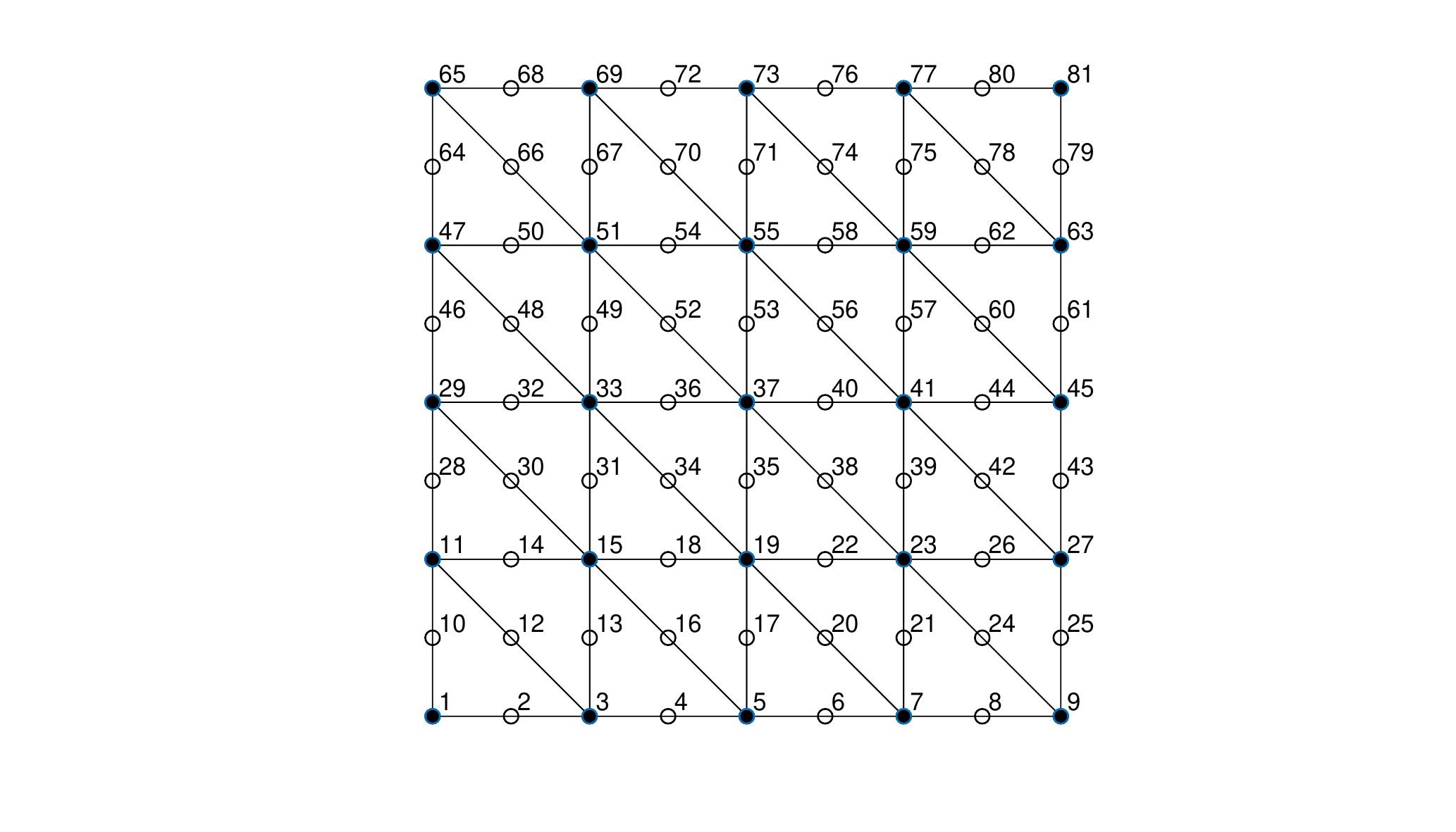}
\includegraphics[width=0.7\textwidth]{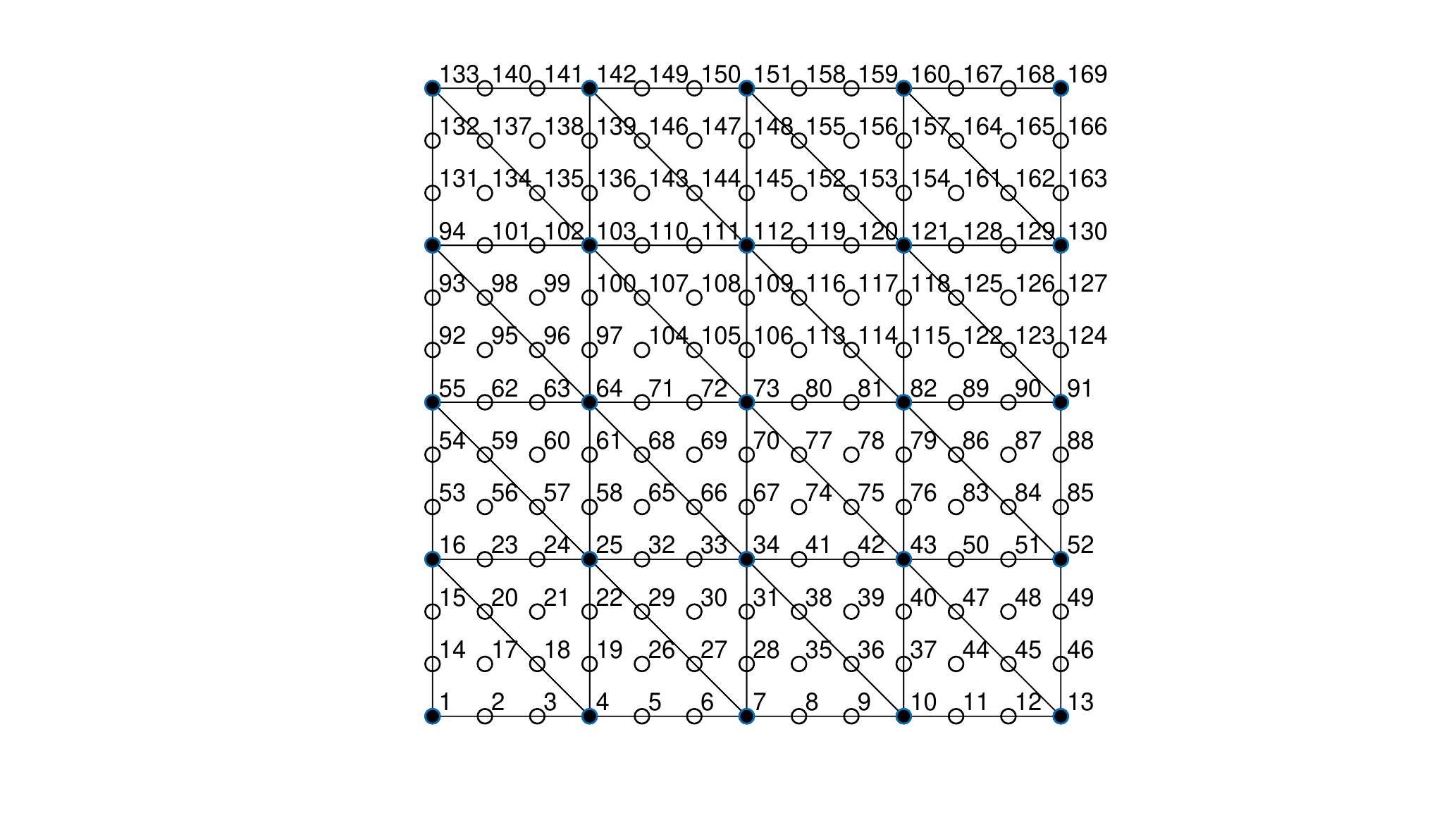}
\caption{Nodal points reordering in $\mathbb{P}_2$ and
$\mathbb{P}_3$ cases.} \label{fig:ordinamentoFK_P2}
\end{figure}
\subsection{Case $k=3$}

In the case of cubic Lagrangian FE ($k=3$), by referring to the reference element, (see Figure \ref{fig:elementodiriferimento}), we have the following shape functions
\begin{eqnarray*}
\hat{\varphi}_1(\hat{x}, \hat{y})&=& - \frac{9}{2}\hat{x}^3 - \frac{27}{2}\hat{x}^2\hat{y} + 9\hat{x}^2 - \frac{27}{2}\hat{x} \hat{y}^2 + 18\hat{x} \hat{y}
                - \frac{11}{2}\hat{x}- \frac{9}{2}\hat{y}^3 + 9\hat{y}^2 - \frac{11}{2}\hat{y} + 1,\nonumber\\
\hat{\varphi}_2(\hat{x}, \hat{y})&=& \frac{\hat{x}}{2} (9\hat{x}^2 - 9\hat{x} + 2),  \nonumber \\
\hat{\varphi}_3(\hat{x}, \hat{y})&=& \frac{\hat{y}}{2} (9\hat{y}^2 - 9\hat{y} + 2), \nonumber \\
\hat{\varphi}_4(\hat{x}, \hat{y})&=& \frac{9}{2}\hat{x} \hat{y}(3\hat{x} - 1), \nonumber\\
\hat{\varphi}_5(\hat{x}, \hat{y})&=& \frac{9}{2}\hat{x} \hat{y}(3\hat{y} - 1),  \label{eq:baseP32D} \\
\hat{\varphi}_6(\hat{x}, \hat{y})&=& -\frac{9}{2}\hat{y} (3\hat{y} - 1)(\hat{x} + \hat{y} - 1),  \nonumber \\
\hat{\varphi}_7(\hat{x}, \hat{y})&=& \frac{9}{2}\hat{y} (3\hat{x}^2 + 6\hat{x} \hat{y} - 5\hat{x} + 3\hat{y}^2 - 5\hat{y} + 2),  \nonumber \\
\hat{\varphi}_8(\hat{x}, \hat{y})&=& \frac{9}{2}\hat{x} (3\hat{x}^2 + 6\hat{x} \hat{y} - 5\hat{x} + 3\hat{y}^2 - 5\hat{y} + 2),  \nonumber \\
\hat{\varphi}_9(\hat{x}, \hat{y})&=& -\frac{9}{2}\hat{x} (3\hat{x} - 1)(\hat{x} + \hat{y} - 1),  \nonumber \\
\hat{\varphi}_{10}(\hat{x}, \hat{y})&=&  -27\hat{x} \hat{y}(\hat{x} + \hat{y} - 1), \nonumber
\end{eqnarray*}
so that, according to (\ref{eq:int_ij}), the
{elemental} matrix for a generic triangle equals
\begin{equation} \label{eq:Ael_k3} \renewcommand{\arraystretch}{1.2}
\!\!\! A_{K_1}^{El} =  \left [\begin{array}{rrrrrrrrrr}
  \frac{17}{20}   & -\frac{7}{80}    & -\frac{7}{80}    & -\frac{3}{40}  & -\frac{3}{40}  &  \frac{3}{8}     & -\frac{51}{80}   & -\frac{51}{80}   &  \frac{3}{8}     & \mbox{\small 0}\\
 -\frac{7}{80}    &  \frac{17}{40}   & \mbox{\small 0}  &  \frac{3}{80}  &  \frac{3}{80}  & -\frac{3}{80}    & -\frac{3}{80}    &  \frac{27}{80}   & -\frac{27}{40}   & \mbox{\small 0}\\
 -\frac{7}{80}    &  \mbox{\small 0} & \frac{17}{40}    &  \frac{3}{80}  &  \frac{3}{80}  & -\frac{27}{40}   &  \frac{27}{80}   & -\frac{3}{80}    & -\frac{3}{80}    & \mbox{\small 0}\\
 -\frac{3}{40}    &  \frac{3}{80}    & \frac{3}{80}     &  \frac{27}{8}  & -\frac{27}{40} &  \frac{27}{80}   &  \frac{27}{80}   &  \frac{27}{80}   & -\frac{27}{16}   & -\frac{81}{40}\\
 -\frac{3}{40}    &  \frac{3}{80}    & \frac{3}{80}     & -\frac{27}{40} &  \frac{27}{8}  & -\frac{27}{16}   &  \frac{27}{80}   &  \frac{27}{80}   &  \frac{27}{80}   & -\frac{81}{40}\\
  \frac{3}{8}     & -\frac{3}{80}    & -\frac{27}{40}   &  \frac{27}{80} & -\frac{27}{16} &  \frac{27}{8}    & -\frac{27}{16}   &  \mbox{\small 0} &  \mbox{\small 0} &  \mbox{\small 0}\\
 -\frac{51}{80}   & -\frac{3}{80}    &  \frac{27}{80}   &  \frac{27}{80} &  \frac{27}{80} & -\frac{27}{16}   &  \frac{27}{8}    &  \mbox{\small 0} &  \mbox{\small 0} & -\frac{81}{40}\\
 -\frac{51}{80}   &  \frac{27}{80}   &  -\frac{3}{80}   &  \frac{27}{80} &  \frac{27}{80} &  \mbox{\small 0} &  \mbox{\small 0} &  \frac{27}{8}    & -\frac{27}{16}   & -\frac{81}{40}\\
  \frac{3}{8}     & -\frac{27}{40}   &  -\frac{3}{80}   & -\frac{27}{16} &  \frac{27}{8}  &  \mbox{\small 0} &  \mbox{\small 0} & -\frac{27}{16}   &  \frac{27}{8}    & \mbox{\small 0}\\
  \mbox{\small 0} &  \mbox{\small 0} &  \mbox{\small 0} & -\frac{81}{40} & -\frac{81}{40} &  \mbox{\small 0} & -\frac{81}{40}   & -\frac{81}{40}   &  \mbox{\small 0} & \frac{81}{10}\\
\end{array} \right ]
\end{equation}
in the case of triangles of type 1, or a suitable permutation in the case of triangles of type 2.
The stiffness matrix $A_n=A_n(1,\Omega,\mathbb{P}_3)$ shows again
a block tridiagonal structure
$A_n=\mathrm{tridiag}(A_1,A_0,A_{-1})$ as in previous cases, the
higher approximation stressing its influence just inside the
blocks $A_i$. In fact, we observe in each block $A_i$ a $3\times
3$ block structure as  follows:
\begin{equation}
A_0=\left [\begin{array}{ccc}
   B_0^{11}      & B_0^{12} & B_0^{13}\\
   (B_0^{12})^T  & B_0^{22} & B_0^{23}\\
   (B_0^{13})^T  & B_0^{23} & B_0^{33}\\
\end{array} \right ], \quad
A_{-1}=\left [\begin{array}{ccc}
   0 & 0 & 0\\
   0 & 0 & 0\\
   B_{-1}^{31}  & B_{-1}^{32} & B_{-1}^{33}\\
\end{array} \right ], \quad A_{1} = A_{-1}^T.
\end{equation}

More important, the very same structure  appears in every block $B_l^{ij}$ by considering suitable $3\times 3$ matrices and
indeed we have
\begin{equation*}
B_l^{ij}=\mathrm{tridiag}\left(a_{1}^{B_l^{ij}},a_0^{B_l^{ij}},a_{-1}^{B_l^{ij}}
\right ), \quad l\in \{-1, 0, 1\}, \ i,j \in \{1, 2, 3\},
\end{equation*}
where
\begin{equation*}\renewcommand{\arraystretch}{1.2}
a_0^{B_0^{11}} = \left [\begin{array}{rrr}
   \frac{81}{10}   & -\frac{81}{40} & \mbox{\small 0} \\
  -\frac{81}{40}   & \frac{27}{4}   & -\frac{27}{16}\\
  \mbox{\small 0}  & -\frac{27}{16} & \frac{27}{4}\\
\end{array} \right ], \quad
a_{-1}^{B_0^{11}}  =\left [\begin{array}{rrr}
   \mbox{\small 0}  & \mbox{\small 0} & \mbox{\small 0}\\
   \mbox{\small 0}  & \mbox{\small 0} & \mbox{\small 0}\\
  -\frac{81}{40} & \frac{27}{80} & \mbox{\small 0} \\
\end{array} \right ], \quad
a_{1}^{B_0^{11}} = \left (a_{-1}^{B_0^{11}}\right )^T,
\end{equation*}
%-----------------
\begin{equation*}\renewcommand{\arraystretch}{1.2}
a_0^{B_0^{22}} = \left [\begin{array}{rrr}
  \frac{27}{4}   & -\frac{81}{40} &  \frac{27}{80}\\
 -\frac{81}{40}  &  \frac{81}{10}& -\frac{81}{40}\\
  \frac{27}{80}  & -\frac{81}{40} & \frac{27}{4}\\
\end{array} \right ], \quad
a_{-1}^{B_0^{22}}  =\left [\begin{array}{rrr}
   \mbox{\small 0}  & \mbox{\small 0} & \mbox{\small 0}\\
   \mbox{\small 0}  & \mbox{\small 0} & \mbox{\small 0}\\
   -\frac{27}{16} & 0 & 0 \\
\end{array} \right ], \quad
a_{1}^{B_0^{22}} = \left (a_{-1}^{B_0^{22}}\right )^T,
\end{equation*}
%-----------------
\begin{equation*}\renewcommand{\arraystretch}{1.2}
a_0^{B_0^{33}} = \left [\begin{array}{rrr}
  \frac{27}{4} &  -\frac{27}{8} & \frac{57}{80} \\
 -\frac{27}{8} &  \frac{27}{4}  & -\frac{21}{16}\\
 \frac{57}{80} &  -\frac{21}{16}& \frac{17}{5}\\
\end{array} \right ], \quad
a_{-1}^{B_0^{33}}  =\left [\begin{array}{rrr}
   \mbox{\small 0}  & \mbox{\small 0} & \mbox{\small 0}\\
   \mbox{\small 0}  & \mbox{\small 0} & \mbox{\small 0}\\
   -\frac{21}{16} &  \frac{57}{80} &  -\frac{7}{40} \\
\end{array} \right ], \quad
a_{1}^{B_0^{33}} = \left (a_{-1}^{B_0^{33}}\right )^T,
\end{equation*}
%-----------------
\begin{equation*}\renewcommand{\arraystretch}{1.2}
a_0^{B_0^{12}} = \left [\begin{array}{rrr}
     -\frac{81}{40}  & \mbox{\small 0} & \mbox{\small 0} \\
     -\frac{27}{20}  & -\frac{81}{40} &  \frac{27}{80}\\
      \frac{27}{80}  &  \mbox{\small 0} & -\frac{27}{8}\\
\end{array} \right ], \quad
a_{-1}^{B_0^{12}}  =\left [\begin{array}{rrr}
   \mbox{\small 0}  & \mbox{\small 0} & \mbox{\small 0}\\
   \mbox{\small 0}  & \mbox{\small 0} & \mbox{\small 0}\\
   \frac{27}{80} & \mbox{\small 0}  & \mbox{\small 0} \\
\end{array} \right ], \quad
a_{1}^{B_0^{12}} = \left [\begin{array}{rrr}
   \mbox{\small 0}  & \mbox{\small 0} & \mbox{\small 0}\\
   \mbox{\small 0}  & \mbox{\small 0} & \frac{27}{80}\\
   \mbox{\small 0}  & \mbox{\small 0} & \mbox{\small 0}   \\
\end{array} \right ],
\end{equation*}
%-----------------
\begin{equation*}\renewcommand{\arraystretch}{1.2}
a_0^{B_0^{13}} = \left [\begin{array}{rrr}
   \mbox{\small 0}             & \mbox{\small 0} &  \mbox{\small 0}\\
   \frac{27}{80} & \frac{27}{80} & -\frac{3}{40}\\
   \mbox{\small 0}             & \mbox{\small 0}             & \frac{57}{80}\\
\end{array} \right ], \quad
a_{-1}^{B_0^{13}}  = O_{3}, \quad a_{1}^{B_0^{13}} = \left
[\begin{array}{rrr}
  \mbox{\small 0}  & \mbox{\small 0} & \mbox{\small 0}\\
   \mbox{\small 0}  & \mbox{\small 0} & \frac{3}{40}\\
   \mbox{\small 0}  & \mbox{\small 0} & -\frac{3}{80} \\
\end{array} \right ],
\end{equation*}
%-----------------
\begin{equation*}\renewcommand{\arraystretch}{1.2}
a_0^{B_0^{23}} = \left [\begin{array}{rrr}
   -\frac{27}{16} &  \frac{27}{80} & -\frac{3}{40} \\
   \mbox{\small 0}              & -\frac{81}{40} & \mbox{\small 0} \\
   \mbox{\small 0}             &  \mbox{\small 0} & -\frac{21}{16}\\
\end{array} \right ], \quad
a_{-1}^{B_0^{23}}  = O_{3}, \quad a_{1}^{B_0^{23}} = \left
[\begin{array}{rrr}
   \mbox{\small 0}  & \mbox{\small 0} & \frac{3}{40}\\
   \mbox{\small 0}  & \mbox{\small 0} & \mbox{\small 0}\\
   \mbox{\small 0}  & \mbox{\small 0} & -\frac{3}{80} \\
\end{array} \right ],
\end{equation*}
%-----------------
\begin{equation*}\renewcommand{\arraystretch}{1.2}
a_0^{B_{-1}^{31}} = \left [\begin{array}{rrr}
   -\frac{81}{40}   &  \frac{27}{80}  &  \mbox{\small 0} \\
  \mbox{\small 0}                &  -\frac{27}{16} &  \mbox{\small 0} \\
   \mbox{\small 0}                &  \frac{3}{40}   & -\frac{21}{16}\\
\end{array} \right ], \quad
a_{-1}^{B_{-1}^{31}}  =\left [\begin{array}{rrr}
   \mbox{\small 0}  & \mbox{\small 0} & \mbox{\small 0}\\
   \mbox{\small 0}  & \mbox{\small 0} & \mbox{\small 0}\\
   \mbox{\small 0}  & -\frac{3}{40} & \mbox{\small 0}  \\
\end{array} \right ], \quad
a_{1}^{B_{-1}^{31}} = \left [\begin{array}{rrr}
   \mbox{\small 0}  & \mbox{\small 0} & \mbox{\small 0}\\
   \mbox{\small 0}  & \mbox{\small 0} & \mbox{\small 0}\\
   \mbox{\small 0}  & \mbox{\small 0} & -\frac{3}{80} \\
\end{array} \right ],
\end{equation*}
%-----------------
\begin{equation*}\renewcommand{\arraystretch}{1.2}
a_0^{B_{-1}^{32}} = \left [\begin{array}{rrr}
   \frac{27}{80}  & \mbox{\small 0} &  \mbox{\small 0} \\
   \frac{27}{80}  & \mbox{\small 0} &  \mbox{\small 0}\\
   \frac{3}{40}   & \mbox{\small 0} & \frac{57}{80}\\
\end{array} \right ], \quad
a_{-1}^{B_{-1}^{32}}  =\left [\begin{array}{rrr}
   \mbox{\small 0}  & \mbox{\small 0} & \mbox{\small 0}\\
   \mbox{\small 0}  & \mbox{\small 0} & \mbox{\small 0}\\
  - \frac{3}{40}& \mbox{\small 0} &  \mbox{\small 0}\\
\end{array} \right ], \quad
a_{1}^{B_{-1}^{32}} = \left [\begin{array}{rrr}
   \mbox{\small 0}  & \mbox{\small 0} & \mbox{\small 0}\\
   \mbox{\small 0}  & \mbox{\small 0} & \mbox{\small 0}\\
   \mbox{\small 0}  & \mbox{\small 0} & -\frac{3}{80}  \\
\end{array} \right ],
\end{equation*}
%-----------------
\begin{equation*}\renewcommand{\arraystretch}{1.2}
a_0^{B_{-1}^{33}} = \left [\begin{array}{rrr}
   \mbox{\small 0}  & \mbox{\small 0} & \mbox{\small 0}\\
   \mbox{\small 0}  & \mbox{\small 0} & \mbox{\small 0}\\
   -\frac{3}{80}  & -\frac{3}{80} & -\frac{7}{40}\\
\end{array} \right ], \quad
a_{-1}^{B_{-1}^{33}}  = O_{3}, \quad a_{1}^{B_{-1}^{33}} = \left
[\begin{array}{rrr}
   \mbox{\small 0}  & \mbox{\small 0} & -\frac{3}{80}\\
   \mbox{\small 0}  & \mbox{\small 0} & -\frac{3}{80}\\
   \mbox{\small 0}  & \mbox{\small 0} & \mbox{\small 0} \\
\end{array} \right ].
\end{equation*}
%-----------------
Thus, once again, just by taking into account that we are now facing a
matrix-valued symbol,  we can easily read the underlying symbol as
 follows:
 \begin{equation} \label{eq:simbolotridiag_P3}
\mathbf{f}_{\mathbb{P}_3}(\theta_{1},\theta_{2})=f_{A_0}(\theta_{1})+f_{A_{-1}}(\theta_{1}) e^{-\hat{\imath} \theta_{2}} +f_{A_{1}}(\theta_{1}) e^{\hat{\imath} \theta_{2}},
\end{equation}
with
\begin{eqnarray*}
f_{A_0}(\theta_{1}) & = & \left [\begin{array}{ccc}
   f_{B_0^{11}}(\theta_{1})      & f_{B_0^{12}}(\theta_{1})     & f_{B_0^{13}}(\theta_{1})\\
   f_{(B_0^{12})^T }(\theta_{1}) & f_{B_0^{22}}(\theta_{1})     & f_{B_0^{23}}(\theta_{1})\\
   f_{(B_0^{13})^T }(\theta_{1}) & f_{(B_0^{23})^T}(\theta_{1}) & f_{B_0^{33}}(\theta_{1})\\
\end{array} \right ], \
\begin{array}{l}
f_{B_l^{ij}}(\theta_{1}) = a_0^{B_l^{ij}} +a_{-1}^{B_l^{ij}} e^{-\hat{\imath} \theta_{1}}+a_{1}^{B_l^{ij}} e^{\hat{\imath} \theta_{1}},\\
f_{(B_l^{ij})^T}(\theta_{1}) = \overline{f_{B_l^{ij}}(\theta_{1})},
\end{array}
\\%
%-------------
%
f_{A_{-1}}(\theta_{1}) &=&\left [\begin{array}{ccc}
   0      & 0 & 0\\
   0      & 0 & 0\\
   f_{B_{-1}^{31}}(\theta_{1}) & f_{B_{-1}^{32}}(\theta_{1}) & f_{B_{-1}^{33}}(\theta_{1}) \\
\end{array} \right ], \quad
f_{A_1}(\theta_{1}) =\left [\begin{array}{ccc}
   0  & 0  & f_{(B_{-1}^{31})^T}(\theta_{1}) \\
   0  & 0  & f_{(B_{-1}^{32})^T}(\theta_{1}) \\
   0  & 0  & f_{(B_{-1}^{33})^T}(\theta_{1}) \\
\end{array} \right ].
\end{eqnarray*}
To sum up, we find the expression of  $\mathbf{f}_{\mathbb{P}_3}: [-\pi,\pi]^2 \longrightarrow \mathbb{C}^{9\times 9}$ with
%-----*****----
\begin{equation}\label{eq:P3symbol_c} \renewcommand{\arraystretch}{.2}
\hspace{-2.5cm}\mathbf{f}_{\mathbb{P}_3}(\theta_{1},\theta_{2})  = { \footnotesize \left [
\begin{array}{rrr|rrr|rrr}
%1
\alpha & -\frac{\alpha}{4} & -\frac{\alpha}{4}e^{\hat{\imath} \theta_{1}} & -\frac{\alpha}{4} & 0 & 0 & -\frac{\alpha}{4}e^{\hat{\imath} \theta_{2}} & 0 & 0 \\
%2
-\frac{\alpha}{4} & \beta & -\frac{\beta}{4}+\frac{\beta}{20}e^{\hat{\imath} \theta_{1}} & -\frac{\beta}{5} & -\frac{\alpha}{4}
     & \frac{\beta}{20}(1+e^{\hat{\imath} \theta_{1}}) & \frac{\beta}{20}(1+e^{\hat{\imath} \theta_{2}})&  \frac{\beta}{20}-\frac{\beta}{4}e^{\hat{\imath} \theta_{2}} & f_{29}\\
%3
 -\frac{\alpha}{4}e^{-\hat{\imath} \theta_{1}} & -\frac{\beta}{4}+\frac{\beta}{20}e^{-\hat{\imath} \theta_{1}} & \beta & \frac{\beta}{20}(1+e^{-\hat{\imath} \theta_{1}}) & 0 & -\frac{\beta}{2} & 0 & 0
     & f_{39}\\
\hline
%-----4
-\frac{\alpha}{4} &  -\frac{\beta}{5}&  \frac{\beta}{20}(1+e^{\hat{\imath} \theta_{1}}) & \beta & -\frac{\alpha}{4} & \frac{\beta}{20}-\frac{\beta}{4}e^{\hat{\imath} \theta_{1}}
     & -\frac{\beta}{4}+\frac{\beta}{20}e^{\hat{\imath} \theta_{2}} & \frac{\beta}{20}(1+e^{\hat{\imath} \theta_{2}}) & f_{49} \\
%5
0 & -\frac{\alpha}{4} & 0  & -\frac{\alpha}{4} & \alpha & -\frac{\alpha}{4} & 0 & -\frac{\alpha}{4} & 0\\
%6
0 & \frac{\beta}{20}(1+e^{-\hat{\imath} \theta_{1}})  & -\frac{\beta}{2}& \frac{\beta}{20}-\frac{\beta}{4}e^{-\hat{\imath} \theta_{1}} & -\frac{\alpha}{4}  & \beta
     & 0 & 0 & f_{69}\\
\hline
%-----7
-\frac{\alpha}{4}e^{-\hat{\imath} \theta_{2}} & \frac{\beta}{20}(1+e^{-\hat{\imath} \theta_{2}}) & 0 & -\frac{\beta}{4}+\frac{\beta}{20}e^{-\hat{\imath} \theta_{2}}& 0 & 0 & \beta & -\frac{\beta}{2}
     & f_{79}\\
%8
0 & \frac{\beta}{20}-\frac{\beta}{4}e^{-\hat{\imath} \theta_{2}} & 0 & \frac{\beta}{20}(1+e^{-\hat{\imath} \theta_{2}}) & -\frac{\alpha}{4} & 0 & -\frac{\beta}{2}& \beta
     & f_{89}\\
%9
0 & \overline{f_{29}} & \overline{f_{39}} & \overline{f_{49}} & 0 & \overline{f_{69}} & \overline{f_{79}} & \overline{f_{89}} & f_{99}\\
\end{array}
\right ]}
\end{equation}
where
\[\renewcommand{\arraystretch}{1.4}
\begin{array}{ll}
f_{29}= -\gamma(1-e^{\hat{i} \theta_{1}})(1-e^{\hat{i} \theta_{2}}), & f_{39}=  \delta -\frac{\gamma}{2}e^{\hat{i} \theta_{1}}-\varepsilon e^{\hat{i} \theta_{2}} -\frac{\gamma}{2}e^{-\hat{i} \theta_{1}}e^{\hat{i} \theta_{2}},\\
f_{49}= -\gamma(1-e^{\hat{i} \theta_{1}})(1-e^{\hat{i} \theta_{2}}), & f_{69}= -\varepsilon -\frac{\gamma}{2}e^{\hat{i} \theta_{1}}+\delta e^{\hat{i} \theta_{2}} -\frac{\gamma}{2}e^{-\hat{i} \theta_{1}}e^{\hat{i} \theta_{2}},\\
f_{79}= \delta -\varepsilon e^{\hat{i} \theta_{1}}-\frac{\gamma}{2}e^{\hat{i} \theta_{1}}e^{-\hat{i} \theta_{2}}-\frac{\gamma}{2}e^{\hat{i} \theta_{2}},  &   f_{89}= -\varepsilon +\delta e^{\hat{i} \theta_{1}}-\frac{\gamma}{2}e^{\hat{i} \theta_{1}}e^{-\hat{i} \theta_{2}}-\frac{\gamma}{2}e^{\hat{i} \theta_{2}},\\
%f_{92}=-\gamma(1-e^{-\hat{i} s})(1-e^{-\hat{i} t}), & f_{93}=\delta -\frac{\gamma}{2}e^{-\hat{i} s}-\varepsilon e^{-\hat{i} t} -\frac{\gamma}{2}e^{\hat{i} s}e^{-\hat{i} t},\\
%f_{94}=-\gamma(1-e^{-\hat{i} s})(1-e^{-\hat{i} t}), & f_{96}= \varepsilon -\frac{\gamma}{2}e^{-\hat{i} s}+\delta e^{-\hat{i} t} -\frac{\gamma}{2}e^{\hat{i} s}e^{-\hat{i} t},    \\
%f_{97}=\delta -\varepsilon e^{-\hat{i} s} -\frac{\gamma}{2}e^{-\hat{i} s}e^{\hat{i} t}-\frac{\gamma}{2}e^{-\hat{i} t}, & f_{98}=-\varepsilon +\delta e^{-\hat{i} s} -\frac{\gamma}{2}e^{-\hat{i} s}e^{\hat{i} t}-\frac{\gamma}{2}e^{-\hat{i} t}, \\
f_{99}= \zeta -2\eta (\cos(\theta_{1})+\cos(\theta_{2})),\\
\end{array}
\]
 and $\alpha={81}/{10}$, $\beta={27}/{4}$, $\gamma={3}/{40}$,
$\delta={57}/{80}$, $\varepsilon={21}/{16}$,
$\zeta={17}/{5}$, $\eta={7}/{40}$.\\
Finally, the stiffness matrix $A_n=A_n(1,\Omega,\mathbb{P}_3)$ is
a principal submatrix of a suitable permutation of the Toeplitz
matrix $T_n(\mathbf{f}_{\mathbb{P}_3})$. In order to obtain the stiffness
matrix $A_n$ from $T_n(\mathbf{f}_{\mathbb{P}_3})$ it is enough to group
internal nodal values nine by nine. Again by referring to Theorem
\ref{extradimensional}, the following proposition holds.
\begin{Proposition}\label{distr-P_3}
The two matrix-sequences $\{T_n(\mathbf{f}_{\mathbb{P}_3})\}_n$ and $\{A_n(1,\Omega,\mathbb{P}_3)\}_n$  are spectrally distributed as
$\mathbf{f}_{\mathbb{P}_3}$ in the sense of Definition \ref{def-distribution}.
\end{Proposition}
As an immediate consequence of Proposition \ref{distr-P_3} we
deduce a corollary regarding the clustering  and localization of
the spectra of $\{T_n(\mathbf{f}_{\mathbb{P}_3})\}_n$ and
$\{A_n(1,\Omega,\mathbb{P}_3)\}_n$ as well.
\begin{Corollary}\label{cor:distr-P_3}
The range of $\mathbf{f}_{\mathbb{P}_3}$ is a weak cluster set for the spectra of the two matrix-sequences $\{T_n(\mathbf{f}_{\mathbb{P}_3})\}_n$ and $\{A_n(1,\Omega,\mathbb{P}_3)\}_n$ in the sense of Definition \ref{def-cluster}. Furthermore, the convex hull of the range of $\mathbf{f}_{\mathbb{P}_3}$ contains all the eigenvalues of the involved matrices.
\end{Corollary}
\begin{proof}
The proof of the first part is a direct consequence  of Proposition \ref{distr-P_3}, taking into account \cite[Theorem 4.2]{gol-serra} and observing that in this setting the standard range and the essential range coincide since $\mathbf{f}_{\mathbb{P}_3}$ is continuous, (see also the Subsection \ref{ssez:matrix-seq}). For the second part we observe that the result is known for Toeplitz matrices with Hermitian valued symbols \cite{marko}: then the localization  result for the eigenvalues of $A_n(1,\Omega,\mathbb{P}_3)$ follows, because   $A_n(1,\Omega,\mathbb{P}_3)$ is a principal submatrix of $T_n(\mathbf{f}_{\mathbb{P}_3})$ and since all the involved matrices are Hermitian.
The thesis follows with the same reasoning considered in Corollary
\ref{cor:distr-P_2}.
\end{proof}
\section{Symbol spectral analysis}\label{sez:SA-Pk-2D}
We start the spectral analysis of symbols obtained in the previous section from a numerical point of view. As well known,
$\mathbf{f}_{\mathbb{P}_1}$ shows a  zero of order 2 in $(0,0)$, while it is positive elsewhere. %{\color{red}
{In the case $k\ge 2$ the symbol is a matrix-valued function, so
we consider an equispaced sampling in $[-\pi,\pi]^2$ of the symbol
and for each point we evaluate the $k^2$   eigenvalues, ordering
them in  non-decreasing way. {Thus,  $k^2$ surfaces are defined by $s_i$,
$i=1,\ldots,k^2$ and the $i^{th}$ eigenvalue in a given point of the sampling is given by the evaluation
of the surface $s_i$ in such a point}. In Table \ref{tab:minmaxPk}  the minimal and
maximal values of each surface $s_i$, $i=1,\ldots,k^2$, are
reported (for a comparison among the surfaces obtained by using
the eigenvalues of the considered matrix-sequence and the
corresponding surfaces obtained by properly sampling the symbol of
the same matrix-sequence, see the subsequent Figures
\ref{fig:es0}-\ref{fig:es3}).
\begin{table}
\small \centering
\scalebox{0.85}{\begin{tabular}{|c|cc|cc|}
\hline
$i$ & $\min(s_i)$ & $\mathrm{argmin}(s_i)$ & $\max(s_i)$  & $\mathrm{argmax}(s_i)$\\
\hline
\multicolumn{5}{|c|}{$k=1$} \\
\hline
1 & 0 & $(0,0)$ & 8 & $(-\pi,-\pi)$ \\
\hline
\multicolumn{5}{|c|}{$k=2$} \\
\hline
1   & -2.122988181725368e-17 & $(0,0)$       &  2.666666666666667e+00 & $(-\pi,-\pi)$\\
2   &  2.666666666666667e+00 & $(0,-\pi)$    &  5.333333333333330e+00 & $(0,0)$\\
3   &  5.333333333333325e+00 & $(-u,-u)$     &  7.415403750411773e+00 & $(0,-\pi)$ \\
4   &  5.333333333333333e+00 & $(-\pi,-\pi)$ &  1.066666666666667e+01 & $(0,0)$ \\
\hline
\multicolumn{5}{|c|}{$k=3$} \\
\hline
1   & -2.947870832408496e-16  & $(0,0)$       &   1.752299219210445e+00 & $(-\pi,-\pi)$ \\
2   &  1.077001420967619e+00  & $(-\pi,0)$    &   2.649326100400095e+00 & $(0,0)$ \\
3   &  2.024999999999999e+00  & $(\pi,-\pi)$  &   3.374999999999998e+00 & $(0,0)$ \\
4   &  2.417725227846304e+00  & $(-\pi,-\pi)$ &   5.473900873539699e+00 & $(0,0)$ \\
5   &  6.074999999999998e+00  & $(0,0)$       &   8.150826984062711e+00 & $(-v,v)$\\
6   &  6.075000000000001e+00  & $(0,0)$       &   9.450000000000005e+00 & $(-\pi,-\pi)$\\
7   &  8.100000000000000e+00  & $(-\pi,0)$    &   1.145177302606021e+01 & $(0,0)$ \\
8   &  1.012500000000000e+01  & $(-\pi,-\pi)$ &   1.306461248424784e+01 & $(-\pi,0)$\\
9   &  1.215000000000001e+01  & $(0,0)$       &   1.542003087979332e+01 & $(-\pi,-\pi)$\\
\hline
\multicolumn{5}{|c|}{$k=4$} \\
\hline
1   & 8.665811124242140e-15  &$(0,0)$        &   1.154132889535501e+00 & $(\pi,-\pi)$  \\
2   & 6.091179158637314e-01  &$(-\pi,0)$     &   2.028216383055356e+00 &  $(-\pi,-\pi)$ \\
3   & 1.183562035003593e+00  & $(w,-w)$      &   2.278229389751864e+00 &  $(-\pi,0)$  \\
4   & 1.228189268889777e+00  & $(-\pi,-\pi)$ &   2.796565325232735e+00 & $(z,\pi/10)$  \\
5   & 2.706589845271391e+00  & $(0,0)$       &   3.280192294561743e+00 &  $(-\pi,-\pi)$ \\
6   & 3.100532625333826e+00  & $(0,0)$       &   4.876190476190478e+00 &  $(\pi,-\pi)$ \\
7   & 4.086126343234464e+00  & $(a,b)$       &   5.001164336911962e+00 &  $(c,-c)$ \\
8   & 4.923102258884252e+00  & $(d,-d)$      &   6.507154754218933e+00 &  $(0,-\pi)$ \\
9   & 6.351060427971650e+00  & $(e,e)$       &   7.524544180802064e+00 &  $(f,-f)$ \\
10  & 1.124369260315089e+01  & $(-\pi,0)$    &   1.277464393947364e+01 & $(\pi,-\pi)$   \\
11  & 1.221231815611003e+01  & $(-g,g)$      &   1.319265448651837e+01 & $(0,-\pi)$   \\
12  & 1.312292935024724e+01  & $(h,h)$       &   1.551857649581396e+01 & $(i,-i)$   \\
13  & 1.403908928314477e+01  & $(\pi,-\pi)$  &   1.715087064330462e+01 & $(l,-l)$   \\
14  & 1.715307867863427e+01  & $(l,-l)$      &   2.041132693707404e+01 & $(-\pi,-\pi)$   \\
15  & 1.987073604165514e+01  & $(0,0)$       &   2.261907577519149e+01 & $(0,\pi)$   \\
16  & 2.236934933910699e+01  & $(m,-m)$ &  2.492211941947813e+01  & $(0,0)$   \\
 \hline
\end{tabular}}
\caption{Minimum and maximum of  surfaces \small{$s_i$, $i=1,\ldots,k^2$.
\quad \quad $u=-7.351326809400116e-01$, $v=2.500707752257475e+00$,
$w=3.053628059289279e+00$, $z=1.734159144781565e+00$,
$a=2.576105975943630e-01$, $b=-2.224247598741574e+00$,
$c=2.896548426609789e+00$, $d=1.507964473723100e+00$,
$e=1.043008760991811e+00$, $f=2.161415745669778e+00$,
$g=1.627344994559513e+00$, $h=2.796017461694915e+00$,
$i=7.099999397112930e-01$, $l= 9.550441666912972e-01$,
$m=2.519557308179014e+00$.} } \label{tab:minmaxPk}
\end{table}
In the case $k=2$, it is worth stressing that the chosen sorting
of the eigenvalues influences the surfaces definition, the minimal
value of the $i$-th surface being lower of the maximal value of
the $(i-1)$-th surface: this implies that the union of the ranges
of the eigenvalue functions of the symbol produces a connected set
which is a  cluster for the spectra of the given matrix-sequence.
 When $k=3$ the union of the ranges of the first four surfaces  is well separated from the union of the remaining five surfaces and hence the cluster is divided into two sub-clusters in the sense of Definition \ref{def-cluster}. In the  case $k=4$, the union of the ranges of the first nine surfaces  is well separated from the union of the remaining seven surfaces and consequently, as in the case of $k=3$, the cluster is divided into two sub-clusters.

However, there  is a phenomenon which is expected and it is
independent of the value of $k$: only the first surface
 reaches zero as minimum, while all the other
surfaces are strictly positive everywhere.}
\par
Now we give a general result regarding the main features of the
involved symbols, with the proof in various cases, including both
${\mathbb{P}_k}$ and ${\mathbb{Q}_k}$ Finite Element
approximations.
\begin{Theorem}\label{theorem general}
Given the symbols $\mathbf{f}_{\mathbb{P}_k}, \mathbf{f}_{\mathbb{Q}_k}$ in
dimension $d\ge 1$, the following statements hold true. For every
$\mathbf{f}\in \{\mathbf{f}_{\mathbb{P}_k}, \mathbf{f}_{\mathbb{Q}_k}\}$, setting
\[
\lambda_1(\mathbf{f}(\boldsymbol{\theta}))\le \cdots \le \lambda_{k^d}(\mathbf{f}(\boldsymbol{\theta})),
\]
 we obtain
\begin{enumerate}
\item $f( \mathbf{0}) \mathbf{e}= \mathbf{0}$;
\item there exist constants
$C_1,C_2>0$ (dependent on $\mathbf{f}$) such that
\begin{equation}
C_1 \sum_{j=1}^{d} (2-2\cos(\theta_{j})) \le
\lambda_1(\mathbf{f}(\boldsymbol{\theta})) \le C_2 \sum_{j=1}^{d}
(2-2\cos(\theta_{j}));
\end{equation}
\item there exist constants $m,M>0$ (dependent on $\mathbf{f}$) such that
\begin{equation}
0 < m \le \lambda_j(\mathbf{f}(\boldsymbol{\theta})) \le M,\ \ \ \
j=2,\ldots,k^d.
\end{equation}
\end{enumerate}
For $\mathbf{f}_{\mathbb{Q}_k}$ the proof is given for every $k,d\ge 1$
(for $d=1$ we notice again that $\mathbf{f}_{\mathbb{P}_k}\equiv
\mathbf{f}_{\mathbb{Q}_k}$). For $\mathbf{f}_{\mathbb{P}_k}$ the proof is given for
$d=2$ and $k=2,3$.
\end{Theorem}
\begin{Remark}\label{rem general}
While in the case of $ \mathbb{Q}_k$ Finite Elements the analysis of the symbol $\mathbf{f}_{\mathbb{Q}_k}$ given in \cite{Q_k} and
in Theorem \ref{theorem general} is general, for the $\mathbb{P}_k$ Finite Elements  in dimension $d> 1$ there is still room for a substantial
improvement of the analysis and this will be a target in future researches.
\end{Remark}
\begin{proof}
$\blacksquare$ Case $\mathbb{Q}_k$ Finite Elements: any $k\ge 1$,
$d=1$. \\
Claims 2. and 3. have been proved in Theorem 8 and Corollary 1 in
\cite{Q_k}. Here, we prove Claim 1.: as first thing we  recall
that the
relation $\mathbf{f}( \mathbf{0}) \mathbf{e}= \mathbf{0}$,  % $0,e\in \mathbb{R}^{N(k)}$,
 with $\mathbf{e}$ vector of all ones and $k\ge 1$, is equivalent to say that every
row of $\mathbf{f}( \mathbf{0})$ is a vector having rowsum equal to zero.
We now show the latter feature.
Taking into consideration the notations in Section \ref{sez:Pk-1D},
we have
\[
\left(\mathbf{f}( \mathbf{0})\mathbf{e}\right)_s=
\displaystyle\sum_{j=1}^{k}\left(\mathbf{f}( \mathbf{0})\right)_{s,j}=
\displaystyle\sum_{j=1}^{k}\left(K_{0}+K_{1}+K_{1}^{T}\right)_{s,j}\hbox{,
}\ \ \ s=1,\ldots,k.
\]
 We first observe that the Lagrange
polynomial interpolating the constant $1$ is exactly equal to $1$,
by the uniqueness of the interpolant. Therefore
$\sum_{j=0}^{k}L_{j}=1$,
$\left(\sum_{j=0}^{k}L_{j}\right)'=\sum_{j=0}^{k}L'_{j}=0$, and
 hence, for $1\leq s\leq k-1,$
\begin{eqnarray*}
 %\no number to remove numbering (before each equation)
     \sum_{j=1}^{k}(\mathbf{f}( \mathbf{0}))_{s,j}&=&\sum_{j=1}^{k}\left<L'_{j},L'_{s}\right>+\left<L'_{0},L'_{s}\right>  \\
     &=&  \left<\sum_{j=0}^{k}L'_{j},L'_{s}\right>\\
     &=&\left<0,L'_{s}\right> =0.
  \end{eqnarray*}
Finally, for $s=k$ we  have
\begin{eqnarray*}
  \sum_{j=1}^{k}(\mathbf{f}( \mathbf{0}))_{k,j} &=& \sum_{j=1}^{k}\left<L'_{j},L'_{k}\right>+\left<L'_{0},L'_{0}\right>
   +\left<L'_{0},L'_{k}\right>+\sum_{j=1}^{k}\left<L'_{0},L'_{j}\right> \\
   &=&\sum_{j=0}^{k}\left<L'_{j},L'_{k}\right>+ \sum_{j=0}^{k}\left<L'_{0},L'_{j}\right>  \\
   &=& \left<\sum_{j=0}^{k}L'_{j},L'_{k}\right>+\left<L'_{0},\sum_{j=0}^{k}L'_{j}\right> =0,
\end{eqnarray*}
and consequently we conclude that $\mathbf{f}( \mathbf{0})\mathbf{e}=0.$
\par
%======================================================================================
$\blacksquare$ Case $\mathbb{Q}_k$ Finite Elements: any $k\ge 1$,
any $d\ge 2$. Claim 1. is a direct consequence of the proof for
$\mathbb{Q}_k$ Finite Elements, $k\ge 1$, and $d=1$, given its
tensorial structure, (see Formula (5.1) in \cite{Q_k}): in reality
it is sufficient to observe that $x\otimes y$ has rowsum equal to
zero if and only if either $x$ or $y$ has rowsum equal to zero,
with any $x$, $y$ complex vectors of any size. The case of more
than two vectors can be handled by an inductive argument.
Furthermore, Claims 2. and 3. are contained in Section 5.1 in
\cite{Q_k}.
\par
%======================================================================================
$\blacksquare$ Case $\mathbb{P}_k$ Finite Elements: $k=2$, $d=2$.
Claim 1. follows by direct check of the zero rowsum property from
the expression of the symbol $\mathbf{f}_{\mathbb{P}_k}$ in
(\ref{eq:P2symbol}), taking into account
$\boldsymbol{\theta}=(0,0)$ and the numerical values of the
involved parameters.
\par
%For proving Claims 2. and 3., we consider
%$\theta=(\theta_{1},\theta_{2})\in[-\pi,\pi]^{2}$.
%
Now, since the determinant of a matrix is the product of its
eigenvalues and since $\mathbf{f}$ is bounded in infinity norm, in order to
prove Claim 2. and 3. with $d=2$
%that is
%$$C_{1}\sum_{j=1}^{2}2-2\cos(\theta_{j})\leq\lambda_{1}(f_{\mathbb{P}_{2}}(\theta))\leq
%C_{2}\sum_{j=1}^{2}2-2\cos(\theta_{j}), $$
%$$
%m \le \lambda_j(f(\theta)) \le M,\ \ \ \ j=2,\ldots,4
%$$
 it is sufficient to show that:
 \begin{description}\begin{centering}
   \item[A.] $\textrm{ det}(\mathbf{f}(\boldsymbol{\theta}))\sim
   \displaystyle\sum_{j=1}^{2}(2-2\cos(\theta_{j}))$, \quad $\boldsymbol{\theta}=(\theta_1,\theta_2)$, \end{centering} %
  \begin{centering} \item[B.] there exists $C>0$ such that
  $\lambda_{2}(\mathbf{f}(\boldsymbol{\theta})) \geq    C>0$, \\
 \end{centering}\end{description}
 with $\mathbf{f}=\mathbf{f}_{\mathbb{P}_{2}}$.
\par We remind that the relation A. means there
   exist $C_{1},C_{2}>0 $ such that
\[
C_{1}\displaystyle\sum_{j=1}^{2}(2-2\cos(\theta_{j}))\leq
\textrm{det}(\mathbf{f}_{\mathbb{P}_{2}}(\boldsymbol{\theta}))\leq
C_{2}\sum_{j=1}^{2}(2-2\cos(\theta_{j}))
\]
uniformly in the domain $(\theta_{1},\theta_{2})\in[-\pi,\pi]^{2}$.
\par
By direct computation, we find
\begin{eqnarray*}
%0 \le
\textrm{ det} (\mathbf{f}_{\mathbb{P}_{2}}(\boldsymbol{\theta})) &=&
C'\left(-2\cos(\theta_{1})-2\cos(\theta_{2})-\cos(\theta_{1})\cos(\theta_{2})+5\right)
\\
&=& C'\left(\displaystyle\sum_{j=1}^{2}(2-2\cos(\theta_{j}))
+(1-\cos(\theta_{1})\cos(\theta_{2}))\right) \\
& \ge & C'\left(\displaystyle\sum_{j=1}^{2}(2-2\cos(\theta_{j}))
\right)
\end{eqnarray*} with
$C'={4096}/{81}$, being
$-1\leq\cos(\theta_{1})\cos(\theta_{2})\leq 1$ for all
$(\theta_{1},\theta_{2})\in[-\pi,\pi]^{2}$. Thus, $C_1=C'$.
%
%We remind that the relation det$(f_{\mathbb{P}_{2}}(\theta))\sim
%\displaystyle\sum_{j=1}^{2}2-2\cos(\theta_{j})$  is equivalent to
%find constants  $c_{1},c_{2}>0 $ such that
%$c_{1}\displaystyle\sum_{j=1}^{2}2-2\cos(\theta_{j})\leq {\rm
%det}(f_{\mathbb{P}_{2}}(\theta))\leq
%c_{2}\sum_{j=1}^{2}2-2\cos(\theta_{j}) $ uniformly in the domain
%$(\theta_{1},\theta_{2})\in[-\pi,\pi]^{2}$.
% Furthermore we have   $-1\leq\cos(\theta_{1})\cos(\theta_{2})\leq1 \mbox{ ,}\forall(\theta_{1},\theta_{2})\in[-\pi,\pi]^{2}$ and consequently
%  $$(\theta_{1},\theta_{2})\in[-\pi,\pi]^{2}(f_{\mathbb{P}_{2}}(\theta))\geq C'\sum_{j=1}^{2}2-2\cos(\theta_{j}).$$
\par
Furthermore, for $c=1/2$ it holds
$1-\cos(\theta_{1})\cos(\theta_{2}) \le c
\left(\displaystyle\sum_{j=1}^{2}(2-2\cos(\theta_{j})) \right )$
for all $(\theta_{1},\theta_{2})\in[-\pi,\pi]^{2}$. Thus,
$C_2=3C'/2$.
% For proving the second inequality it is sufficient  to show that $\exists c>0,$  for which $ 1-\cos(\theta_{1})\cos(\theta_{2})\leq c \sum_{j=1}^{2}2-2\cos(\theta_{j})$
% which is equivalent to the inequality
% i.e $4c-2c\cos(\theta_{1})-2c\cos(\theta_{2})+\cos(\theta_{1})\cos(\theta_{2})-1\geq0$.
%
% We have $$4c-2c\cos(\theta_{1})-2c\cos(\theta_{2})+\cos(\theta_{1})\cos(\theta_{2})-1=(\cos(\theta_{1})-2c)(\cos(\theta_{2})-2c)-4c^{2}+4c-1$$
% and $$-1-2c\leq\cos(\theta_{j})-2c\leq1-2c \mbox{ ,}j=1,2 \mbox{ }\forall c>0\mbox{ }\forall \theta_{1},\theta_{2}\in[-\pi,\pi]$$
% if we assume that $c\geq\frac{1}{2}$ then $0\leq-(1-2c)\leq-(\cos(\theta_{j})-2c)\mbox{ ,}j=1,2$
% \\that implies $(1-2c)^{2}\leq(\cos(\theta_{1})-2c)(\cos(\theta_{2})-2c)$
%$$\Longrightarrow 0\leq(\cos(\theta_{1})-2c)(\cos(\theta_{2})-2c)-4c^{2}+4c-1 $$
%$$\Longrightarrow 1-\cos(\theta_{1})\cos(\theta_{2})\leq c\sum_{j=1}^{2}2-2\cos(\theta_{j}) $$
%so that it is sufficient to choose $c=\frac{1}{2}$. As a consequence
%\begin{eqnarray*}
%{\rm det}(f_{\mathbb{P}_{2}}(\theta)) & = & C'\left(\sum_{j=1}^{2}2-2\cos(\theta_{j})+1-\cos(\theta_{1})\cos(\theta_{2})\right) \\
%& \leq &C'\left(\sum_{j=1}^{2}2-2\cos(\theta_{j})+\frac{1}{2}\left(\sum_{j=1}^{2}2-2\cos(\theta_{j})\right)\right)
%\end{eqnarray*}
%so that
%$${\rm det}(f_{\mathbb{P}_{2}}(\theta))\leq\frac{3}{2}C'\sum_{j=1}^{2}2-2\cos(\theta_{j}).$$
\par Finally, let
$\lambda_{1}\leq\lambda_{2}\leq\lambda_{3}\leq\lambda_{4}$ be the
eigenvalues of the Hermitian matrix
$\mathbf{f}_{\mathbb{P}_{2}}(\boldsymbol{\theta})$ and let
$\mu_{1}\leq\mu_{2}\leq\mu_{3}$ be the eigenvalues of the
principal submatrix $\mathbf{g}(\boldsymbol{\theta})$ chosen as
$\mathbf{g}(\boldsymbol{\theta})=(\mathbf{f}_{\mathbb{P}_{2}}(\boldsymbol{\theta}))_{i,j=2}^{4}$.\\

{Since the approximation matrices of problem (\ref{eq:modello}) are all Hermitian Positive Definite due to
\begin{itemize}
\item coerciveness of the continuous problem,
\item the use of Galerkin techniques such as the Finite Elements,
\end{itemize}
it follows that the symbol $$\mathbf{f}_{\mathbb{P}_{2}}(\boldsymbol{\theta})$$ of the related matrix-sequence has to be Hermitian nonnegative Definite, which means that $\lambda_{1} \geq 0$ on the whole definition domain. By contradiction if $\lambda_{1}$ is negative in a set of positive measure then, by the distribution results, (see \cite{Ba2-ETNA,Ba1-ETNA,glt-book-I,glt-book-II}),  many eigenvalues of the approximation matrices would be negative for a matrix size large enough and this is impossible.}

 By using the interlacing theorem, we have
\begin{equation}\label{interlacing1}
\lambda_{1}\leq\mu_{1}\leq\lambda_{2}\leq\mu_{2}\leq\lambda_{3}\leq\mu_{3}\leq\lambda_{4},
\end{equation}
with $\lambda_1$ equal to zero at
$\boldsymbol{\theta}=\boldsymbol{0}$ and positive elsewhere. By
direct computation of the determinant we find that $\rm
det(\mathbf{g}(\boldsymbol{\theta}))>0$ and hence, taking into account that
$(\mathbf{g}(\boldsymbol{\theta})$ is continuous and Hermitian Positive Definite on
the compact square $[-\pi,\pi]^{2}$, we conclude that all the
eigenvalues of $\mathbf{g}(\boldsymbol{\theta})$ are strictly positive and
continuous on $[-\pi,\pi]^{2}$ that is $\mu_{j}>0\mbox{ for
}j=1,2,3$. Thus, using (\ref{interlacing1}), we conclude
$\lambda_{2}\geq\mu_{1}\ge \min_{\boldsymbol{\theta}\in
[-\pi,\pi]^{2}}\mu_1 >0$.
\par
%======================================================================================
$\blacksquare$ Case $\mathbb{P}_k$ Finite Elements: $k=3$, $d=2$.
As in the case  $k=2$, Claim 1. follows by direct inspection from
the expression of the symbol $\mathbf{f}_{\mathbb{P}_k}$ in
(\ref{eq:P3symbol_c}), taking into account
$\boldsymbol{\theta}=(0,0)$ and the numerical values of the
involved parameters. \par In order to prove Claims 2. and 3. we
follow the very same steps as for the case $k = 2$,
 that is, we prove {\bf A.} and {\bf B.} with
$\mathbf{f}=\mathbf{f}_{\mathbb{P}_3}$.

By direct computation we have
\begin{eqnarray*}
\textrm{det}(\mathbf{f}_{\mathbb{P}_{3}}(\boldsymbol{\theta})) & = & a(-\cos(\theta_{2})\cos^{2}(\theta_{1})-\cos(\theta_{1})\cos^{2}(\theta_{2})+4\cos^{2}(\theta_{1})+4\cos^{2}(\theta_{2}) \\
     &&  -80\cos(\theta_{1})\cos(\theta_{2})-195\cos(\theta_{1})-195\cos(\theta_{2})+464)
\end{eqnarray*}
where $a={205891132094649}/{81920000000}$.  We write
det$(\mathbf{f}_{\mathbb{P}_{3}}(\boldsymbol{\theta}))$ in the
 form
$$
{\rm det}((\mathbf{f}_{\mathbb{P}_{3}}(\boldsymbol{\theta}))=
a\left(h(\boldsymbol{\theta})+\frac{195}{2}\sum_{j=1}^{2}(2-2\cos(\theta_{j}))\right)
$$
where
\[
h(\boldsymbol{\theta})=-\cos(\theta_{2})\cos^{2}(\theta_{1})-\cos(\theta_{1})\cos^{2}(\theta_{2})+4\cos^{2}(\theta_{1})+4\cos^{2}(\theta_{2})-80\cos(\theta_{1})\cos(\theta_{2})+74.
\]
Since
 $-\cos^{2}(\theta_{k})\leq-\cos(\theta_{j})\cos^{2}(\theta_{k})$ and  $1-\cos(\theta_{1})\cos(\theta_{2})\geq0$
we obtain
\begin{eqnarray*}
    h(\boldsymbol{\theta}) &\geq & 3\cos^{2}(\theta_{1})+3\cos^{2}(\theta_{2})-80\cos(\theta_{1})\cos(\theta_{2})+74  \\
              &\geq & 3(\cos(\theta_{1})-\cos(\theta_{2}))^{2}-74\cos(\theta_{1})\cos(\theta_{2})+74      \\
              &\geq & 0,
  \end{eqnarray*}
 which implies directly $ {\rm det}(\mathbf{f}_{\mathbb{P}_{3}}(\boldsymbol{\theta}))\geq C_1\sum_{j=1}^{2}(2-2\cos(\theta_{j}))$ with $C_1=
 \frac{195}{2}a$.
\\
On the other side, taking into account $\cos^{2}(\theta_{j})\leq 1,
j=1,2$, we deduce
\begin{eqnarray*}
   h(\boldsymbol{\theta})&=&\cos^{2}(\theta_{1})(4-\cos(\theta_{2}))+ \cos^{2}(\theta_{2})(4-\cos(\theta_{1}))-80\cos(\theta_{1})\cos(\theta_{2})+74    \\
    &\leq &8 -\cos(\theta_{1})-\cos(\theta_{2})-80\cos(\theta_{1})\cos(\theta_{2})+74     \\
    &\leq&2 -\cos(\theta_{1})-\cos(\theta_{2})+80(1-\cos(\theta_{1})\cos(\theta_{2})) \\
    &\leq&\frac{81}{2}\sum_{j=1}^{2}(2-2\cos(\theta_{j})).
 \end{eqnarray*}
Owing to the relation
$1-\cos(\theta_{1})\cos(\theta_{2})\leq\frac{1}{2}(4
-2\cos(\theta_{1})-2\cos(\theta_{2}))$, as already observed  in
the case $k = 2$, we  find
$\textrm{det}(\mathbf{f}_{\mathbb{P}_{3}}(\boldsymbol{\theta}))\leq138a\sum_{j=1}^{2}(2-2\cos(\theta_{j}))$
with $C_2=138a$.
\par
Finally, let $\lambda_{1}\leq\lambda_{2}\leq\cdots\leq\lambda_{9}$
be the eigenvalues of the Hermitian matrix
$\mathbf{f}_{\mathbb{P}_{3}}(\boldsymbol{\theta}),$ and let
$\mathbf{g}(\boldsymbol{\theta})=(\mathbf{f}_{\mathbb{P}_{3}}(\boldsymbol{\theta}))_{i,j=1}^{8}$
be the principal submatrix and $\mu_{1}\leq\cdots\leq\mu_{8}$ its
eigenvalues.

 {Since the approximation matrices of problem (\ref{eq:modello}) are all Hermitian Positive Definite due to
\begin{itemize}
\item coerciveness of the continuous problem,
\item the use of Galerkin techniques such as the Finite Elements,
\end{itemize}
it follows that the symbol $\mathbf{f}_{\mathbb{P}_{3}}(\boldsymbol{\theta})$ of the related matrix-sequence has to be Hermitian nonnegative Definite, which means that $\lambda_{1} \geq 0$ on the whole definition domain. By contradiction if $\lambda_{1}$ is negative in a set of positive measure then, by the distribution results, (see \cite{glt-book-I,glt-book-II}),  many eigenvalues of the approximation matrices would be negative for a matrix size large enough and this is impossible.}

 By using the interlacing theorem, we have
\begin{equation}\label{interlacing2}
\lambda_{1}\leq\mu_{1}\leq\lambda_{2}\leq\mu_{2}\leq\cdots\leq\mu_{8}\leq\lambda_{9},
\end{equation}
with $\lambda_1$ equal to zero at
$\boldsymbol{\theta}=\boldsymbol{0}$ and positive elsewhere. By
direct computation we find
$\textrm{det}(\mathbf{g}(\boldsymbol{\theta}))=\displaystyle\prod_{j=1}^{8}
\mu_{j}>0$ so that, taking into account that
$\mathbf{g}(\boldsymbol{\theta})$ is continuous and Hermitian Positive Definite on
the compact square $[-\pi,\pi]^2$, we deduce $\mu_{j}~>~0$ for all
$j=1,\ldots,8$.  Consequently, by (\ref{interlacing2}), we
conclude
 $\lambda_{2}\geq\mu_{1}\ge \min_{\boldsymbol{\theta}\in [-\pi,\pi]^{2}}\mu_1 >0$.  \end{proof}
%
%\begin{Remark}
\subsection{Extremal eigenvalues and conditioning}\label{rem consequence}

As already observed, direct consequences of Proposition
\ref{distr-P_2}, Corollary \ref{cor:distr-P_2}, Proposition
\ref{distr-P_3}, Corollary \ref{cor:distr-P_3} are  that the
sequences of Finite Element matrices are distributed as the symbol
$\mathbf{f}$ and that the union of the ranges of the eigenvalue functions
of $\mathbf{f}$ represent a cluster for their spectra, while the convex
hull of the the union of the ranges of the eigenvalue functions of
$\mathbf{f}$ contains all the eigenvalues of the involved matrices.
\par
On the other hand,  Theorem \ref{theorem general} gives information on the analytical properties of $\mathbf{f}$, which are relevant for giving
results on the extreme eigenvalues and the asymptotic conditioning.

Indeed, from Theorem \ref{theorem general}, we know that the
minimal eigenvalue function of $\mathbf{f}$ behaves as the symbol of the
standard Finite Difference Laplacian, while  the other eigenvalue
functions are well separated from zero and bounded. Furthermore,
thanks to the analysis in \cite{marko}, the fact that the minimal
eigenvalue of $\mathbf{f}$ has a zero of order two implies that
\begin{itemize}
\item the minimal eigenvalue goes to zero as $N^{-2/d}$,
\item the maximal eigenvalue converges from below to the maximum of the
maximal eigenvalue function of  $\mathbf{f}${,}
\item and hence the conditioning of the involved matrices grows asymptotically exactly
as $N^{2/d}$,
\end{itemize}
 with $N$ being the global matrix size, (see also the argument in Section 5.1 in \cite{Q_k}, \cite{S-LAA-1998} and the last part of Subsection \ref{bloctoep} of the present thesis).

 \section{The case of variable coefficients and non-Cartesian domains} \label{sez:var coeff}
When the diffusion coefficient $a(\mathbf{x})$ in
(\ref{eq:modello}) is not constant, the structure of the stiffness
matrix  is no longer Toeplitz, but somehow the Toeplitz character
is hidden in an asymptotic sense and indeed the sequence of
matrices $\{A_n(a,\Omega,{\mathbb{P}_k})\}_n$ approximating (\ref{eq:modello})
can be spectrally treated with the help of the GLT technology with $k=1,2,3$.
\par
Below we report the essentials of the steps for computing the spectral symbol.
 \begin{description}
  \item[Step 1.] If $\Omega=(0,1)^d$, $d\ge 1$, then $\{A_n(a,\Omega,{\mathbb{P}_k})\}_n$ can be written as a sequence of principal submatrices of a linear combination   of products involving the multilevel block Toeplitz sequence generated by $\mathbf{f}_{\mathbb{P}_{k}}$, the diagonal sampling sequence of $a(\mathbf{x}) I_{N(k,d)}$, and zero
  distributed sequences. The use of items {\bf GLT~1.}--{\bf GLT~3.}, combined with Theorem \ref{extradimensional}, leads to the conclusion
\begin{equation}\label{symbol var1}
\{A_n(a,\Omega,{\mathbb{P}_k})\}_n\sim_{\sigma,\lambda} a(\mathbf{x})
f_{\mathbb{P}_{k}}(\boldsymbol{\theta}),\ \ \ \mathbf{x}\in
(0,1)^d,\ \boldsymbol{\theta}\in [-\pi,\pi]^d.
\end{equation}
  \item[Step 2.] If $\Omega$ is Peano-Jordan measurable, then without loss of generality, we assume $\Omega\subset \Omega_d=(0,1)^d$ and $d\ge 2$. Hence $\{A_n(a,\Omega,{\mathbb{P}_k})\}_n$ can be seen, up to zero-distributed sequences, as a sequence of principal submatrices of $\{A_n(\hat a,\Omega,{\mathbb{P}_k})\}_n$, where $\hat a$ is equal to $a$ on the domain $\Omega$ and it is identically zero in the complement $\Omega_d\backslash \Omega$. In this way we are reduced to  {\bf Step 1.} and the use of a reduction argument, (see Section 6 in \cite{glt-1} and Section 3.1.4 in \cite{glt-2}) implies the distribution result
 \begin{equation}\label{symbol var2}
\{A_n(a,\Omega,{\mathbb{P}_k})\}_n\sim_{\sigma,\lambda} a(\mathbf{x})
\mathbf{f}_{\mathbb{P}_{k}}(\boldsymbol{\theta}),\ \ \ \mathbf{x}\in
\Omega,\ \boldsymbol{\theta}\in [-\pi,\pi]^d.
\end{equation}
 \end{description}
The rest of the section is now devoted to show that the
predictions in (\ref{symbol var1}) and (\ref{symbol var2}) are
numerically confirmed. Indeed, in the constant coefficient case,
we plotted the surface of the different eigenvalue functions
$\lambda_j\left(\mathbf{f}_{\mathbb{P}_k}(\boldsymbol{\theta})\right)$,
$j=1,\ldots k^2$, $k=1,2,3,4$, and this was technically possible
because the functions are all bivariate  as
$\boldsymbol{\theta}\in [-\pi,\pi]^2$.
\par
In the variable coefficient case, the visualization is
substantially more involved, since the symbol is $a(\mathbf{x})
\mathbf{f}_{\mathbb{P}_{k}}(\boldsymbol{\theta})$ and hence the eigenvalue
functions
$\lambda_j\left(a(\mathbf{x})\mathbf{f}_{\mathbb{P}_k}(\boldsymbol{\theta})\right)$,
$j=1,\ldots k^2$, $k=1,2,3,4$, are all functions in $4$ variables
 as $\mathbf{x}\in \Omega$, $\boldsymbol{\theta}\in
[-\pi,\pi]^2$. Consequently, for visualization purposes, we choose
a different technique: for a fixed $k$ and for a fixed matrix size,
we make an ordering (nondecreasing) of all the eigenvalues of
$A_n(a,\Omega,{\mathbb{P}_k})$ and we take the same ordering (nondecreasing) of
the values given by an equispaced sampling of all the functions
$\lambda_j\left(a(\mathbf{x})\mathbf{f}_{\mathbb{P}_k}(\boldsymbol{\theta})\right)$,
$j=1,\ldots k^2$.
\par
As it can be seen from Figures \ref{fig:es0}-\ref{fig:es3}, all
concerning the case $\Omega=(0,1)^2$, the match is perfect showing
that the distribution result in (\ref{symbol var1}) is fully
confirmed with $a(x,y)=1$, $a(x,y)=e^{x+y}$,
$a(x,y)=1+2\sqrt{x}+y$, $a(x,y)=1$ if $y\ge x$ and $a(x,y)=2$
otherwise.\\

% DA ESPANDERE
%\end{Remark}
%=========================================================================================================
%\begin{landscape}
\begin{figure}[htb]
\centering
\includegraphics[width=\textwidth]{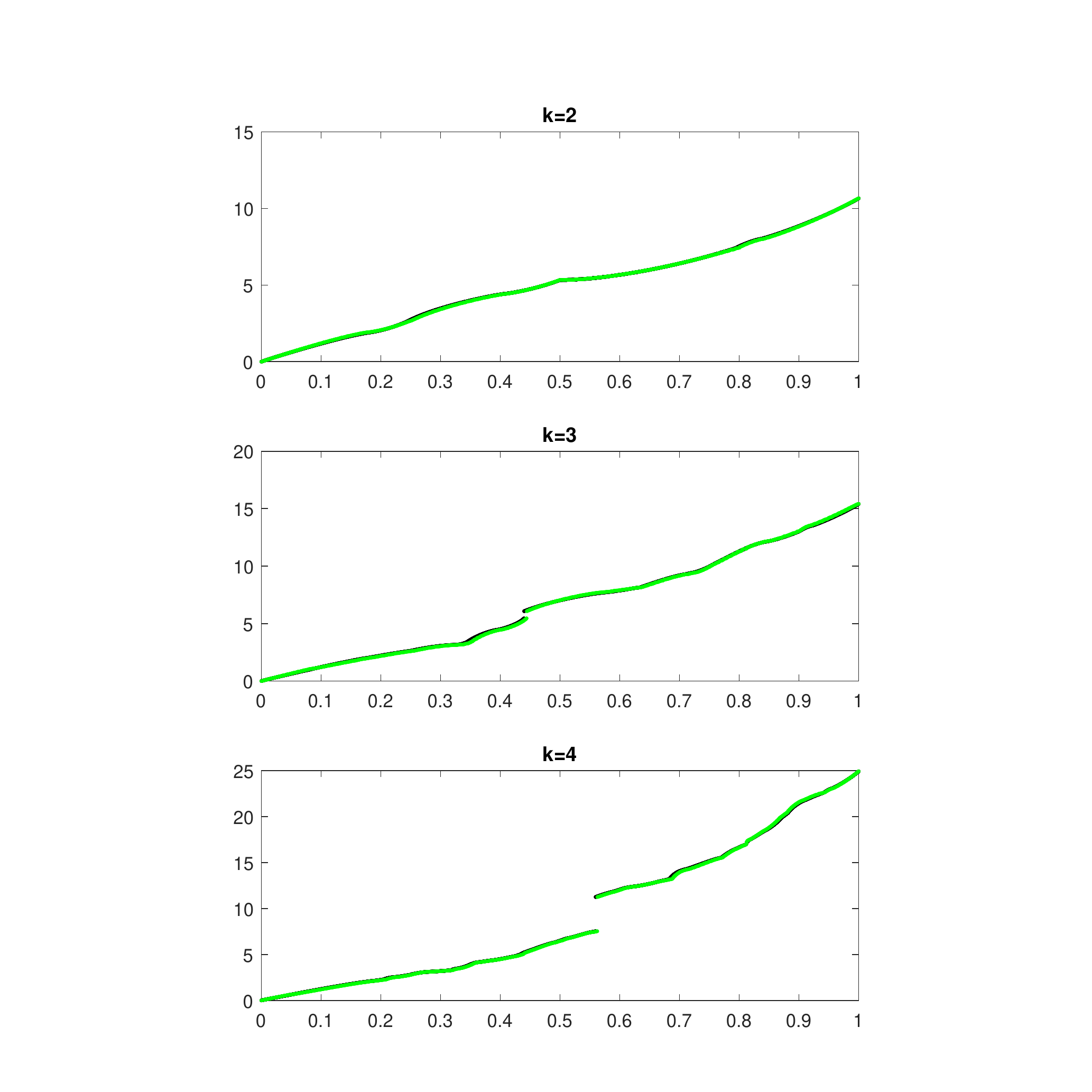}
\caption{Ordered equispaced samplings of
$\lambda_j\left(a(x,y)\mathbf{f}_{\mathbb{P}_k}(\boldsymbol{\theta})\right)$,
$j=1,\ldots k^2$ (green dots) and ordered eigenvalues
$\lambda_l(A_n(a,\Omega,{\mathbb{P}_k}))$ with $a(x,y)\equiv1$.}
\label{fig:es0}
\end{figure}
%-------------------------------------------------------------------------------
\begin{figure}[htb]
\centering
\includegraphics[width=\textwidth]{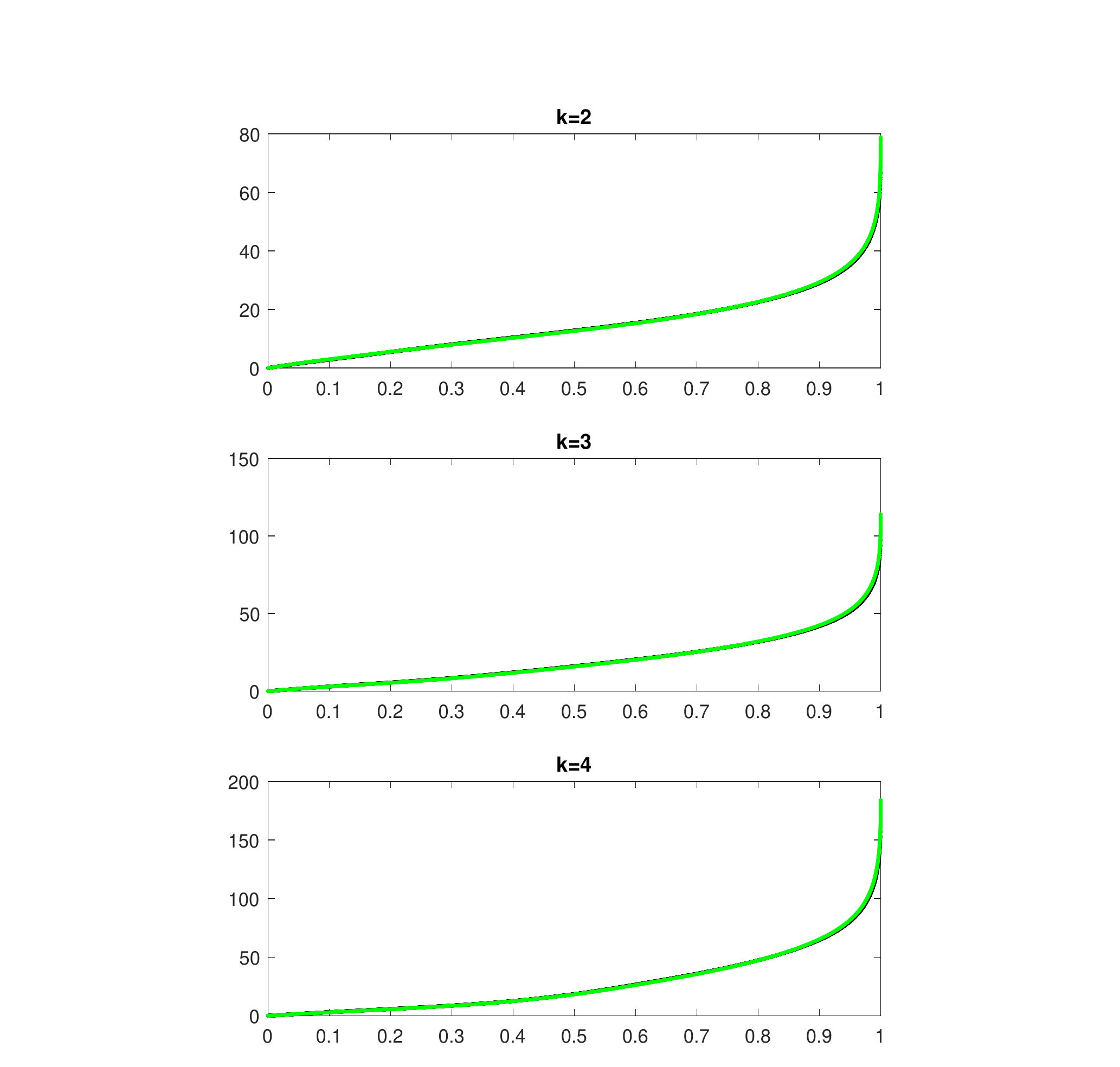}
\caption{Ordered equispaced samplings of
$\lambda_j\left(a(x,y)\mathbf{f}_{\mathbb{P}_k}(\boldsymbol{\theta})\right)$,
$j=1,\ldots k^2$ (green dots) and ordered eigenvalues
$\lambda_l(A_n(a,\Omega,{\mathbb{P}_k}))$ with $a(x,y)=e^{x+y}$.}
\label{fig:es1}
\end{figure}
%-------------------------------------------------------------------------------
\begin{figure}[htb]
\centering
\includegraphics[width=\textwidth]{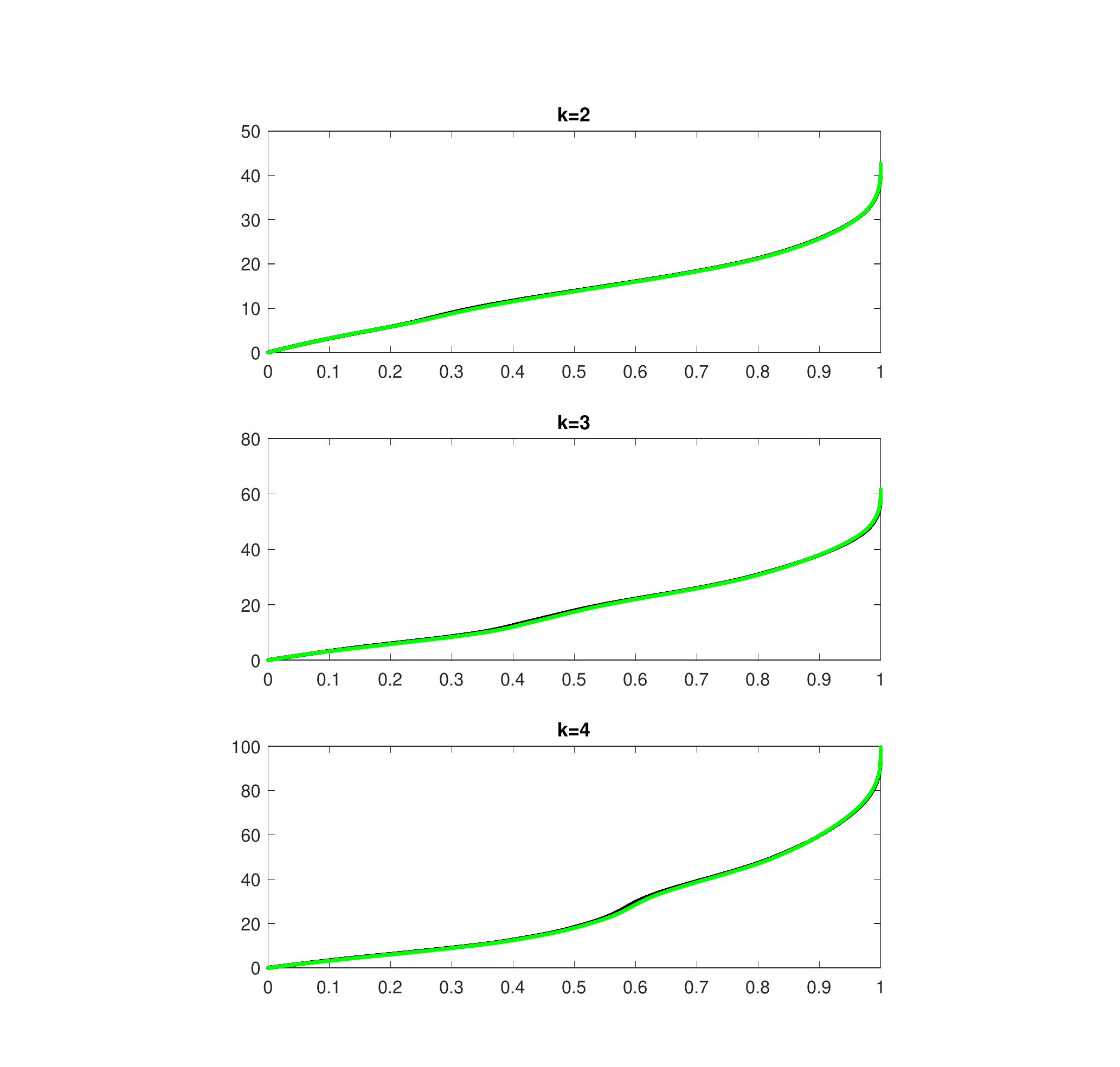}
\caption{Ordered equispaced samplings of
$\lambda_j\left(a(x,y)\mathbf{f}_{\mathbb{P}_k}(\boldsymbol{\theta})\right)$,
$j=1,\ldots k^2$ (green dots) and ordered eigenvalues
$\lambda_l(A_n(a,\Omega,{\mathbb{P}_k}))$ with
$a(x,y)=1+2\sqrt{x}+y$.} \label{fig:es2}
\end{figure}
%-------------------------------------------------------------------------------
\begin{figure}[htb]
\centering
\includegraphics[width=\textwidth]{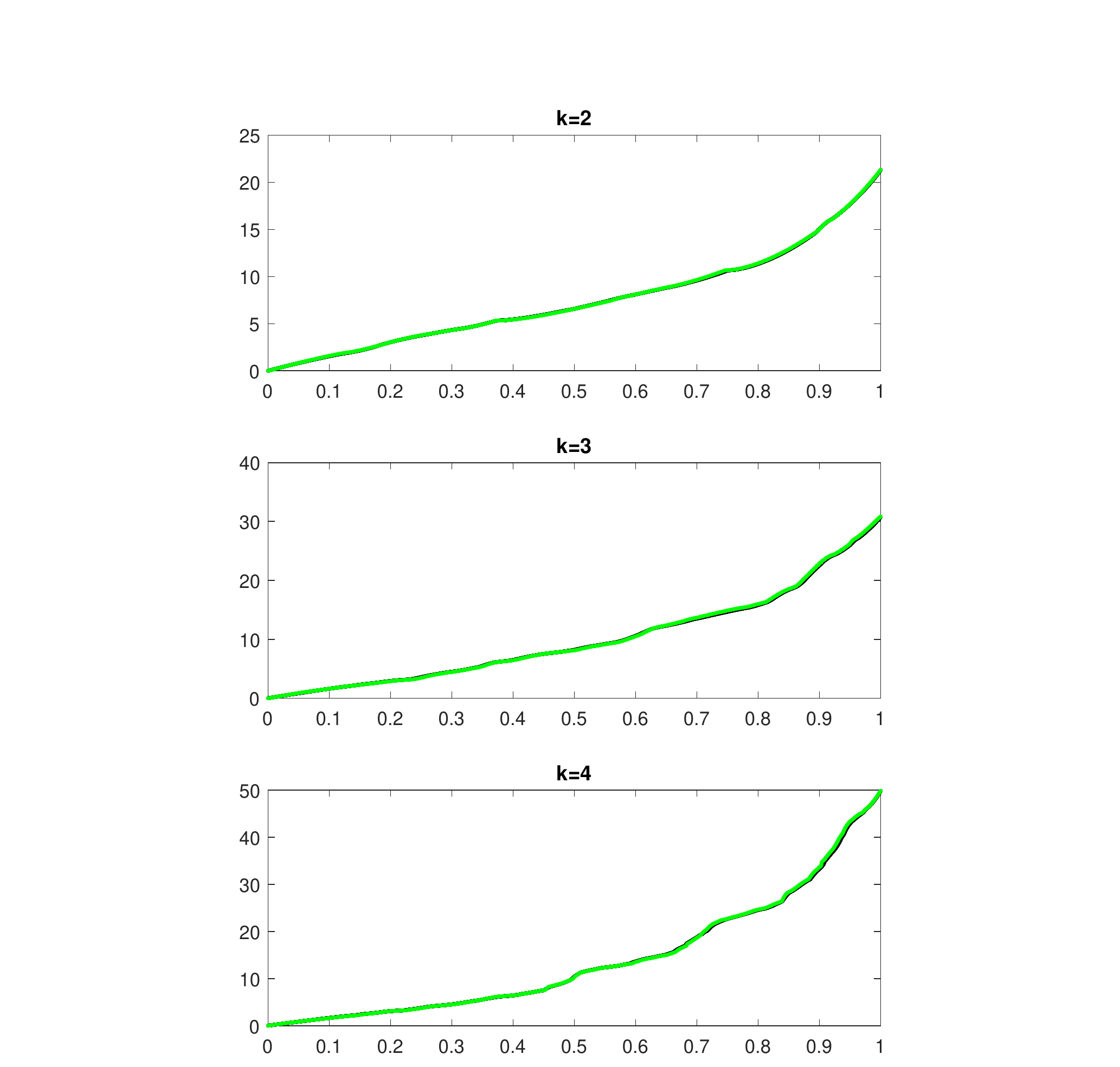}
\caption{Ordered equispaced samplings of
$\lambda_j\left(a(x,y)\mathbf{f}_{\mathbb{P}_k}(\boldsymbol{\theta})\right)$,
$j=1,\ldots k^2$ (green dots) and ordered eigenvalues
$\lambda_l(A_n(a,\Omega,{\mathbb{P}_k})$ with $a(x,y)=1$ if $y\ge
x$ and $a(x,y)=2$ otherwise.} \label{fig:es3}
\end{figure}
%----------------------------------------------------------------------
We have four relevant remarks.

\begin{itemize}
\item As a general observation, the graph of the ordered
equispaced sampling of
  $\lambda_j\left(a(\mathbf{x})\mathbf{f}_{\mathbb{P}_k}(\boldsymbol{\theta})\right)$, $j=1,\ldots k^2$, represent a
  monotone rearrangement, (see \cite{DFS} for seminal results and \cite{BBG} for a recent revue) of the different eigenvalue functions and this global rearrangement is a  fortiori a univariate monotone function.
\item When $a(\mathbf{x})\equiv 1$, the symbol in (\ref{symbol
var1}) reduces to $\mathbf{f}_{\mathbb{P}_k}$ and we observe jumps which
correspond to the existence of an index $l$ such that
 \[
\max \lambda_l\left(\mathbf{f}_{\mathbb{P}_k}(\boldsymbol{\theta})\right) <
\lambda_{l+1}\left(\mathbf{f}_{\mathbb{P}_k}(\boldsymbol{\theta})\right),
\]
with $1\le l\le k^2-1$, $k=1,2,3,4$. In other words the rearranged
function has a few discontinuity points. When $a(\mathbf{x})$ is
not constant such a phenomenon disappears, since the range of all
eigenvalue functions becomes wider and all the ranges intersect
(there is not a range not intersecting at least another range).
Furthermore,  beside the latter smoothing effect due to
$a(\mathbf{x})$, it is worthwhile observing that the regularity of
the diffusion coefficient does not  affect the qualitative
behaviour of the eigenvalue distribution. In fact, the
reconstruction given by the symbol of the eigenvalues of
$A_n(a,\Omega,{\mathbb{P}_k})$  is accurate in all the considered
examples and, more specifically, the figures look very similar
independently of the fact that the diffusion coefficient is smooth
($a(\mathbf{x})=e^{x+y}$ in Figure \ref{fig:es1}), or is ${ C}^0$
but not ${ C}^1$ ($a(\mathbf{x})=1+2\sqrt{x}+y$ in Figure
\ref{fig:es2}), or is discontinuous ($a(\mathbf{x})=1$ if $y\ge x$
and $a(\mathbf{x}) =2$ otherwise, in Figure \ref{fig:es3}).
\item In all Figures \ref{fig:es1}--\ref{fig:es3} the matrix size
is quite moderate (of the   order of $10^4$), showing that the
spectral distribution effects, which represent an asymptotic
property, can be already visualized for small orders of the
considered matrices.
\item We did not show any figure regarding the distribution
formula (\ref{symbol var2}) just because the check has been done
and there is no difference with respect to the case of
$\Omega=(0,1)^2$ as in (\ref{symbol var1}).
\end{itemize}
\section{Preconditioning and complexity issues} \label{sez:numerical}
Lastly, we consider a few numerical experiments on
preconditioning. First of all, the interest in solving the
constant coefficient case $a(x,y)\equiv 1$ refers to its use in
optimally preconditioning the nonconstant coefficient case
whenever $a$ is smooth enough, (see \cite{C1,C2} for the case
$k=1$).  This is evident from Table \ref{tab:a=exp}, where we
report the number of iterations required by PCG applied to
$A_n(a,\Omega,{\mathbb{P}_k})$, with $a(x,y)=\exp(x+y)$ in the
case of preconditioning with $P_n(a)=\tilde{D}_n^{1/2}(a)
A_n(1,\Omega,{\mathbb{P}_k}) \tilde{D}_n^{1/2}(a)$ with
$\tilde{D}_n(a)=D_n(a)D_n^{-1} (1)$, where $D_n(a)$ is the main
diagonal of $A_n(a,\Omega,{\mathbb{P}_k})$ and $D_n(1)$ is the
main diagonal
of $A_n(1,\Omega, {\mathbb{P}_k})$.\\

Therefore, we now focus our attention on the case $a(x,y)\equiv
1$. Taking into account the Toeplitz  nature of the matrices at
hand, our aim is to preliminarily test a classical preconditioner
as  the Circulant one, clearly by considering the Strang
correction, in order to deal with its singularity. More precisely, we will
consider as preconditioner the Circulant matrix generated by the
very same function $f_{\mathbb{P}_{k}}(\boldsymbol{\theta})$ plus
the correction $h^2e e^T$, $e$ being the vector of all ones and
$h$ the constant triangle edge. In the even columns of Table
\ref{tab:CS_k2} (case $k=2$) we report the number of PCG
iterations  required to solve the system with Toeplitz matrix
$T_n(f_{\mathbb{P}_{k}})$, in the case of no preconditioning
($P_n=I_n$), preconditioning by the incomplete Cholesky
factorization, and by the Circulant $C_n(f_{\mathbb{P}_{k}})$ plus
the Strang correction, respectively. To this end, it is worth
 stressing that we have to consider the dimension of the Toeplitz/Circulant matrix fitting with its natural
dimension with  respect to the symbol size. Therefore, when
instead we want to solve the system with the FEM matrix
$A_n(1,\Omega, {\mathbb{P}_k)}$, principal submatrix of the matrix
$T_n(f_{\mathbb{P}_{k}})$, we need to match its dimension with the
one previously considered for the Circulant preconditioner. We
obtain that goal just by imposing boundary conditions to
$T_n(f_{\mathbb{P}_{k}})$, but keeping the size unchanged. The
related numerical results are reported in odd columns of Table
\ref{tab:CS_k2}. In both cases, the number of required iterations
increases as the dimension  increases. No significant difference
in even or odd column is observed in the case of no
preconditioning or incomplete Cholesky preconditioning (the
results are slightly better in the second case). In the case of
the Circulant preconditioning we observe a clear worsening in the
effectiveness when the preconditioner is applied not to the
Toeplitz matrix, but to its principal submatrix plus boundary
conditions, though the iteration growth rate seems smaller than
the one observed in the case of incomplete Cholesky
preconditioning. Furthermore, as expected from the theory, a weak
cluster around $1$ is observed, (see Table \ref{weak cluster
table}).
\par
In Tables \ref{tab:CS_k3} and \ref{tab:CS_k4} the same numerical
experiments are reported in the case $k=3$ and $k=4$. The
numerical behavior seems to be substantially of the same type,
independently of the parameter $k$,  also in reference to the weak
cluster phenomenon observed for $k=2$ in Table \ref{weak cluster
table}.

%-------------------------------------------------------------
%-------------------------------------------------------------------------------
\begin{table}
\centering
\begin{tabular}{|cc|cc|cc|}
\hline
\multicolumn{2}{|c|}{$k=2$} & \multicolumn{2}{|c|}{$k=3$} & \multicolumn{2}{|c|}{$k=4$} \\
\hline
 $N$ & $P_n(a)$ & $N$ & $P_n(a)$ & $N$ & $P_n(a)$ \\
 \hline
 49   & 4 & 121   & 5 & 225   & 5 \\
225   & 3 & 529   & 4 & 961   & 5 \\
961   & 3 & 2209  & 4 & 3969  & 4 \\
3969  & 3 & 9025  & 4 & 16129 & 4 \\
16129 & 3 & 36481 & 4 & 65025 & 4\\
\hline
\end{tabular} \caption{Number of PCG iterations  to reach
convergence with  respect to relative residual less than 1.e-6,
Preconditioner $P_n(a)$, $a(x,y)=exp(x+y)$.}  \label{tab:a=exp}
\end{table}

%-------------------------------------------------------------------------------
\begin{table}
 \centering
\begin{tabular}{|c|cc|cc|cc|}
\hline
\multicolumn{7}{|c|}{$k=2$}  \\
\hline
$N$ & \multicolumn{2}{|c|}{$P=I$} & \multicolumn{2}{|c|}{$P=IC$} & \multicolumn{2}{|c|}{$P=C_S$} \\
 \hline
 64   & 26  & 19  & 11  & 10  & 14 & 19 \\
256   & 47  & 42  & 18  & 17  & 19 & 30 \\
1024  & 90  & 86  & 32  & 31  & 26 & 42 \\
4096  & 174 & 170 & 56  & 55  & 38 & 59 \\
16384 & 336 & 331 & 98  & 96  & 53 & 87 \\
%65536 & 644 & 636 & 185 & 183 &    &    \\
 \hline
\end{tabular} \caption{Number of PCG iterations  to reach
convergence with  respect to relative relative residual less than
1.e-6 - case $k=2$.}\label{tab:CS_k2}
\end{table}
%-------------------------------------------------------------------------------
\begin{table}  \centering
 \begin{tabular}{|c|cc|cc|c|cc|cc|} \hline
\multicolumn{5}{|c|}{$k=2$}   \\
\hline
$N$ & $n_{out}$ & $\%$ & $n_{out}$ & $\%$  \\
 \hline
 64   & 27   & 4.2e-1  & 27  &  4.2e-1  \\
256   & 59   & 2.3e-1  & 55  &  2.1e-1  \\
1024  & 123  & 1.2e-1  & 111 &  1.1e-1  \\
4096  & 251  & 6.1e-2  & 225 &  5.5e-2  \\
 \hline
\end{tabular} \caption{Number of outliers $n_{out}$ (eigenvalues not belonging to $(1-\varepsilon,1+\varepsilon)$ with $\varepsilon=1.e-1$) and
their percentage with respect to the dimension. The second and third
columns refer to the Toeplitz case, the fourth and fifth columns
to the case of the FEM matrix.}\label{weak cluster table}

\end{table}

%-------------------------------------------------------------------------------
\begin{table}
 \centering
\begin{tabular}{|c|cc|cc|cc|}
\hline
\multicolumn{7}{|c|}{$k=3$}  \\
\hline
$N$ & \multicolumn{2}{|c|}{$P=I$} & \multicolumn{2}{|c|}{$P=IC$} & \multicolumn{2}{|c|}{$P=C_S$} \\
 \hline
 144  & 45  & 39  & 16  & 15  & 19 & 30 \\
576   & 82  & 78  & 28  & 27  & 26 & 42 \\
2304  & 159 & 155 & 50  & 48  & 36 & 61 \\
9216  & 306 & 301 & 86  & 84  & 49 & 90 \\
 \hline
\end{tabular} \caption{Number of PCG iterations  to reach
convergence with  respect to relative relative residual less than
1.e-6 - case $k=3$.}\label{tab:CS_k3}
\end{table}
%
%-------------------------------------------------------------------------------
\begin{table}
 \centering
\begin{tabular}{|c|cc|cc|cc|}
\hline
\multicolumn{7}{|c|}{$k=4$}  \\
\hline
$N$ & \multicolumn{2}{|c|}{$P=I$} & \multicolumn{2}{|c|}{$P=IC$} & \multicolumn{2}{|c|}{$P=C_S$} \\
 \hline
 256   & 74  & 66  & 10 & 10 & 23 & 38 \\
1024  & 134 & 129 & 18 & 17 & 31 & 57 \\
4096  & 261 & 254 & 31 & 30 & 43 & 81 \\
16384 & 502 & 490 & 54 & 51 & 61 & 116\\
 \hline
\end{tabular} \caption{Number of PCG iterations  to reach
convergence with  respect to relative relative residual less than
1.e-6 - case $k=4$.}\label{tab:CS_k4}
\end{table}

\chapter{Multigrid for \texorpdfstring{$\mathbb{Q}_k$}{Qk} Finite Element Matrices Using a (Block) Toeplitz  Symbol Approach}
In the present chapter, we consider multigrid strategies for the resolution of linear systems arising from the \texorpdfstring{$\mathbb{Q}_k$}{Qk} Finite Elements approximation of the elliptic problem

\begin{equation} \label{eq:modello1}
\left \{
\begin{array}{l}
\mathrm{div} \left(-a(\mathbf{x}) \nabla u\right)
=f, \quad \mathbf{x}\in \Omega \subseteq \mathbb{R}^d, \\
u_{|\partial \Omega}=0,
\end{array}
\right.
\end{equation}
with {$\Omega$ a bounded subset} of $\mathbb{R}^d$, having smooth boundaries, and with $a$ being continuous and positive on $\overline \Omega$. While the analysis is performed in one dimension, the  numerics are carried out also in higher dimension \texorpdfstring{$d\ge 2$}{greater than one}, showing an optimal behavior in terms of the dependency on the matrix size and a substantial robustness with respect to the dimensionality \texorpdfstring{$d$}{d} and to the polynomial degree \texorpdfstring{$k$}{k}.

\section{Optimality of two-grid method in the case of Toeplitz matrices }\label{optimal2grid}
We start the current subsection with a key remark.
\begin{Remark}\label{rmk:condition_scalar}
{In the relevant literature, (see for instance, \cite{ADS}), the convergence analysis of the two-grid method splits into the validation of two separate conditions: the smoothing property and the approximation property. Regarding the latter, with reference to {scalar} structured matrices \cite{ADS,Fiorentino-Serra}, the optimality of two-grid methods is given in terms of choosing the proper conditions that the symbol $p$ of a family of projection operators has to fulfill. Indeed, {consider} $T_{n}(f)$ with ${n}=(2^t-1)$ and $f$ being a nonnegative trigonometric polynomial. Let $ {\theta}^0$ be the unique zero of $f$. Then,   the optimality of the two-grid method applied to $T_{n}(f)$ is guaranteed if we choose the symbol $p$ of the {family of} projection operators such that
 \begin{equation}
\begin{split}
\underset{ {\theta}\to  {\theta}^0}{\lim\sup}\frac{|p( {\eta})|^2}{f( {\theta})}&<\infty, \quad  {\eta}\in \mathcal{M}( {\theta}),\\
\sum_{ {\eta}\in \Omega( {\theta})}p^2( {\eta})&>0,
\end{split}
 \end{equation}
where the sets $\Omega( {\theta})$ and $\mathcal{M}( {\theta})$ are the following {corner} and {mirror} points
\[\Omega( {\theta})=\{ \eta\in \{\theta,\theta+\pi\}\},\qquad \mathcal{M}( {\theta})=\Omega( {\theta})\setminus \{ {\theta}\},\]
respectively.}
\end{Remark}
{Informally, it means that the optimality} of the two-grid method is obtained by choosing the {family of} projection operators associated  to a symbol $p$ such that $|p|^2(\vartheta) +|p|^2(\vartheta+\pi)$ does not have zeros and $|p|^2(\vartheta+\pi)/f(\vartheta)$ is bounded
%with $f$ being the spectral symbol of the considered matrix-sequence
(if we require the optimality of the V-cycle, then the second condition is a bit stronger, {see \cite{ADS})}. In a differential context, the previous conditions mean that $p$ has a zero of order at least $\alpha$ at $\vartheta=\pi$, whenever $f$ has a zero at {$ {\theta}^0=0$} of order $2\alpha$.

In our specific block setting, by interpreting the analysis given in \cite{multi-block}, all the involved symbols are matrix-valued and the {conditions which are sufficient for the} two-grid convergence {and optimality} are the following:
\begin{description}
\item[(A)] zero of order $2$ at $\vartheta=\pi$ of the proper eigenvalue function of the symbol of the projector for $\mathbb{Q}_k$, $k=1,2,3$ (mirror point theory \cite{ADS,Fiorentino-Serra});
\item[(B)] positive definiteness of $\mathbf{pp}^*(\vartheta) +\mathbf{pp}^*(\vartheta+\pi)$; and%Whether it is correct? Correct.
\item[(C)] commutativity of   $\mathbf{p}(\vartheta)$ and $\mathbf{p}(\vartheta+\pi)$.
\end{description}
Our choices are in agreement with the mathematical conditions set in {Items}  {\bf (A)} and {\bf (B)}, while {Condition}    {\bf (C)} is not satisfied. The violation of Condition    {\bf (C)} is discussed later, while, in relation to Condition    {\bf (A)}, we observe that  {a stronger condition is met}, since the considered order of the zero at $\vartheta=\pi$ is $k+1$, which is larger than $2$ for $k=2,3$.
%%%%%%%%%%%%%%%%%%%%%%%%%%%%%%%%%%%%%%%%%%%%%%%%%%%%%%%%%%%

\section{Structure of the Matrices and Spectral Analysis} %
%{$\mathbb{Q}_k\equiv \mathbb{P}_k$}}
%{equivalence}, \texorpdfstring{$d=1$}{in one dimension}}\label{sec:Qk-1D}
Since in the case $d=1$, $\mathbb{Q}_k\equiv \mathbb{P}_k$, then everything provided in the Section \ref{sez:Pk-1D} remains valid.
\section{Multigrid Strategy Definition, Symbol Analysis, and Numerics}\label{sec:MuDef-Pk-1D}
Let us consider a family of meshes

\[ \{ {\cal T}_{2^sh}
\}_{s=0,\ldots , \overline{s}}\textrm{ such that  } {\cal
T}_{2^sh} \subseteq {\cal T}_{2^{s-1}h} \subseteq \ldots \subseteq
{\cal T}_{2h} \subseteq {\cal T}_{h}.\]

Clearly, the same inclusion property is inherited by the
corresponding Finite Element functional spaces and hence we find ${\cal V}_{2^sh}
\subseteq {\cal V}_{2^{s-1}h} \subseteq \ldots \subseteq {\cal
V}_{2h} \subseteq {\cal V}_{h}$.

Therefore,  to
formulate a multigrid strategy, it is quite natural to follow a
functional approach and to impose the prolongation operator
$p_{2h}^h: {\cal V}_{2h} \rightarrow {\cal V}_{h}$ to be defined
as the identity operator, that is
\[
p_{2h}^h    v_{2h}= v_{2h} \textrm{ for all } v_{2h} \in {\cal
V}_{2h}.
\]
Thus, the matrix representing the prolongation operator is formed,
column by column, by representing each function of the basis of
${\cal V}_{2h}$ as linear combination of the basis of ${\cal
V}_{h}$, the coefficients being the values of the functions
$\varphi_i^{2h}$ on the fine mesh grid points, i.e.,

\begin{equation}
\varphi_i^{2h} (x) = \sum_{x_j \in {\cal T}_{h} }
\varphi_i^{2h}(x_j) \varphi_j^{h} (x). \label{eq:comb_lin}
\end{equation}
In the following subsections, we   consider in detail the case of $\mathbb{Q}_k$ Finite Element approximation with $k=2$ and $k=3$,
the case $k=1$ being reported in short just for the sake of completeness.

\subsection{\texorpdfstring{$\mathbb{Q}_1$}{Q1} Case}
Firstly, let us  consider the case of $\mathbb{Q}_1$ Finite Elements, where, as is well known, the stiffness matrix is
the {scalar} Toeplitz matrix generated by $\mathbf{f}_{\mathbb{Q}_1}(\vartheta)=2-2\cos(\vartheta)$, and,
for the sake of simplicity, let us consider the case of ${\cal
T}_{2h}$ partitioning with five equispaced points (three internal points)
and ${\cal T}_{h}$ partitioning with nine equispaced points (seven
internal points) obtained from ${\cal T}_{2h}$ by considering the
midpoint of each subinterval.
{In the standard geometric multigrid, the prolongation operator matrix is {defined}  as}

\begin{equation}
P_{h\times 2h} = P_3^7= {
 \left [
\renewcommand{\arraystretch}{1.2}
\begin{array}{rrr}
\displaystyle \frac{1}{2} &              & \\
1                         &              & \\
\displaystyle \frac{1}{2} & \displaystyle \frac{1}{2}  & \\
                          & 1            & \\
                          & \displaystyle \frac{1}{2}  & \displaystyle \frac{1}{2} \\
                          &              & 1 \\
                          &              & \displaystyle\frac{1}{2} \\
\end{array}
\right ].
 }\label{eq:mat_prol_P1}
\end{equation}

Indeed, the basis functions with respect to the reference
interval $[0,1]$ are
\begin{eqnarray*}
\hat{\varphi}_1(\hat{x}) &=&1-\hat{x},   \\
\hat{\varphi}_2(\hat{x}) &=& \hat{x},
\end{eqnarray*}
and, according to
Equation      (\ref{eq:comb_lin}), the $\varphi_i^{2h}$ coefficients are

\[
\begin{array}{cccc}
\hat{\varphi}_2(1/2)={1/2},&  \hat{\varphi}_2(1)=1,& \hat{\varphi}_1(1/2)=1/2,
\end{array}
\]
giving the columns of the matrix in Equation      (\ref{eq:mat_prol_P1}).
However, we can think the prolongation matrix  above as the
product of the Toeplitz matrix generated by the polynomial
$p_{_{\mathbb{Q}_1}}(\vartheta)=1+\cos(\vartheta)$ and a suitable
cutting matrix, (see \cite{Fiorentino-Serra} for the terminology and the related notation) defined as

\begin{equation} K_{m_{s+1} \times m_s}=
{ \left [
\begin{array}{cccccccccccc}
0 & 1 & 0 \\
  &   & 0 & 1      & 0 \\
  &   &   &        & \ddots & \ddots & \ddots \\
  &   &   &        &        & 0      & 1      & 0 \\
\end{array}
\right ],} \label{eq:cutting_matrix}
\end{equation}

i.e.,%\begin{equation*}
$P^{m_s}_{m_{s+1}} = (P^{m_{s+1}}_{m_s} )^T =
A_{m_s}(p_{_{\mathbb{Q}_1}}) (K_{m_{s+1} \times m_s})^T$.
%\end{equation*}
%
\ \\
Two-grid/Multigrid convergence with the above defined restriction/prolongation operators and a simple smoother (for
instance, Gauss--Seidel iteration) is a classical result, both from the point of view of the literature of approximated differential operators \cite{Hack} and from the  point of view of the literature of structured matrices \cite{ADS,Fiorentino-Serra}.\\

{It should be noticed that the coarse grid correction is chosen in such a way that the corresponding iteration, which is not convergent, is very fast in the subspace where the original matrix is ill-conditioned: hence any simple smoother is usually effective, since it converges very fast in the well conditioned subspace, where conversely the coarse grid correction is non convergent. In other words the magic of the multigrid is the combination of at least two iterative solvers whose iteration matrices have spectral complementarity (see e.g. \cite{multi-iterative,ST-SISC} and references therein).}

%%%%%%%%%%%%%%%%%%%%%%%%%%%%%%%%%%%%%%%%%%%%%%%%

%%%%%%%%%%%%%%%%%%%%%%%%%%%%%%%%%%%%

In the first panel of Table \ref{tab:Pk1D}, we report the number of iterations needed for achieving the predefined tolerance $10^{-6}$, when increasing the matrix size in the setting of the current subsection. Indeed, we use $A_{m_s}(p_{_{\mathbb{Q}_1}}) (K_{m_{s+1} \times m_s})^T$ and its transpose as restriction and prolongation operators and Gauss--Seidel as a smoother. We {highlight} that only one iteration of pre-smoothing and only one iteration of post-smoothing are employed in the current numerics.  {Therefore, by considering the statement in Remark \ref{rmk:condition_scalar} and the subsequent explanation,  there is no surprise in observing that  the number of iterations needed for the two-grid and V-cycle convergence remains almost constant, when we increase the matrix size. The numerical results confirm the {predicted} optimality of the methods in the present scalar setting}.

\begin{table}[H]
	\centering
	\caption{Number of iterations needed for the convergence of the two-grid and V-cycle methods for $k = 1,2,3$ in one dimension with $a(x)\equiv 1$ and  ${\rm tol}=1 \times 10^{-6}$.} \label{tab:Pk1D}
	
\scalebox{0.95}[0.95]{
	\begin{tabular}{cccccccccc}
		\toprule
		& \multicolumn{2}{c}{\boldmath{$k=1$}} & \multicolumn{2}{c}{\boldmath{$k=2$}} & \multicolumn{2}{c}{\boldmath{$k=3$}}\\
		\midrule
		\textbf{$\#$ Subintervals} & \textbf{TGM} & \textbf{V-Cycle}  & \textbf{TGM} & \textbf{V-Cycle} & \textbf{TGM} & \textbf{V-Cycle}\\
		\midrule
		8   & 5 & 5 & 7 & 7 & 9 & 9 \\
		16  & 6 & 7 & 7 & 7 & 9 & 9 \\
		32  & 7 & 7 & 7 & 7 & 9 & 9 \\
		64  & 7 & 7 & 7 & 7 & 9 & 9 \\
		128 & 6 & 7 & 7 & 7 & 9 & 9 \\
		256 & 6 & 7 & 7 & 7 & 9 & 9 \\
		512 & 6 & 7 & 7 & 7 & 9 & 9 \\
	\bottomrule
	\end{tabular}}
\end{table}
\subsection{\texorpdfstring{$\mathbb{Q}_2$}{Q2} Case}
Let us  consider the case of $\mathbb{Q}_2$ Finite Elements, where
we have that the basis functions with respect to the reference
interval $[0,1]$ are

\begin{eqnarray*}
\hat{\varphi}_1(\hat{x}) &=& 2\hat{x}^2-3\hat{x} +1,   \\
\hat{\varphi}_2(\hat{x}) &=& -4\hat{x}^2+4\hat{x}, \\
\hat{\varphi}_3(\hat{x}) &=& 2\hat{x}^2 -\hat{x}.
\end{eqnarray*}

For the sake of simplicity, let us consider the case of ${\cal
T}_{2h}$ partitioning with five equispaced points (three~internal points)
and ${\cal T}_{h}$ partitioning with nine equispaced points (seven
internal points) obtained from ${\cal T}_{2h}$ by considering the
midpoint of each subinterval.

Thus, with respect to %Please check whether need to indent? Yes.
Equation      (\ref{eq:comb_lin}), the $\varphi_1^{2h}$ coefficients are
$$
\begin{array}{cccc}
\hat{\varphi}_2(1/4)=3/4, & \hat{\varphi}_2(1/2)=1,& \hat{\varphi}_2(3/4)=3/4, &
\hat{\varphi}_2(1)=0,
\end{array}
$$

while the $\hat{\varphi}_2^{2h}$ coefficients are

\[
\begin{array}{llll}
\hat{\varphi}_3(1/4)= -1/8, & \hat{\varphi}_3(1/2)=0, &
\hat{\varphi}_3(3/4)=3/8,   & \hat{\varphi}_3(1)=1, \\
\hat{\varphi}_1(1/4)=3/8,   & \hat{\varphi}_1(1/2)=0, & \hat{\varphi}_1(3/4)=-1/8, &
\hat{\varphi}_1(1)=0,
\end{array}
\]
and so on again as for that first couple of basis functions.
Notice also that, to evaluate the coefficients, for the sake of simplicity, we are referring to the basis functions on
the reference interval, as depicted in Figure
\ref{fig:costr_prol_P2}. As a conclusion, the obtained prolongation
matrix is as follows
\begin{equation}
 P_{h\times 2h} =
 P_3^7= {
 \left [
\renewcommand{\arraystretch}{1.2}
\begin{array}{rrr}
\displaystyle \frac{3}{4} & \displaystyle -\frac{1}{8} & \\
1           & 0            & \\
\displaystyle \frac{3}{4} & \displaystyle \frac{3}{8}  & \\
 0          & 1            & \\
            & \displaystyle \frac{3}{8}  & \displaystyle\frac{3}{4} \\
            & 0            & 1 \\
            & \displaystyle -\frac{1}{8} & \displaystyle\frac{3}{4} \\
\end{array}
\right ]
 .}\label{eq:mat_prol_P2}
\end{equation}

%-------------------------------------------------------------------------------
\begin{figure}[H]
\centering
\includegraphics[width=0.95\textwidth]{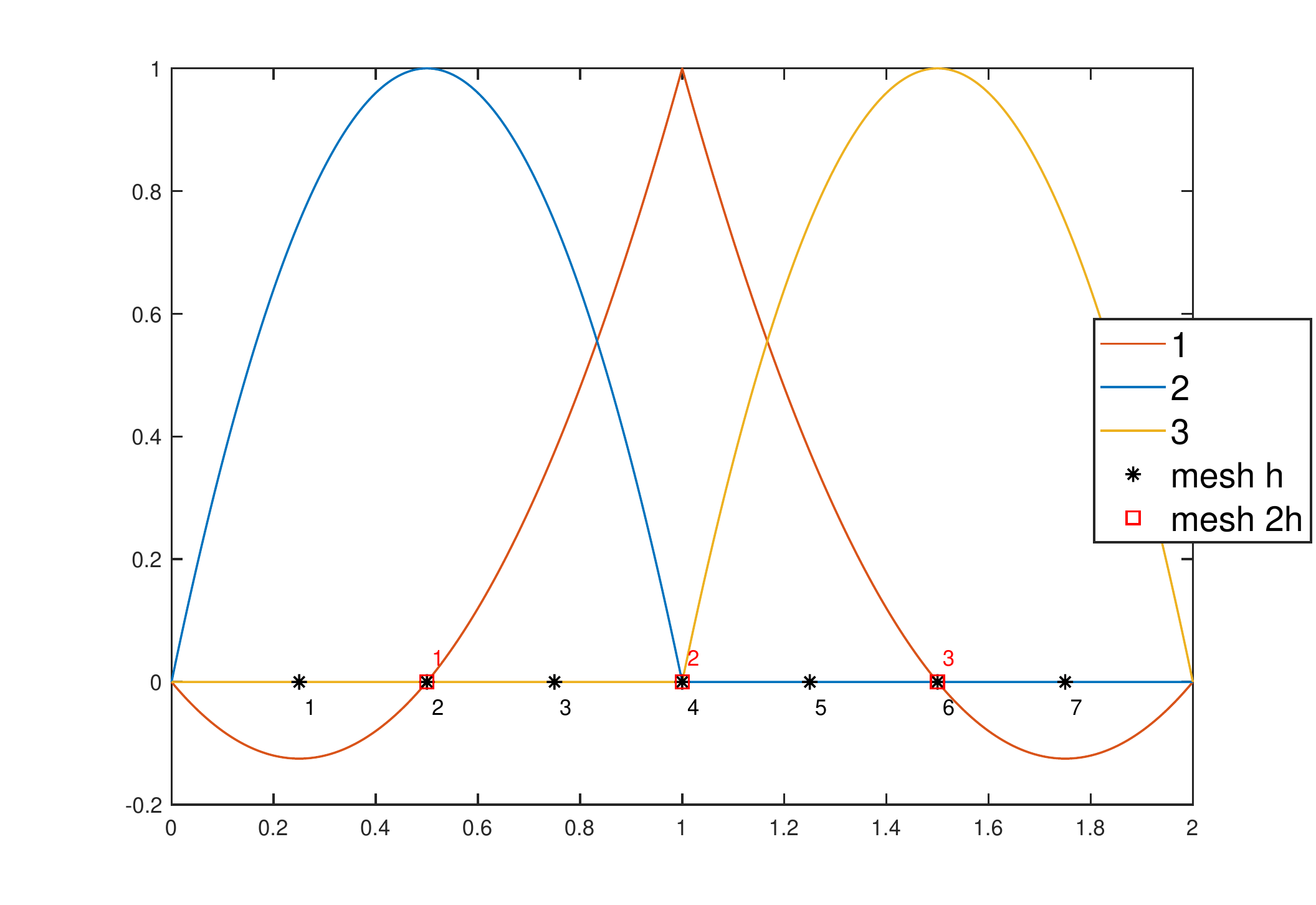}
\caption{Construction of the $\mathbb{Q}_2$ prolongation
operator: basis functions on the reference element.} \label{fig:costr_prol_P2}
\end{figure}
%-------------------------------------------------------------------------------
%-------------------------------------------------------------------------------
%
Hereafter, we are interested in setting such a geometrical
multigrid strategy, {proposed in \cite{Brandt,Hack,Hack2},} in the framework of the more general algebraic
multigrid theory and in particular in the one  driven by the
matrix symbol analysis. To this end, we   represent the
prolongation operator quoted above as the product of a Toeplitz
matrix generated by a polynomial $\mathbf{p}_{_{\mathbb{Q}_2}}$ and a
suitable cutting matrix. {We recall} that the Finite Element  stiffness matrix
could be thought as a principal submatrix of a Toeplitz matrix
generated by the matrix-valued symbol {that, from Equation      (\ref{fr(theta)}), has the compact form}

\begin{equation}
 \mathbf{f}_{\mathbb{Q}_2}(\vartheta) =\left [
 \begin{array}{rr}
                     \frac{16}{3} &           - \frac{8}{3} (1+e^{\hat{i} \vartheta}) \\
& \\
- \frac{8}{3} (1+e^{-\hat{i} \vartheta})    &             \frac{14}{3}+ \frac{1}{3} (e^{\hat{i} \vartheta}+e^{\hat{i} \vartheta}) \\
\end{array} \right  ].
\end{equation}
{Thus,} it is quite natural to look for a matrix-valued symbol for the
polynomial $\mathbf{p}_{_{\mathbb{Q}_2}}$ as well. In addition, the cutting
matrix is also formed through the Kronecker product of the scalar
cutting matrix in Equation      (\ref{eq:cutting_matrix}) and the identity matrix of order 2, so that
\begin{equation*}
P^{m_s}_{m_{s+1}} = (P^{m_{s+1}}_{m_s} )^T =
A_{m_s}(p_{_{\mathbb{Q}_2}}) ((K_{m_{s+1} \times m_s})^T \otimes
I_2).
%p_{n \times k} = A_n(p_{_{\mathbb{Q}_2}}) (K_n^T \otimes I_2).
\end{equation*}

Taking into account the action of the cutting matrix $(K_{m_{s+1}
\times m_s})^T \otimes I_2$, we can easily identify from
Equation      (\ref{eq:mat_prol_P2}) the generating polynomial as

\begin{equation}
\mathbf{p}_{_{\mathbb{Q}_2}}(\vartheta)=K_0+K_1 e^{\hat{i}\vartheta}+K_{-1}
e^{-\hat{i} \vartheta}+K_2 e^{2\hat{i}\vartheta}+K_{-2} e^{-2\hat{i}
\vartheta} \label{eq:p2}.
\end{equation}
where

\[
K_0={ \left [  \begin{array}{rr} \frac{3}{4} & \frac{3}{8} \\ 0 & 1 \\
\end{array}\right ]},\  K_1={ \left [  \begin{array}{rr} 0 & \frac{3}{8} \\ 0 & 0\\
\end{array}\right ]},\
K_{-1}={ \left [  \begin{array}{rr} \frac{3}{4} & -\frac{1}{8}\\ 1 & 0  \\
\end{array}\right ]},\  K_2={ \left [  \begin{array}{rr} 0 & -\frac{1}{8} \\ 0 & 0 \\
\end{array}\right ]},\  K_{-2}=O_{2},
\]
 that is

\begin{equation*}
 \mathbf{p}_{_{\mathbb{Q}_2}}(\vartheta) =\left [
 \begin{array}{rr}
 \frac{3}{4}(1+e^{-\hat{i}\vartheta}) & \quad
 \frac{3}{8} (1 + e^{\hat{i} \vartheta}) - \frac{1}{8}( e^{-\hat{i}  \vartheta} +e^{2\hat{i} \vartheta}) \\
 \\
                    e^{-\hat{i} \vartheta}&                    1 \\
\end{array} \right  ].
\end{equation*}

A very preliminary analysis, just by computing the determinant of
$\mathbf{p}_{_{\mathbb{Q}_2}}(\vartheta)$ shows there is a zero of third
order in the mirror point $\vartheta=\pi$, being
\[
\textrm{det}(\mathbf{p}_{_{\mathbb{Q}_2}}(\vartheta))=\frac{1}{8}
e^{-2\hat{i} \vartheta} (e^{\hat{i}  \vartheta} +1)^3.
\]
{We highlight that our choices are in agreement with the mathematical conditions set in Items \textbf{(A)} and \textbf{(B)}. Condition {\bf (C)} is violated and we discuss it in Remark \ref{rmk:singular_commutator}.
Nevertheless, it is possible to derive the following TGM convergence and optimality sufficient conditions that should be verified by $\mathbf{f}$ and $\mathbf{p}=\mathbf{p}_{_{\mathbb{Q}_2}}$,  exploiting the idea in the proof of the main result of \cite{multi-block}:}

\begin{eqnarray}
\mathbf{p}(\vartheta)^*\mathbf{p}(\vartheta)+\mathbf{p}(\vartheta+\pi)^* \mathbf{p}(\vartheta+\pi) &>& {O_k}\, \textrm{ for all } \vartheta \in [0,2\pi] \label{eqn:cond_1} \\
R(\vartheta) &\le& \gamma I_{2k} \label{eqn:cond_2}
\end{eqnarray}
with

\[
R(\vartheta)=
  \begin{bmatrix}
\mathbf{f}(\vartheta) & \\
& \mathbf{f}(\vartheta+\pi)
\end{bmatrix}^{-\frac{1}{2}}
    \left( I_{2k} - \left[  \begin{array}{c} \!\!\!  \mathbf{p}(\vartheta) \!\!\! \\  \!\!\!
    \mathbf{p}(\vartheta+\pi)\!\!\!
    \end{array}\right ] \!\!\!
    \,\, q(\vartheta) \,\, \!\!\!
    \left[ \begin{array}{cc} \!\!\! \mathbf{p}(\vartheta)^* \!\!\! & \!\!\! \mathbf{p}(\vartheta+\pi)^* \!\!\! \end{array}\right ]
\right )
 \begin{bmatrix}
\mathbf{f}(\vartheta) & \\
& \mathbf{f}(\vartheta+\pi)
\end{bmatrix}^{-\frac{1}{2}},
\]

{where $q(\vartheta) = \left[ \mathbf{p}(\vartheta)^*\mathbf{p}(\vartheta)+\mathbf{p}(\vartheta+\pi)^* \mathbf{p}(\vartheta+\pi) \right]^{-1}$,  $\gamma>0$  is a constant independent on $n$. The condition in} Equation      (\ref{eqn:cond_2})   requires the matrix-valued function
$R(\vartheta)$ to be uniformly bounded in the spectral norm. These conditions are { obtained from the proof of the main convergence result in  \cite{multi-block}, where, after several numerical derivations, {it was concluded that the above conditions are the final requirements needed. } }

To this end, we have explicitly formed the matrices involved in the
conditions in   Equations      (\ref{eqn:cond_1}) and        (\ref{eqn:cond_2}) and computed
their eigenvalues for $\vartheta \in [0, 2\pi]$. The results are
reported in Figure \ref{fig:check12_cond_prol_P2} and are in
perfect agreement with the theoretical requirements.

In the second panel of Table \ref{tab:Pk1D}, we report the number of iterations needed for achieving the predefined tolerance
$10^{-6}$,   when increasing the matrix size in the setting of the current subsection. Indeed, we use $A_{m_s}(p_{_{\mathbb{Q}_2}}) (K_{m_{s+1} \times m_s})^T$ and its transpose as restriction and prolongation operators and Gauss--Seidel as a smoother. Again, we remind that only one iteration of pre-smoothing and only one iteration of post-smoothing are employed in our numerical setting.

As expected, we observe that the number of iterations needed for the two-grid convergence remains constant, when we increase the matrix size, numerically confirming the optimality of the~method.

Moreover, we notice that the V-cycle method is characterised by optimal convergence properties. Although this behaviour is expected from the point of view of differential approximated operators, it is interesting in the setting of algebraic multigrid methods. Indeed, constructing an optimal V-cycle method for matrices in this block setting might require a specific analysis of the spectral properties of the restricted operators, (see \cite{multi-block}).
%%%%%
%-------------------------------------------------------------------------------
%-------------------------------------------------------------------------------
\begin{figure}[H]
\centering
\includegraphics[width=\textwidth,trim=3.5cm 3cm 3.5cm 3cm,clip]{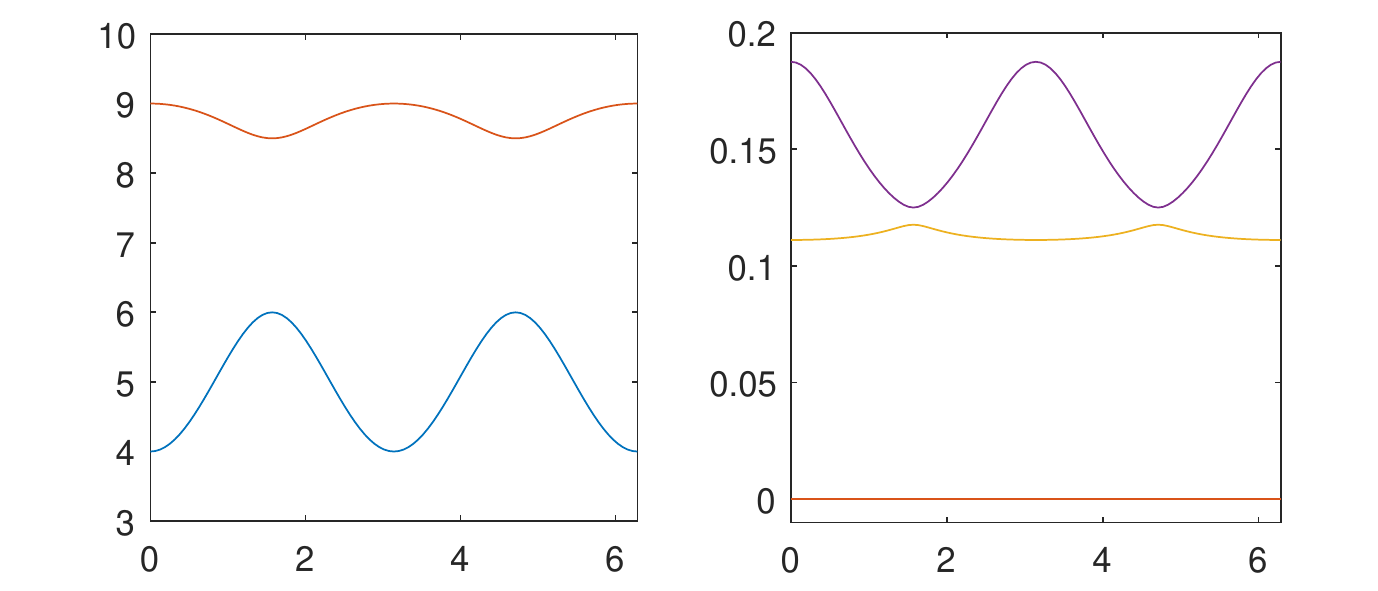} \,
\caption{Check of conditions for $\mathbb{Q}_2$ prolongation: (\textbf{left}) the plot of the eigenvalues of $\mathbf{p}(\vartheta)^*\mathbf{p}(\vartheta)+\mathbf{p}(\vartheta+\pi)^* \mathbf{p}(\vartheta+\pi)$ for $\vartheta \in [0, 2\pi]$; and  (\textbf{right}) the plot of the eigenvalues of $R(\vartheta)$ for $\vartheta \in [0, 2\pi]$.}
\label{fig:check12_cond_prol_P2}
\end{figure}
%
%\clearpage
%-------------------------------------------------------------------------------
%-------------------------------------------------------------------------------
%

\subsection{\texorpdfstring{$\mathbb{Q}_3$}{Q3} Case}

Hereafter, we briefly summarize the case of $\mathbb{Q}_3$ Finite
Elements, following the very same path we already considered in
the previous section for $\mathbb{P}_2$ Finite Elements. The basis functions with respect to the reference interval $[0,1]$ are

\begin{eqnarray}
\hat{\varphi}_1(\hat{x}) &=& - \frac{9}{2}\hat{x}^3 + 9\hat{x}^2 - \frac{11}{2}\hat{x} + 1,  \nonumber \\
\hat{\varphi}_2(\hat{x}) &=& \frac{27}{2}\hat{x}^3 - \frac{45}{2}\hat{x}^2 + 9\hat{x},  \\
\hat{\varphi}_3(\hat{x}) &=& - \frac{27}{2} \hat{x}^3 + 18\hat{x}^2 - \frac{9}{2}\hat{x}, \nonumber \\
\hat{\varphi}_4(\hat{x}) &=& \frac{9}{2}\hat{x}^3 - \frac{9}{2}\hat{x}^2 +
\hat{x}. \nonumber
\end{eqnarray}

For the sake of simplicity, let us consider the case of ${\cal
T}_{2h}$ partitioning with seven equispaced points (five~internal points)
and ${\cal T}_{h}$ partitioning with 13 equispaced points (11
internal points) obtained from ${\cal T}_{2h}$ by considering the
midpoint of each subinterval.

Thus, with respect to %Whether need to indent? Yes.
Equation      (\ref{eq:comb_lin}), (see also Figure \ref{fig:costr_prol_P3}),
the $\varphi_1^{2h}$ coefficients are

\[
\begin{array}{llll}
\hat{\varphi}_2(1/6)=15/16, & \hat{\varphi}_2(1/3)=1,&
\hat{\varphi}_2(1/2)=9/16, \\
\hat{\varphi}_2(2/3)=0, & \hat{\varphi}_2(5/6)=-5/16, & \hat{\varphi}_2(1)=0,
\end{array}
\]

while, the $\varphi_2^{2h}$ coefficients are

\[
\begin{array}{llll}
\hat{\varphi}_3(1/6)= -5/16, & \hat{\varphi}_3(1/3)=0, &
\hat{\varphi}_3(1/2)=9/16, \\
\hat{\varphi}_3(2/3)= 1, & \hat{\varphi}_3(5/6)=15/16, & \hat{\varphi}_3(1)=0,
\end{array}
\]

and the $\varphi_3^{2h}$ coefficients are

\[
\begin{array}{llll}
\hat{\varphi}_4(1/6)= 1/16, & \hat{\varphi}_4(1/3)=0, &
\hat{\varphi}_4(1/2)=-1/16, \\
\hat{\varphi}_4(2/3)= 0, &
\hat{\varphi}_4(5/6)=5/16, & \hat{\varphi}_4(1)=1,\\
\hat{\varphi}_1(1/6)= 5/16, & \hat{\varphi}_1(1/3)=0,&
\hat{\varphi}_1(1/2)=-1/16, \\
\hat{\varphi}_1(2/3)=0, & \hat{\varphi}_1(5/6)=1/16, & \hat{\varphi}_1(1)=0.
\end{array}
\]
Thus, the obtained prolongation matrix is as follows:
\begin{equation}
P_{h\times 2h} = P_5^{11}=  {\left [
 \renewcommand{\arraystretch}{1.2}
\begin{array}{rrrrr}
\displaystyle \frac{15}{16} & \displaystyle -\frac{5}{16} & \displaystyle \frac{1}{16}           &             & \\
1 & 0 & 0 & & \\
\displaystyle \frac{9}{16} & \displaystyle \frac{9}{16} & \displaystyle -\frac{1}{16}           &             & \\
 0          & 1            & 0 & & \\
\displaystyle -\frac{5}{16} & \displaystyle \frac{15}{16} & \displaystyle \frac{5}{16}           &             & \\
 0           & 0           & 1 & \\
& & \displaystyle \frac{5}{16} & \displaystyle \frac{15}{16} & \displaystyle -\frac{5}{16}    \\
        &    & 0       & 1  & 0 \\
& & \displaystyle -\frac{1}{16} & \displaystyle \frac{9}{16} & \displaystyle \frac{9}{16}    \\
        &    & 0       & 0  & 1 \\
& & \displaystyle \frac{1}{16} & \displaystyle -\frac{5}{16} & \displaystyle \frac{15}{16}    \\
\end{array}
\right ]}. \label{eq:mat_prol_P3}
\end{equation}
Therefore, taking into consideration that the stiffness matrix is a principal
submatrix of the Toeplitz matrix generated by the matrix-valued
function

\begin{equation}
 \mathbf{f}_{\mathbb{Q}_3}(\vartheta) =\left [
  \renewcommand{\arraystretch}{1.2}
 \begin{array}{rrr}
   \frac{54}{5}   &      -\frac{297}{40} &  \frac{27}{20}- \frac{189}{40}e^{\hat{i} \vartheta}  \\
   %\\
  -\frac{297}{40} &       \frac{54}{5}   &  -\frac{189}{40} + \frac{27}{20}e^{\hat{i} \vartheta} \\
  %\\
\frac{27}{20}- \frac{189}{40}e^{-\hat{i} \vartheta}   &
-\frac{189}{40} + \frac{27}{20}e^{-\hat{i} \vartheta} & \frac{37}{5}
- \frac{13}{40}(e^{\hat{i} \vartheta}+ e^{-\hat{i} \vartheta})
\end{array}
\right  ],
\end{equation}

we are looking for the matrix-valued symbol $\mathbf{p}_{_{\mathbb{Q}_3}}$
as well.
By defining
\begin{equation*}
P^{m_s}_{m_{s+1}} = (P^{m_{s+1}}_{m_s} )^T =
A_{m_s}(p_{_{\mathbb{Q}_3}}) ((K_{m_{s+1} \times m_s})^T \otimes
I_3)
\end{equation*}
it is easy to identify the generating polynomial as
\begin{equation}
\mathbf{p}_{_{\mathbb{Q}_3}}(\vartheta)=K_0+K_1 e^{\hat{i}\vartheta}+K_{-1}
e^{-\hat{i} \vartheta}+K_2 e^{2\hat{i}\vartheta}+K_{-2} e^{-2\hat{i}
\vartheta} \label{eq:p3},
\end{equation}

where
\[
K_0={ \left [  \begin{array}{rrr} 0 & 1 & 0\\ -\frac{5}{16}
&
\frac{15}{16} & \frac{5}{16} \\ 0 & 0 & 1\\
\end{array}\right ]},\
K_1={  \left [  \begin{array}{rrr} 0 & 0 & \frac{5}{16}\\ 0 &  0&  0\\ 0 &  0 & -\frac{1}{16}\\
\end{array}\right ]}, \
K_{-1}={  \left [  \begin{array}{rrr} \frac{15}{16}  &  -\frac{5}{16}  & \frac{1}{16} \\ 1 & 0 & 0\\ \frac{9}{16}  & \frac{9}{16}  & -\frac{1}{16}   \\
\end{array}\right ]},
\]
\[
K_2={  \left [  \begin{array}{rrr} 0  & 0  & 0\\ 0 & 0  & \frac{1}{16}\\ 0  & 0  & 0\\
\end{array}\right ]},\ \  K_{-2}=O_{3},
\]
 that is

 \begin{equation}
 \mathbf{p}_{_{\mathbb{Q}_3}}(\vartheta) =\left[
  \begin{array}{rrr}
 \frac{15}{16}e^{-\hat{i} \vartheta} & 1 - \frac{5}{16}e^{-\hat{i} \vartheta} &  \frac{1}{16}e^{-\hat{i} \vartheta} +
\frac{5}{16}e^{\hat{i} \vartheta} \\
e^{-\hat{i} \vartheta} - \frac{5}{16} &                \frac{15}{16} & \frac{5}{16}+\frac{1}{16}e^{2\hat{i} \vartheta} \\
  \frac{9}{16}e^{-\hat{i} \vartheta}&     \frac{9}{16}e^{-\hat{i} \vartheta} & 1 - \frac{1}{16}(e^{\hat{i} \vartheta}
  + e^{-\hat{i} \vartheta})
\end{array} \right].
\end{equation}

A trivial computation shows again {that} there is a zero of fourth
order in the mirror point $\vartheta=\pi$, being

\[
\textrm{det}(\mathbf{p}_{_{\mathbb{Q}_3}}(\vartheta))=\frac{1}{64}
e^{-3\hat{i} \vartheta} (e^{\hat{i}  \vartheta} +1)^4.
\]

However, the main goal is to verify the conditions in   Equations      (\ref{eqn:cond_1})
and        (\ref{eqn:cond_2}): we have explicitly formed the matrices
involved and computed their eigenvalues for $\vartheta \in [0,
2\pi]$. The results are in perfect agreement with the theoretical
requirements, (see Figure \ref{fig:check12_cond_prol_P3}). {This analysis links the geometric approach proposed in~\cite{Brandt,Hack,Hack2} to the novel algebraic multigrid methods for block-Toeplitz matrices.}

In the third panel of Table \ref{tab:Pk1D}, we report the number of iterations needed for achieving the predefined tolerance $10^{-6}$,  when increasing the matrix size in the setting of the current subsection. Indeed, we use $A_{m_s}(p_{_{\mathbb{Q}_3}}) (K_{m_{s+1} \times m_s})^T$ and its transpose as restriction and prolongation operators and Gauss--Seidel as a smoother (one iteration of pre-smoothing and one iteration of post-smoothing).%NOTE: Column? Row? With "panel" we mean the groups of three columns that are under the titles "k=1", "k=2", "k=3".

As expected, we observe that the number of iterations needed for the two-grid convergence remains constant, when we increase the matrix size, numerically confirming the optimality of the method. In analogy to the $\mathbb{Q}_2$ case, we notice that the V-cycle methods show the same optimal convergence properties.

Comparing the three panels in Table \ref{tab:Pk1D}, we also notice a mild dependency  of the number of iterations on the polynomial degree $k$. {
In addition, we can see in Tables \ref{tab:Pk21D} and \ref{tab:Pk31D} that the optimal behaviour of the two-grid and V-cycle methods for $k = 2,3$ remains unchanged if we test different tolerance values.}

\begin{table}[H]
    \centering

     \caption{{Number of iterations needed for the convergence of the two-grid and V-cycle methods for $k = 2$ in one dimension with $a(x)\equiv 1$ and ${\rm tol}=1 \times 10^{-2}, 1 \times 10^{-4},$ and $ 1 \times 10^{-8}$.}}  \label{tab:Pk21D}
      \scalebox{0.95}[0.95]{
    \color{black}{ \begin{tabular}{cccccccccc}
        \toprule
        & \multicolumn{2}{c}{\boldmath{${\rm tol}=1 \times 10^{-2}$}} & \multicolumn{2}{c}{\boldmath{${\rm tol}=1 \times 10^{-4}$}} & \multicolumn{2}{c}{\boldmath{${\rm tol}=1 \times 10^{-8}$}} \\
        \midrule
        \textbf{$\#$ Subintervals} & \textbf{TGM} & \textbf{V-Cycle}  & \textbf{TGM} & \textbf{V-Cycle}  & \textbf{TGM} & \textbf{V-Cycle}  \\
       \midrule
         8  & 3 & 3 & 5 & 5 & 8 & 8  \\
        16  & 3 & 3 & 5 & 5 & 9 & 9  \\
        32  & 3 & 3 & 5 & 5 & 9 & 10 \\
        64  & 3 & 3 & 5 & 5 & 9 & 10 \\
        128 & 3 & 3 & 5 & 5 & 9 & 10 \\
        256 & 3 & 3 & 5 & 5 & 9 & 10 \\
        512 & 3 & 3 & 5 & 5 & 9 & 10 \\
         \bottomrule
    \end{tabular}}}
\end{table}

\begin{table}[H]
    \centering
     \caption{ {Number of iterations needed for the convergence of the two-grid and V-cycle methods for $k = 3$ in one dimension with $a(x)\equiv 1$ and ${\rm tol}=1 \times 10^{-2}, 1 \times 10^{-4},$ and $ 1 \times 10^{-8}$.}} \label{tab:Pk31D}
\scalebox{0.95}[0.95]{
    \color{black}{\begin{tabular}{cccccccccc}
        \toprule
        & \multicolumn{2}{c}{\boldmath{${\rm tol}=1 \times 10^{-2}$}} & \multicolumn{2}{c}{\boldmath{${\rm tol}=1 \times 10^{-4}$}} & \multicolumn{2}{c}{\boldmath{${\rm tol}=1 \times 10^{-8}$}} \\
        \midrule
        \textbf{$\#$ Subintervals} & \textbf{TGM} & \textbf{V-Cycle} & \textbf{TGM} & \textbf{V-Cycle}  & \textbf{TGM} & \textbf{V-Cycle}  \\
         \midrule
         8  & 3 & 3 & 6 & 6 & 12  & 12  \\
        16  & 3 & 3 & 6 & 6 & 12  & 12  \\
        32  & 3 & 3 & 6 & 6 & 12  & 12  \\
        64  & 3 & 3 & 6 & 6 & 12  & 12  \\
        128 & 3 & 3 & 6 & 6 & 12  & 12  \\
        256 & 3 & 3 & 6 & 6 & 12  & 12  \\
        512 & 3 & 3 & 6 & 6 & 12  & 12  \\
         \bottomrule
    \end{tabular}}}
\end{table}

\begin{Remark}\label{rmk:singular_commutator}
{From} the cases analysed in this section, we notice that, even though $\mathbf{p}(0)$ and $\mathbf{p}(\pi)$ do not commute, the two-grid method is still convergent and optimal. The latter commutation property, along with Conditions {\bf (A)} and {\bf (B)} reported in {Remark }\ref{rmk:condition_scalar}, is sufficient to have optimal convergence of the two-grid method. This analysis reveals that commutativity is not a necessary property. Indeed, in our examples, we show  that the operator $R(\vartheta)$ is uniformly bounded in the spectral norm.

However, we notice that in all cases the commutator $S_{\mathbb{Q}_k}(\vartheta)=\mathbf{p}_{_{\mathbb{Q}_k}}(\vartheta)\mathbf{p}_{_{\mathbb{Q}_k}}(\vartheta+\pi)-\mathbf{p}_{_{\mathbb{Q}_k}}(\vartheta)\mathbf{p}_{_{\mathbb{Q}_k}}(\vartheta+\pi)$ computed in 0 is a singular matrix. In particular, computing our commutator matrix $S_{_{\mathbb{Q}_k}}(\vartheta)$ in $\vartheta=0$, we~obtain:

\[
S_{_{\mathbb{Q}_2}}(0)=\frac{1}{2}\begin{pmatrix}
            -1 & 1 \\
            -1 & 1 \\
          \end{pmatrix}, \qquad
S_{_{\mathbb{Q}_3}}(0)=\frac{1}{256}\begin{pmatrix}
       -462 & 330 & 132 \\
       -438 & 354 & 84 \\
       -378 & 270 & 108 \\
     \end{pmatrix},
\]

which are indeed singular matrices.
\end{Remark}
%-------------------------------------------------------------------------------
\begin{figure}[H]
\centering
\includegraphics[width=\textwidth,trim=2cm 1.5cm 2cm 1cm,clip]{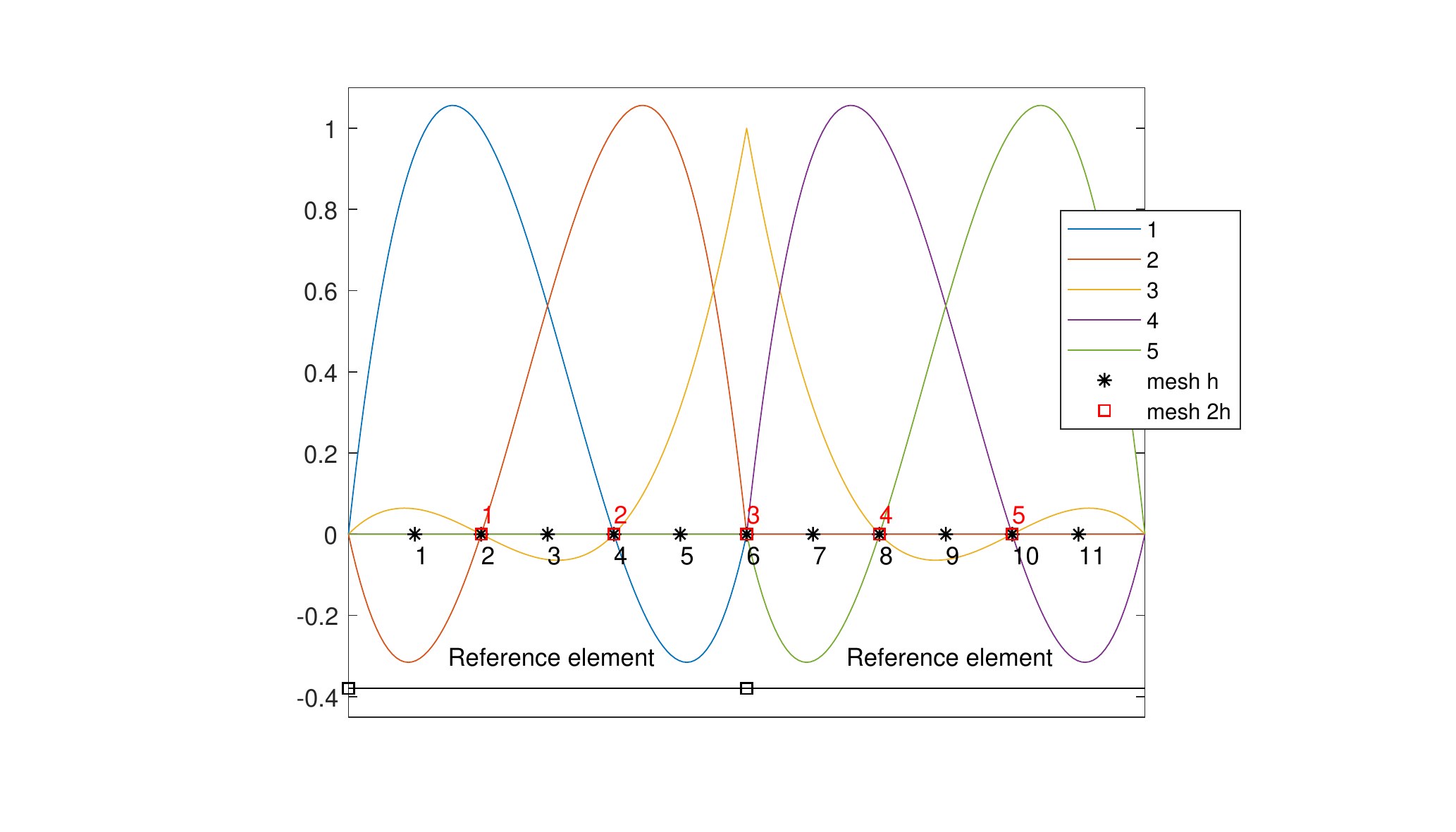}
\caption{Construction of the $\mathbb{Q}_3$ prolongation
operator: basis functions on the reference element.}
\label{fig:costr_prol_P3}
\end{figure}
\begin{figure}[H]
\centering
\includegraphics[width=0.47\textwidth,trim=3.5cm 1cm 3.5cm 1cm,clip]{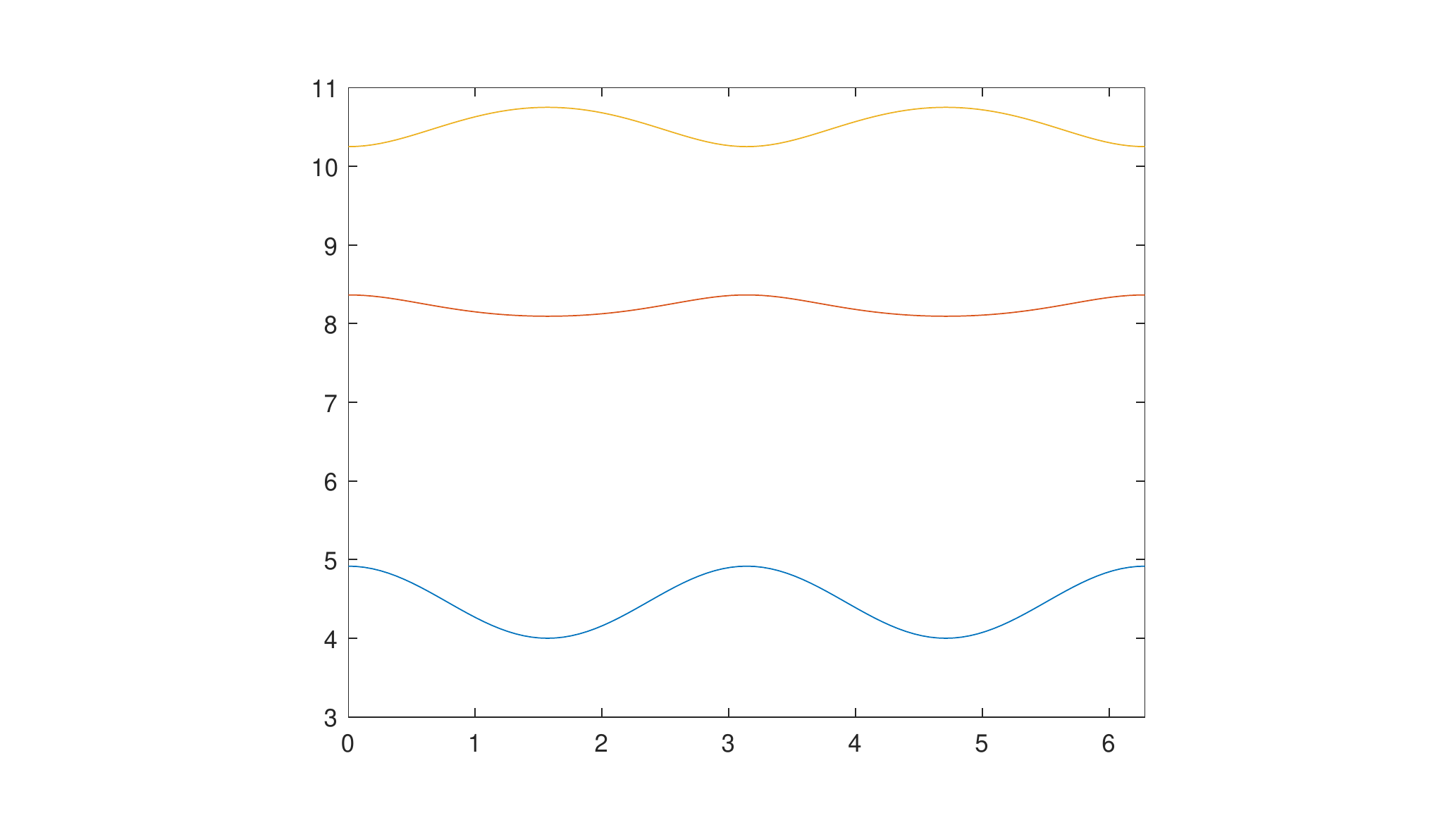} \,
\includegraphics[width=0.47\textwidth,trim=3.5cm 1cm 3.5cm 1cm,clip]{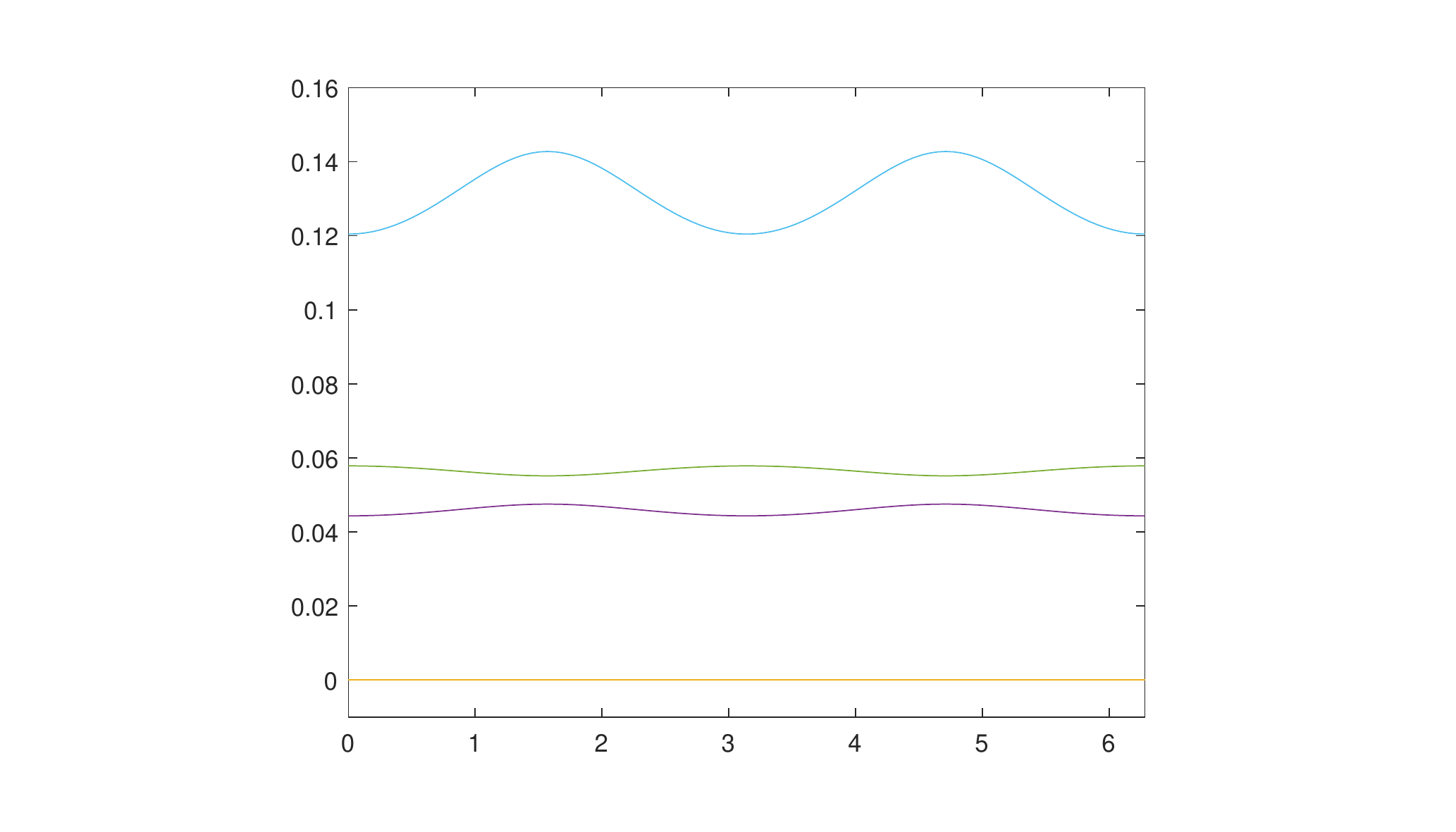}
\caption{Check of conditions for $\mathbb{Q}_3$ prolongation{:}  {(\textbf{left})} the plot of the eigenvalues of $p(\vartheta)^Hp(\vartheta)+p(\vartheta+\pi)^H p(\vartheta+\pi)$ for $\vartheta \in [0, 2\pi]${; and}  {(\textbf{right})} the plot of the eigenvalues of $R(\vartheta)$ for $\vartheta \in [0, 2\pi]$.}
\label{fig:check12_cond_prol_P3}
\end{figure}

\begin{Remark}\label{rmk:numer larger d}
It is worth stressing that the results hold also in dimension $d\ge 2$. In fact, interestingly, we observe that the dimensionality $d$ does not affect the efficiency of the proposed method, as well shown in Table \ref{tab:Qk2D} for the case $d=2$.
We finally remind that the tensor structure of the resulting matrices highly facilitates the generalization and extension of the numerical code to the case of $d\ge 2$.  {Indeed,    the prolongation operators in the multilevel setting are constructed by a proper tensorization of those in 1D.}
 \end{Remark}
 \begin{table}[H]
\small \centering
 \caption{Number of iterations needed for the convergence of the two-grid and V-cycle methods for $k = 1,2,3$ in  dimension $d=2$ with $a(\mathbf{x})\equiv 1$.} \label{tab:Qk2D}
\begin{tabular}{cccccccccccc}
\toprule
 & \multicolumn{2}{c}{\boldmath{$k=1$}} & & \multicolumn{2}{c}{\boldmath{$k=2$}} & & \multicolumn{2}{c}{\boldmath{$k=3$}}\\
\midrule
\textbf{$\#$} & \textbf{Two} & \textbf{V-} &\textbf{ $\#$} &\textbf{ Two} & \textbf{V-}  & \textbf{$\#$} & \textbf{Two} & \textbf{V- } \\
\textbf{Nodes} & \textbf{Grid} & \textbf{Cycle}  & \textbf{Nodes} & \textbf{Grid} & \textbf{Cycle}  & \textbf{Nodes} & \textbf{Grid} & \textbf{Cycle}  \\
\midrule
$7^2$   & 5 & 5  & $15^2$  & 6 & 6 &  $23^2$  & 7 & 7  \\
$15^2$  & 5 & 6  & $31^2$  & 6 & 6 &  $47^2$  & 7 & 7  \\
$31^2$  & 5 & 6  & $63^2$  & 6 & 6 &  $95^2$  & 7 & 7  \\
$63^2$  & 5 & 6  & $127^2$ & 6 & 6 &  $191^2$ & 7 & 7   \\
$127^2$ & 5 & 6  & $255^2$ & 6 & 6 &  $383^2$ & 7 & 7  \\
\bottomrule
\end{tabular}
\end{table}

{ Furthermore, we highlight that the presented analysis for  $a\equiv 1$ can be easily extended to the case on non-constant coefficients $a(x)\neq 1$ in 1D and $a(x,y)\neq 1$ in 2D, since, following a geometric approach, the prolongation operators for the general variable coefficients remain unchanged. In Tables~\ref{tab:Pk21D_acoeff} and \ref{tab:Pk22D_acoeff}, we show the number of iterations needed for the convergence of the two-grid and V-cycle methods for $k = 2$ in one and two dimensions for different values of $a \not\equiv 1$.}
\begin{table}[H]
    \centering
     \caption{{Number of iterations needed for the convergence of the two-grid and V-cycle methods for $k = 2$ in one dimension with $a(x)=e^x$, $a(x)=10x+1$,   $a(x)=|x-1/2|+1$,  and ${\rm tol}=1 \times 10^{-6}$.}} \label{tab:Pk21D_acoeff}
\scalebox{0.95}[0.95]{
    \color{black}{\begin{tabular}{cccccccccc}
        \toprule
        & \multicolumn{2}{c}{\boldmath{$a(x)=e^x$}} & \multicolumn{2}{c}{\boldmath{$a(x)=10x+1$}} & \multicolumn{2}{c}{\boldmath{$a(x)=|x-1/2|+1$}}\\
       \midrule
        \textbf{$\#$ Subintervals} & \textbf{TGM} & \textbf{V-Cycle}  & \textbf{TGM} & \textbf{V-Cycle}  & \textbf{TGM} & \textbf{V-Cycle}  \\
      \midrule
        8   & 7 & 7  & 11 & 11 &  7 & 7  \\
        16  & 7 & 7  & 9  & 12 &  7 & 7  \\
        32  & 7 & 8  & 7  & 14 &  7 & 7  \\
        64  & 7 & 8  & 7  & 14 &  7 & 7  \\
        128 & 7 & 8  & 7  & 15 &  7 & 7  \\
        256 & 7 & 8  & 7  & 15 &  7 & 7  \\
        512 & 7 & 8  & 7  & 14 &  7 & 7  \\
        \bottomrule
    \end{tabular}}}
\end{table}

\begin{table}[H]
\small \centering
 \caption{ {Number of iterations needed for the convergence of the two-grid and V-cycle  methods for $k = 2$
    in two dimensions with $a(x,y)=e^{(x+y)}$, $a(x,y)=10(x+y)+1$, $a(x,y)=|x-1/2|+|y-1/2|+1$,   $a(x,y)=1$ if $x\le 1/2$ and $y\le 1/2$, $5000$ otherwise, and ${\rm tol}=1 \times 10^{-6}$.}} \label{tab:Pk22D_acoeff}
\scalebox{0.95}[0.95]{
 \color{black}{\begin{tabular}{c cc| cc| cc| cc}
\toprule
  & \multicolumn{2}{c}{\boldmath{$a(x,y)=e^{(x+y)}$}} & \multicolumn{2}{c}{\boldmath{$10(x+y)+1$}} & \multicolumn{2}{c}{\boldmath{$\left|x-1/2\right|+\left|y-1/2\right|+1$}}
 & \multicolumn{2}{c}{\boldmath{$\left \{
\begin{array}{ll}
\textbf{1} & x,y\le \textbf{1/2} \\
\textbf{5000} & {\rm \textbf{otherwise}}
\end{array}\right.$}}\\
        \midrule

\textbf{$\#$}& \textbf{Two} & \textbf{V-}   &  \textbf{Two} & \textbf{V- }   & \textbf{ Two} & \textbf{V-}  &   \textbf{Two} & \textbf{V-}  \\
\textbf{Nodes} & \textbf{Grid} & \textbf{Cycle}   & \textbf{Grid} & \textbf{Cycle}   & \textbf{Grid }& \textbf{Cycle}  & \textbf{Grid} & \textbf{Cycle} \\
        \midrule
$7^2$   & 6 & 6 &  6 & 6 &  6 & 6   & 6 & 6   \\
$15^2$  & 6 & 6 &  6 & 6 &  6 & 6   & 6 & 6   \\
$31^2$  & 6 & 6 &  6 & 6 &  6 & 6   & 6 & 6   \\
$63^2$  & 6 & 6 &  6 & 6 &  6 & 6   & 6 & 6   \\
$127^2$ & 6 & 6 &  6 & 6 &  6 & 6   & 6 & 6   \\
        \bottomrule
\end{tabular}}}
\end{table}

\chapter*{Conclusion and Perspectives}
\addcontentsline{toc}{chapter}{Conclusion and Perspectives}
%The aim of this thesis is to show the crucial role that plays the symbol of the matrix
%The aim of this thesis is the use the information provided by the symbol of the symbol to design an ad hoc fast method
%is the study
%the aim of is to use the of the symbol in studding the spectral asymptotic and designing

The purpose of this thesis relied in showing the crucial role played the symbol of the matrix-sequence arising from the discretization of PDEs in providing the necessary information for analysing the spectra of this sequence of matrices and in developing an appropriate fast method to solve the linear problem associated. Now, we will take a quick look at what we have done in this thesis and then we will mention some of open problems.
%the issues treated in this thesis,

    %The information provided by the symbol is useful to estimate the conditioning of the matrix and helps to design a fast method to solve the problem associated.

\begin{itemize}
\item\textbf{Chapter 2}: {We provided the basic tools in order to deal with Toeplitz matrices sequences. In particular we defined the asymptotic distribution of matrix sequence in the sense of eigenvalues and singular values, followed by the notion of zero-distributed sequence.
Furthermore, we introduced the concept of Toeplitz and Circulant structures and we concluded with the asymptotic notion of GLT-sequences.   }
\item\textbf{Chapter 3}: {We considered a class of elliptic partial differential equations with Dirichlet boundary
 conditions, where the operator was $\mathrm{div} \left(-a(\mathbf{x}) \nabla \cdot\right)$ with $a$ continuous
 and positive  on $\overline \Omega$. For the numerical approximation, we employed the classical $\mathbb{P}_k$
 Finite Element Method, in the case of Friedrichs-Keller triangulations, which led to a sequence of matrices of
 increasing size. In other to analyse the spectral properties of the resulting matrices, we determined the
  associated spectral symbol, which is a function describing the spectral distribution (in the Weyl sense)
   when the matrix-size tends to infinity.} We studied in detail the case of constant coefficients and we
   have given a brief account in the more involved case of variable coefficients. The mathematical tools
   stem out from the Toeplitz technology and from the rather new theory of GLT matrix-sequences.
    Numerical results are shown for a practical evidence of the theoretical findings.

 \item\textbf{Chapter 4}: We considered  multigrid strategies for the resolution of linear systems arising from the $\mathbb{Q}_k$
Finite Elements approximation of one- and higher-dimensional elliptic partial differential equations with
Dirichlet boundary conditions and where the operator is $\mathrm{div} \left(-a(\mathbf{x}) \nabla \cdot\right)$, with $a$
continuous and positive  over $\overline \Omega$, $\Omega$ being an open and bounded subset of $\mathbb{R}^d$. While the analysis is in one dimension, the numerics are displayed in a higher dimension $d\ge 2$, showing an optimal behaviour in terms of the dependency on the matrix size and a substantial robustness with respect to the dimensionality $d$ and to the polynomial degree $k$, (see Remark \ref{rmk:numer larger d}).

Now, we list some perspectives that will be considered in the future researches:
\begin{enumerate}
\item It would be valuable to find a unified formula for the symbols of the $\mathbb{P}_k$ over a $d$ dimensional  cube for every $k,d$ presented in Chapter 3 as done for the case of  $\mathbb{Q}_k$ discretizations \cite{Q_k}.
\item In Chapter 4, we noticed that the symbol of our projector $\mathbf{p}_{_{\mathbb{Q}_k}}$ satisfies the conditions set in {Items} {\bf (A)} and {\bf (B)}, while {Condition} {\bf (C)} is violated. We can see from the mathematical derivations in \cite{multi-block} that the latter condition is, in fact, a technical one. We believe that condition {\bf (C)} can be replaced with a less restricted one, presumably in accordance with Remark \ref{rmk:singular_commutator}.
\item More in general, our GLT tools are quite broad and show a remarkable versatility. Thus as a general claim, we think that the same techniques could be used in imaging and inverse problems, since also in this context the spectral analysis and the study of ill-conditioned subspaces are of key importance, (see \cite{serra-imag}).
\end{enumerate}
\end{itemize}

\addcontentsline{toc}{chapter}{Bibliography}
\bibliography{reference}

\begin{thebibliography}{10}

\bibitem{ADS}
{\sc Arico, A., Donatelli, M., and Capizzano-Serra, S.}
\newblock V-cycle optimal convergence for certain (multilevel) structured
  linear systems.
\newblock {\em SIAM Journal on Matrix Analysis and Applications 26}, 1 (2004),
  186--214.

\bibitem{BBG}
{\sc Barbarino, G., Bianchi, D., and Garoni, C.}
\newblock Constructive approach to the monotone rearrangement of functions.
\newblock {\em Expositiones Mathematicae 40}, 1 (2022), 155--175.

\bibitem{Ba2-ETNA}
{\sc Barbarino, G., Garoni, C., and Serra-Capizzano, S.}
\newblock Block generalized locally {T}oeplitz sequences: theory and
  applications in the multidimensional case.
\newblock {\em Electronic Transactions on Numerical Analysis 53\/} (2020),
  113--216.

\bibitem{Ba1-ETNA}
{\sc Barbarino, G., Garoni, C., and Serra-Capizzano, S.}
\newblock Block generalized locally {T}oeplitz sequences: theory and
  applications in the unidimensional case.
\newblock {\em Electronic Transactions on Numerical Analysis 53\/} (2020),
  28--112.

\bibitem{BK}
{\sc Beckermann, B., and Kuijlaars, A.~B.}
\newblock Superlinear convergence of conjugate gradients.
\newblock {\em SIAM Journal on Numerical Analysis 39}, 1 (2001), 300--329.

\bibitem{BeSe}
{\sc Beckermann, B., and Serra-Capizzano, S.}
\newblock On the asymptotic spectrum of finite element matrix sequences.
\newblock {\em SIAM Journal on Numerical Analysis 45}, 2 (2007), 746--769.

\bibitem{krause}
{\sc Benedusi, P., Garoni, C., Krause, R., Li, X., and Serra-Capizzano, S.}
\newblock Space-time {FE}-{DG} discretization of the anisotropic diffusion
  equation in any dimension: the spectral symbol.
\newblock {\em SIAM Journal on Matrix Analysis and Applications 39}, 3 (2018),
  1383--1420.

\bibitem{bhatia1997matrix}
{\sc Bhatia, R.}
\newblock {\em Matrix Analysis}.
\newblock Springer-Verlag New York, 1997.

\bibitem{bini2011numerical}
{\sc Bini, D.~A., Iannazzo, B., and Meini, B.}
\newblock {\em Numerical solution of algebraic Riccati equations}.
\newblock SIAM, 2011.

\bibitem{bini2005numerical}
{\sc Bini, D.~A., Latouche, G., and Meini, B.}
\newblock {\em Numerical methods for structured Markov chains}.
\newblock Oxford University Press, 2005.

\bibitem{bottcher2000toeplitz}
{\sc B{\"o}ttcher, A., and Grudsky, S.~M.}
\newblock {\em Toeplitz matrices, asymptotic linear algebra and functional
  analysis}, vol.~67.
\newblock Springer, 2000.

\bibitem{BoSi}
{\sc B{\"o}ttcher, A., and Silbermann, B.}
\newblock {\em Introduction to Large Truncated Toeplitz Matrices}.
\newblock Springer Science \& Business Media, 1999.

\bibitem{Brandt}
{\sc Brandt, A.}
\newblock Multi-level adaptive solutions to boundary-value problems.
\newblock {\em Mathematics of {C}omputation 31}, 138 (1977), 333--390.

\bibitem{briggs2000multigrid}
{\sc Briggs, W.~L., Henson, V.~E., and McCormick, S.~F.}
\newblock {\em A multigrid tutorial}.
\newblock SIAM, 2000.

\bibitem{bunch1985stability}
{\sc Bunch, J.~R.}
\newblock Stability of methods for solving {T}oeplitz systems of equations.
\newblock {\em SIAM Journal on Scientific and Statistical Computing 6}, 2
  (1985), 349--364.

\bibitem{chan1996conjugate}
{\sc Chan, R.~H., and Ng, M.~K.}
\newblock Conjugate {G}radient methods for {T}oeplitz systems.
\newblock {\em SIAM {R}eview 38}, 3 (1996), 427--482.

\bibitem{chan1992circulant}
{\sc Chan, R.~H., and Yeung, M.-C.}
\newblock Circulant preconditioners for {T}oeplitz matrices with positive
  continuous generating functions.
\newblock {\em Mathematics of {C}omputation 58}, 197 (1992), 233--240.

\bibitem{ciarlet1990introduction}
{\sc Ciarlet, P.~G.}
\newblock {\em Introduction a l'analyse numerique matricielle et a
  l'optimisation PG Ciarlet}.
\newblock masson, 1988.

\bibitem{ciarlet}
{\sc Ciarlet, P.~G.}
\newblock {\em The finite element method for elliptic problems}.
\newblock SIAM, 2002.

\bibitem{IgA-book}
{\sc Cottrell, J.~A., Hughes, T.~J., and Bazilevs, Y.}
\newblock {\em Isogeometric analysis: toward integration of {CAD} and {FEA}}.
\newblock John Wiley \& Sons, 2009.

\bibitem{crbh06}
{\sc Cottrell, J.~A., Reali, A., Bazilevs, Y., and Hughes, T.~J.}
\newblock Isogeometric analysis of structural vibrations.
\newblock {\em Computer {M}ethods in {A}pplied {M}echanics and {E}ngineering
  195}, 41-43 (2006), 5257--5296.

\bibitem{del2014symbol}
{\sc Del~Prete, V., Di~Benedetto, F., Donatelli, M., and Serra-Capizzano, S.}
\newblock Symbol approach in a signal-restoration problem involving block
  {T}oeplitz matrices.
\newblock {\em Journal of {C}omputational and {A}pplied {M}athematics 272\/}
  (2014), 399--416.

\bibitem{DFS}
{\sc Di~Benedetto, F., Fiorentino, G., and Serra-Capizzano, S.}
\newblock {CG} preconditioning for {T}oeplitz matrices.
\newblock {\em Computers \& Mathematics with Applications 25}, 6 (1993),
  35--45.

\bibitem{multi-block}
{\sc Donatelli, M., Ferrari, P., Furci, I., Serra-Capizzano, S., and Sesana,
  D.}
\newblock Multigrid methods for block-{T}oeplitz linear systems: convergence
  analysis and applications.
\newblock {\em Numerical Linear Algebra with Applications 28}, 4 (2021), e2356.

\bibitem{cmame2}
{\sc Donatelli, M., Garoni, C., Manni, C., Serra-Capizzano, S., and Speleers,
  H.}
\newblock Robust and optimal multi-iterative techniques for {I}g{A} collocation
  linear systems.
\newblock {\em Computer Methods in Applied Mechanics and Engineering 284\/}
  (2015), 1120--1146.

\bibitem{cmame1}
{\sc Donatelli, M., Garoni, C., Manni, C., Serra-Capizzano, S., and Speleers,
  H.}
\newblock Robust and optimal multi-iterative techniques for {I}g{A} {G}alerkin
  linear systems.
\newblock {\em Computer Methods in Applied Mechanics and Engineering 284\/}
  (2015), 230--264.

\bibitem{our-mathcomp}
{\sc Donatelli, M., Garoni, C., Manni, C., Serra-Capizzano, S., and Speleers,
  H.}
\newblock Spectral analysis and spectral symbol of matrices in isogeometric
  collocation methods.
\newblock {\em Mathematics of Computation 85}, 300 (2016), 1639--1680.

\bibitem{our-sinum}
{\sc Donatelli, M., Garoni, C., Manni, C., Serra-Capizzano, S., and Speleers,
  H.}
\newblock Symbol-based multigrid methods for {G}alerkin {B}-spline isogeometric
  analysis.
\newblock {\em SIAM Journal on Numerical Analysis 55}, 1 (2017), 31--62.

\bibitem{molteni}
{\sc Donatelli, M., Molteni, M., Pennati, V., and Serra-Capizzano, S.}
\newblock Multigrid methods for cubic spline solution of two point (and {2D})
  boundary value problems.
\newblock {\em Applied Numerical Mathematics 104\/} (2016), 15--29.

\bibitem{schur-maya}
{\sc Dorostkar, A., Neytcheva, M., and Serra-Capizzano, S.}
\newblock Spectral analysis of coupled {PDE}s and of their {S}chur complements
  via generalized locally {T}oeplitz sequences in 2{D}.
\newblock {\em Computer Methods in Applied Mechanics and Engineering 309\/}
  (2016), 74--105.

\bibitem{hanke}
{\sc Engl, H.~W., Hanke, M., and Neubauer, A.}
\newblock {\em Regularization of Inverse Problems}, vol.~375.
\newblock Springer Science \& Business Media, 2000.

\bibitem{estatico2008superoptimal}
{\sc Estatico, C., and Serra-Capizzano, S.}
\newblock Superoptimal approximation for unbounded symbols.
\newblock {\em Linear {A}lgebra and its {A}pplications 428}, 2-3 (2008),
  564--585.

\bibitem{multigridQ_k}
{\sc Ferrari, P., Rahla, R.~I., Tablino-Possio, C., Belhaj, S., and
  Serra-Capizzano, S.}
\newblock Multigrid for $\bf{Q}_k$ {F}inite {E}lement {M}atrices using a
  (block) {T}oeplitz symbol approach.
\newblock {\em Mathematics 8}, 1 (2020), 5.

\bibitem{Fiorentino-Serra}
{\sc Fiorentino, G., and Serra-Capizzano, S.}
\newblock Multigrid methods for {T}oeplitz matrices.
\newblock {\em Calcolo 28}, 3-4 (1991), 283--305.

\bibitem{brezzi}
{\sc Fortin, M., and Brezzi, F.}
\newblock {\em Mixed and hybrid finite element methods}.
\newblock New York: Springer-Verlag, 1991.

\bibitem{garoni2014structured}
{\sc Garoni, C.}
\newblock {\em Structured matrices coming from PDE approximation theory:
  spectral analysis, spectral symbol and design of fast iterative solvers.}
\newblock PhD thesis, Universit{\`a} degli Studi dell'Insubria, 2014.

\bibitem{CCFSH}
{\sc Garoni, C., Manni, C., Pelosi, F., Serra-Capizzano, S., and Speleers, H.}
\newblock On the spectrum of stiffness matrices arising from isogeometric
  analysis.
\newblock {\em Numerische Mathematik 127}, 4 (2014), 751--799.

\bibitem{glt-book-I}
{\sc Garoni, C., and Serra-Capizzano, S.}
\newblock {\em Generalized {{L}}ocally {{T}}oeplitz sequences: Theory and
  applications}, vol.~1.
\newblock Springer, 2017.

\bibitem{glt-book-II}
{\sc Garoni, C., and Serra-Capizzano, S.}
\newblock {\em Generalized {{L}}ocally {{T}}oeplitz sequences: Theory and
  applications}, vol.~2.
\newblock Springer, 2018.

\bibitem{Q_k}
{\sc Garoni, C., Serra-Capizzano, S., and Sesana, D.}
\newblock Spectral analysis and spectral symbol of $d$-variate $\mathbf
  {{Q}}_p$ lagrangian {FEM} stiffness matrices.
\newblock {\em SIAM Journal on Matrix Analysis and Applications 36}, 3 (2015),
  1100--1128.

\bibitem{garoni2015tools}
{\sc Garoni, C., Serra-Capizzano, S., and Sesana, D.}
\newblock Tools for determining the asymptotic spectral distribution of
  non-{H}ermitian perturbations of {H}ermitian matrix-sequences and
  applications.
\newblock {\em Integral {E}quations and {O}perator {T}heory 81}, 2 (2015),
  213--225.

\bibitem{block-glt-II}
{\sc Garoni, C., Serra-Capizzano, S., and Sesana, D.}
\newblock Block {{G}}eneralized {{L}}ocally {{T}}oeplitz sequences: topological
  construction, spectral distribution results, and star-algebra structure.
\newblock In {\em Structured Matrices in Numerical Linear Algebra}. Springer,
  2019, pp.~59--79.

\bibitem{block-glt-I}
{\sc Garoni, C., Serra-Capizzano, S., and Sesana, D.}
\newblock Block {L}ocally {T}oeplitz sequences: construction and properties.
\newblock In {\em Structured Matrices in Numerical Linear Algebra}. Springer,
  2019, pp.~25--58.

\bibitem{tom-paper}
{\sc Garoni, C., Speleers, H., Ekstr{\"o}m, S.-E., Reali, A., Serra-Capizzano,
  S., and Hughes, T.~J.}
\newblock Symbol-based analysis of finite element and isogeometric {B}-spline
  discretizations of eigenvalue problems: Exposition and review.
\newblock {\em Archives of Computational Methods in Engineering 26}, 5 (2019),
  1639--1690.

\bibitem{gol-serra}
{\sc Golinskii, L., and Serra-Capizzano, S.}
\newblock The asymptotic properties of the spectrum of nonsymmetrically
  perturbed {J}acobi matrix sequences.
\newblock {\em Journal of Approximation Theory 144}, 1 (2007), 84--102.

\bibitem{golub2013matrix}
{\sc Golub, G., and Van~Loan, C.}
\newblock Matrix computations 4th edition the johns hopkins university press.
\newblock {\em Baltimore, MD\/} (2013).

\bibitem{book:25026}
{\sc Grenander, U.}
\newblock {\em Toeplitz Forms and Their Applications}, 2~ed.
\newblock Chelsea Pub Co, 1984.

\bibitem{Hack}
{\sc Hackbusch, W.}
\newblock {\em Multi-grid methods and applications}.
\newblock Springer Series in Computational Mathematics. Springer, 1985.

\bibitem{Hack2}
{\sc Hackbusch, W.}
\newblock {\em Iterative solution of large sparse systems of equations},
  vol.~95.
\newblock Springer, 1994.

\bibitem{hansen2006deblurring}
{\sc Hansen, P.~C., Nagy, J.~G., and O'leary, D.~P.}
\newblock {\em Deblurring images: matrices, spectra, and filtering}.
\newblock SIAM, 2006.

\bibitem{her14}
{\sc Hughes, T.~J., Evans, J.~A., and Reali, A.}
\newblock Finite element and {NURBS} approximations of eigenvalue,
  boundary-value, and initial-value problems.
\newblock {\em Computer Methods in Applied Mechanics and Engineering 272\/}
  (2014), 290--320.

\bibitem{hrs08}
{\sc Hughes, T.~J., Reali, A., and Sangalli, G.}
\newblock Duality and unified analysis of discrete approximations in structural
  dynamics and wave propagation: comparison of p-method finite elements with
  k-method {NURBS}.
\newblock {\em Computer {M}ethods in {A}pplied {M}echanics and {E}ngineering
  197}, 49-50 (2008), 4104--4124.

\bibitem{hwang2004cauchy}
{\sc Hwang, S.-G.}
\newblock Cauchy's interlace theorem for eigenvalues of {H}ermitian matrices.
\newblock {\em The American Mathematical Monthly 111}, 2 (2004), 157--159.

\bibitem{opiccolo}
{\sc Mazza, M., Ratnani, A., and Serra-Capizzano, S.}
\newblock Spectral analysis and spectral symbol for the 2d curl-curl
  (stabilized) operator with applications to the related iterative solutions.
\newblock {\em Mathematics of Computation 88}, 317 (2019), 1155--1188.

\bibitem{morozov}
{\sc Morozov, S., Serra-Capizzano, S., and Tyrtyshnikov, E.}
\newblock Computation of asymptotic spectral distributions for sequences of
  grid operators.
\newblock {\em Computational Mathematics and Mathematical Physics 60}, 11
  (2020), 1761--1777.

\bibitem{ng2004iterative}
{\sc Ng, M.~K.}
\newblock {\em Iterative methods for Toeplitz systems}.
\newblock Numerical Mathematics and Scie, 2004.

\bibitem{Q}
{\sc Quarteroni, A.}
\newblock {\em Numerical Models for Differential Problems}, vol.~16.
\newblock Springer, 2017.

\bibitem{P_k}
{\sc Rahla, R.~I., Serra-Capizzano, S., and Tablino-Possio, C.}
\newblock Spectral analysis of $\mathbb{P}_k$ {F}inite {E}lement matrices in
  the case of {F}riedrichs--{K}eller triangulations via {G}eneralized {L}ocally
  {T}oeplitz technology.
\newblock {\em Numerical Linear Algebra with Applications 27}, 4 (2020), e2302.

\bibitem{r06}
{\sc Reali, A.}
\newblock An isogeometric analysis approach for the study of structural
  vibrations.
\newblock {\em Journal of Earthquake Engineering 10}, spec01 (2006), 1--30.

\bibitem{ruge1987algebraic}
{\sc Ruge, J.~W., and St{\"u}ben, K.}
\newblock Algebraic multigrid.
\newblock In {\em Multigrid methods}. SIAM, 1987, pp.~73--130.

\bibitem{C2}
{\sc Russo, A., Serra-Capizzano, S., and Tablino-Possio, C.}
\newblock Quasi-optimal preconditioners for finite element approximations of
  diffusion dominated convection--diffusion equations on (nearly) equilateral
  triangle meshes.
\newblock {\em Numerical Linear Algebra with Applications 22}, 1 (2015),
  123--144.

\bibitem{saad}
{\sc Saad, Y.}
\newblock {\em Iterative methods for sparse linear systems}.
\newblock SIAM, 2003.

\bibitem{p-hp-book}
{\sc Schwab, C.}
\newblock {\em p- and hp-{F}inite {E}lement {M}ethods - Theory and Applications
  in solid and fluid mechanics}.
\newblock Oxford Univ. Press, 1998.

\bibitem{multi-iterative}
{\sc Serra-Capizzano, S.}
\newblock Multi-iterative methods.
\newblock {\em Computers \& {M}athematics with {A}pplications 26}, 4 (1993),
  65--87.

\bibitem{marko}
{\sc Serra-Capizzano, S.}
\newblock Asymptotic results on the spectra of block {T}oeplitz preconditioned
  matrices.
\newblock {\em SIAM {J}ournal on {M}atrix {A}nalysis and {A}pplications 20}, 1
  (1998), 31--44.

\bibitem{S-LAA-1998}
{\sc Serra-Capizzano, S.}
\newblock On the extreme eigenvalues of {H}ermitian (block) {T}oeplitz
  matrices.
\newblock {\em Linear {A}lgebra and its {A}pplications 270}, 1-3 (1998),
  109--129.

\bibitem{serra1999korovkin}
{\sc Serra-Capizzano, S.}
\newblock A {K}orovkin-type theory for finite {T}oeplitz operators via matrix
  algebras.
\newblock {\em Numerische Mathematik 82}, 1 (1999), 117--142.

\bibitem{serra1999superlinear}
{\sc Serra-Capizzano, S.}
\newblock Superlinear {{PCG}} methods for symmetric {T}oeplitz systems.
\newblock {\em Mathematics of Computation 68}, 226 (1999), 793--803.

\bibitem{glt-1}
{\sc Serra-Capizzano, S.}
\newblock Generalized {{L}}ocally {{T}}oeplitz sequences: spectral analysis and
  applications to discretized partial differential equations.
\newblock {\em Linear Algebra and its Applications 366\/} (2003), 371--402.

\bibitem{serra-imag}
{\sc Serra-Capizzano, S.}
\newblock A note on antireflective boundary conditions and fast deblurring
  models.
\newblock {\em SIAM Journal on Scientific Computing 25}, 4 (2004), 1307--1325.

\bibitem{glt-2}
{\sc Serra-Capizzano, S.}
\newblock The {GLT} class as a {{G}}eneralized {F}ourier analysis and
  applications.
\newblock {\em Linear Algebra and its Applications 419}, 1 (2006), 180--233.

\bibitem{C1}
{\sc Serra-Capizzano, S., and Tablino-Possio, C.}
\newblock Finite element matrix sequences: the case of rectangular domains.
\newblock {\em Numerical Algorithms 28}, 1 (2001), 309--327.

\bibitem{ST-SISC}
{\sc Serra-Capizzano, S., and Tablino-Possio, C.}
\newblock Multigrid methods for multilevel circulant matrices.
\newblock {\em SIAM {J}ournal on {S}cientific {C}omputing 26}, 1 (2004),
  55--85.

\bibitem{strang1986proposal}
{\sc Strang, G.}
\newblock A proposal for {T}oeplitz matrix calculations.
\newblock {\em Studies in Applied Mathematics 74}, 2 (1986), 171--176.

\bibitem{Tilli}
{\sc Tilli, P.}
\newblock A note on the spectral distribution of {T}oeplitz matrices.
\newblock {\em Linear and Multilinear Algebra 45}, 2-3 (1998), 147--159.

\bibitem{toeplitz1911theorie}
{\sc Toeplitz, O.}
\newblock Zur theorie der quadratischen und bilinearen formen von
  unendlichvielen ver{\"a}nderlichen.
\newblock {\em Mathematische Annalen 70}, 3 (1911), 351--376.

\bibitem{ty-1}
{\sc Tyrtyshnikov, E.}
\newblock A unifying approach to some old and new theorems on distribution and
  clustering.
\newblock {\em Linear {A}lgebra and its {A}pplications 232\/} (1996), 1--43.

\end{thebibliography}

\bibliographystyle{acm}
%\bibliographystyle{abbrv}
%\bibliographystyle{siambadal}
%\bibliographystyle{siam}
%\bibliographystyle{apalike}
%\bibliographystyle{sort}
%\chapter*{Conclusion}

\end{document}